\documentclass[12,reqno]{amsart}
\usepackage{amsfonts}
\usepackage{amssymb}
\usepackage{a4wide}
\usepackage{pgf,pgfarrows,pgfautomata,pgfheaps,pgfnodes,pgfshade}
\usepackage{url}

\usepackage[numbers]{natbib}

\usepackage{hyperref}
\numberwithin{equation}{section}

\newcommand{\diff}{\,\mathrm{d}}
\newcommand{\diffns}{\mathrm{d}}

\newtheorem{defi}{Definition}[section]
\newtheorem{thm}[defi]{Theorem}
\newtheorem{lemm}[defi]{Lemma}
\newtheorem{remark}[defi]{Remark}
\newtheorem{cor}[defi]{Corollary}
\newtheorem{assum}[defi]{Assumption}
\newtheorem{prop}[defi]{Proposition}

\newcommand{\supp}{\operatorname{supp}}

\DeclareMathOperator{\sgn}{sgn}
\DeclareMathOperator{\essup}{essup}

\DeclareMathOperator{\Tr}{Tr}

\begin{document}

		\title[Density Analysis for coupled FBSDEs with non-Lipschitz drifts and Applications]{Density Analysis for coupled forward-backward SDEs with non-Lipschitz drifts and Applications}
		\author{Rhoss Likibi Pellat}
		\address{Department of Mathematics,University of Ghana, P.O.Box LG 62 Legon Accra, Ghana.African Institute for Mathematical Sciences, Ghana}
		\email{rhoss@aims.edu.gh}
		\author{Olivier Menoukeu Pamen}
		\address{African Institute for Mathematical Sciences, Ghana}
		\address{Institute for Financial and Actuarial Mathematics, Department of Mathematical Sciences, University of Liverpool, L69 7ZL, United Kingdom}
		\email{menoukeu@liverpool.ac.uk}
		\thanks{R. Likibi Pellat was supported by DAAD under the programme PhD-study at the African Institute for Mathematical Sciences, Ghana. O. Menoukeu Pamen acknowledges the funding provided by the Alexander von Humboldt Foundation, under the program financed by the German Federal Ministry of Education and Research entitled German Research Chair No 01DG15010.}
		
		
		\date{}
		
		
		
		\maketitle
		\begin{abstract} We explore the existence of a continuous marginal law with respect to the Lebesgue measure for each component $(X,Y,Z)$ of the solution to coupled quadratic forward-backward stochastic differential equations (QFBSDEs) {for which the drift coefficient of the forward component is either bounded and measurable or H\"older continuous. Our approach relies on a combination of the existence of a weak {\it decoupling field} (see \cite{Delarue2}), the integration with respect to space time local time (see \cite{Ein2006}), the analysis of the backward Kolmogorov equation associated to the forward component along with an It\^o-Tanaka trick (see \cite{FlanGubiPrio10})}. The framework of this paper is beyond all existing papers on density analysis for Markovian BSDEs and constitutes a major refinement of the existing results. We also derive a comonotonicity theorem for the control variable in this frame and thus extending the works \cite{ChenKulWei05}, \cite{DosDos13}.
			
		{As applications of our results, we first analyse the regularity of densities of solution to coupled FBSDEs. In the second example, we consider a regime switching term structure interest rate models (see for e.g., \cite{ChenMaYin17}) for which the corresponding FBSDE has discontinuous drift. Our results enables us to: firstly study classical and Malliavin differentiability of the solutions for such models, secondly the existence of density of such solutions. Lastly we consider a pricing and hedging problem of contingent claims on  non-tradable underlying, when the dynamic of the latter is given by a regime switching SDE (i.e., the drift coefficient is allowed to be discontinuous). We obtain a representation of the derivative hedge as the weak derivative of the indifference price function, thus extending the result in \cite{ArImDR10}.}
		\end{abstract}

		\section{Introduction}
		The primary goal of this paper is the analysis of the densities of the solutions for coupled quadratic forward-backward stochastic differential equations (QFBSDEs) with {singular drift coefficients (see Assumption \ref{assum1} and \ref{assum53}). The key points of our method are: the regularity of the {\it weak decoupling field} (solution in the weak sense of the associated quasi-linear PDE) the space time integration with respect to the local time of the Brownian motion (case of discontinuous drifts) and the analysis of the backward Kolmogorov equation associated to the forward SDE (case of H\"older continuous drift)}. From what we know, the first result on the existence of the density of solutions to Markovian BSDEs can be traced back to Antonelli \& Kohatsu-Higa (\cite{F.Antonelli}). On the contrary, the analysis of densities of solutions to SDEs is well understood (see example \cite{Kusuoka10, Nua06} for Lipschitz continuous coefficients; \cite{BanosKruhner17, Romito18} for rough coefficients and \cite{BanosNilssen16, MenozziPesceZhang21,OliveraTudor19} for unbounded and measurable drifts). The authors in \cite{F.Antonelli} proved that, the smoothness properties of the solution to the forward equation transferred to the existence of an absolutely continuous law with respect to the Lebesgue measure  of the first component $Y$ of the solution of the backward equation. 
		
		The study of the marginal densities for solutions to BSDEs is becoming attractive due to its wide range of applications in financial mathematics, as for instance pricing problems (see \cite{Mastrolia17, Mastrolia}), or more recently in biology for the stochastic gene expression model  (see  \cite{Mastrolia17, ChertoShama22}). In particular, explicitly describing the tails of the density of the control variable, can improve the rate of convergence for the numerical approximation of quadratic BSDEs (see \cite{Mastrolia}). Notice that, the Gaussian-type estimates of the law of the control variable $Z$ was first addressed in \cite{AbouraBourguin} for linear drivers. These results were extended in the nonlinear quadratic framework in \cite{Mastrolia}. More precisely, the authors in \cite{Mastrolia} consider the following BSDE
		\begin{align}\label{eq 1}
			Y_t = \phi(X_T) + \int_t^T f(s,X_s,Y_s,Z_s)\mathrm{d}s -\int_t^T Z_s\mathrm{d}W_s,
		\end{align}
		when the terminal value and the driver are both functional of a diffusion process $X$ which satisfy: 
		\begin{align}\label{eq 2}
			X_t = x + \int_0^t b(s,X_s)\mathrm{d}s +\int_0^t \sigma(s,X_s) \mathrm{d}W_s
		\end{align}
		{with $b, \sigma, \phi, f$  deterministic functions satisfying their standing Assumptions (X) (see \cite[page 2823]{Mastrolia}). Assuming that $f$ is of quadratic growth (in the sense of \cite{Kobylansky}) and twice continuously differentiable, the author\textcolor{blue}{s} claim that the solution $(X,Y,Z)$ of the  FBSDE \eqref{eq 1}-\eqref{eq 2} is twice Malliavin differentiable. They then use this fact to analyse the density of the control variable $Z$. We do not however understand why  the solution $X$ of the forward \eqref{eq 2} belongs to $\mathbb{D}^{2,p}$ for $p\geq 1$, under only Lipschitz assumption of its coefficients $b$ and $\sigma$. We believe that, some assumptions that guarantee the second order Malliavin differentiability of the system \eqref{eq 1}--\eqref{eq 2} are provided in \cite[Theorem 4.1]{ImkDos}.}

		The first result {on the analysis} for coupled FBSDEs was recently obtained in \cite{Olivera}. As opposed to the traditional approach, the authors exploited the regularity properties of the solutions to the associated  quasi-linear parabolic PDE to derive the existence, Gaussian-type bounds and the tail estimates of the law of each component $(X,Y,Z)$ solution to a scalar FBSDE, on the entire real line.{ We do, however, point out an inaccuracy in their statement Theorem 3.1 about the density estimates for the law of the forward component $X_t$. An inspection of its proof suggests that additional assumptions are required to obtain the sought estimates. Indeed, the Lamperti transformation utilized by the authors (see their subsection 2.6) to obtain the latter requires more regularity on the coefficients. For example, the drift function must be at least Lipschitz continuous in all variables, whereas the diffusion coefficient must be non-degenerate and admit a finite derivative in time and a bounded second order derivative in space.}
		Let us also mention that, the analysis of densities for BSDEs driven by Gaussian processes (fractional Brownian motion in particular) was addressed in \cite{FanWu21} whereas the multidimensional setting was treated in \cite{ChertoShama22}.


		While the globally Lipschitz condition is a necessary requirement to show the existence of the density of $Y$ in the above works, we do not need this assumption in the present paper. We use the decoupling field along with the {regularisation effect of the Brownian motion} to show that the solution to the forward SDE has a density, {despite the roughness of the drift coefficient(and even the discontinuity of the drift)}. The decoupling field once more enables us to transfer the smoothing properties from the SDEs to the BSDEs.

		

		{ We consider the following coupled forward-backward system
			\begin{align}\label{eq 3}
				\begin{cases}
					X_t = x + \int_0^t b(s,X_s,Y_s,Z_s)\mathrm{d}s + \int_0^t \sigma(s,X_s) \mathrm{d}W_s,\\
					Y_t = \phi(X_T) + \int_t^T f(s,X_s,Y_s,Z_s)\mathrm{d}s -\int_t^T Z_s\mathrm{d}W_s,
				\end{cases}
			\end{align}
			{assuming that the coefficients are H\"{o}lder continuous and that the diffusion $\sigma$ is regular enough (see Assumption \ref{assum53}, therein), we first improve the result on the well posedness of forward-backward SDEs established in Section \ref{solvability} by proving that the system \eqref{eq 3} admits a unique solution on any arbitrarily prescribed Brownian set-up even when the forward component lies in a multiple dimension frame. Moreover, the solution process $(X,Y,Z)$ has a continuous version which is differentiable with respect to the initial state of the forward component and also Malliavin regular. As a matter of fact, the latter leads to the study of existence and regularity of densities of the marginal laws of each component solution to the system \eqref{eq 3}.} 
			
			We will briefly describe our strategy below. Observe that the circular dependence of solutions of both equations and the roughness of the coefficients involved, pose a significant challenge for solving  FBSDE \eqref{eq 3}. {Via the weak decoupling field argument obtained in Section \ref{solvability} , we deduce that the forward equation in \eqref{eq 3} reduces to the following SDE:
				\begin{align}\label{eq 1.4}
					X_t = x + \int_{0}^{t} \tilde{b}(s,X_s)\mathrm{d}s + \int_0^t \sigma(s,X_s) \mathrm{d}W_s,
				\end{align}
				$\tilde{b}(t,x):=b(t,x,v(t,x),\nabla_xv(t,x))$ and $v$ is solution to the associated PDE \eqref{eqmainP}. Unfortunately, the transformed drift $\tilde b$ is only $C_b^{\theta}$ with $\theta \in (0,1)$, thus non differentiable and does not satisfy the standard requirement of most of the scholars. Hence, locating certain results in the literature that deal with the differentiability (Malliavin or classical sense) of $X$ solution to \eqref{eq 1.4} as well as the existence and smoothness of its density, are not straightforward. To overcome this issue of roughness of the drift, we use a version of the It\^o-Tanaka trick (or Zvonkin transform) developed in \cite{FlanGubiPrio10,FlanGubiPrio} to establish a one to one correspondence between the solution to \eqref{eq 1.4} and to the solution to the following SDE
				\begin{align}
					\tilde X_t = \zeta + \displaystyle \int_{0}^{t} \tilde{b}_1(s,\tilde X_s)\mathrm{d}s + \int_{0}^{t} \tilde{\sigma}_1(s,\tilde X_s)\mathrm{d}W_s.
				\end{align}
				Oberve that the above equation has better regular coefficients and therefore the solution $	\tilde X_t$  admits an absolute continuous density with respect to the Lebesgue measure. The latter property is then transferred to the solution $X$ for the SDE \eqref{eq 1.4} via the above correspondence between both solutions. Therefore the density property can be at last transfer to the $(Y,Z)$ solution to the backward SDEs in \eqref{eq 3}, thanks to the non-linear Feynman-Kac relation \eqref{FK}.  Although this technique was used in \cite{OliveraTudor19} and \cite{Romito18} to establish the existence of density of solutions to SDEs when the drift coefficient is beyond the Cauchy-Lipschitz framework, as far as we know it has not yet been implemented in the context of fully coupled FBSDEs \eqref{eq 3} (or even for decoupled FBSDEs) with rough coefficients.}
			

		Below is a summary of this paper's main contributions:
		
		1.  {We extend the existence and uniqueness results in \cite{RhOlivOuk} (and \cite{Delarue2}) to the case of more general  quadratic drivers. In particular, in one dimensional case and when $\sigma$ is $\alpha\text{-H\"older}$ continuous with $\alpha \geq 1/2$, we show the existence of a unique strong  solution of the QFBSDEs (see Theorem \ref{th3.3}).}
		
		2. {We derive comonotonicity property  of coupled QFBSDE when  the drift is merely bounded and measurable in the forward variable. More  precisely, we provide conditions that guarantee the non-negativity of the control process Z component solution to the coupled FBSDE \eqref{eqmainF} with discontinuous drift, thus extending the results derived in \cite{ChenKulWei05, DosDos13}.}
		
		3. {We show that the solution $(X,Y)$ of the system \eqref{eq 3}, has an absolutely continuous distribution with respect to the Lebesgue measure. 
			We also  provide Gaussian type bounds and tail estimates that are satisfied by probability  densities of $X$ and $Y$ respectively. Finally, we prove that the control variable $Z$ has an absolute continuous law with respect to the Lebesgue measure on $\mathbb{R}$. The second order Malliavin differentiability result (Theorem \ref{secondMal}, therein) which is important in this analysis, is novel for coupled FBSDEs with quadratic drivers  and Lipschitz continuous drifts.}
		
		4. {We refine the results obtained in the third point by assuming that the drift coefficient is discontinuous and the diffusion is identity. More explicitly, we obtain a density result for $X$ and $Y$ when the drift $b$ and the driver $f$ are only bounded and measurable in space and time.}

		{We consider several applications of the results described above. First, we prove that the density of each components solution $(X,Y,Z)$ to \eqref{eq 3} is H\"older-continuous, when the diffusive coefficient is identity and the drift is weakly differentiable in $x$, Lipschitz continuous in $y$ and $z$. To our knowledge, such a smoothing result on the density of solutions to coupled FBSDEs has not been derived before. As such, our result constitutes a major improvement of those in \cite{AbouraBourguin,F.Antonelli, Mastrolia, Olivera}. Moreover, an explicit representation of the density of the control $Z$ (resp. $X$ and $Y$) in the sense of \cite{Dung20} is also obtained, that is, if $\rho_{Z_t}$ denotes the density of $Z_t$, then for any point $x_0$ in the interior of $\supp (\rho_{Z_t})$ 
			\[ \rho_{Z_t}(x) = \rho_{Z_t}(x_0) \exp\Big( -\int_{x_0}^x \mathbb{E}\Big[ \delta\Big(\int_0^T D_sZ_t \mathrm{d}s\Big)^{-1}\big|Z_t =z\Big]\mathrm{d}z\Big), \, x \in \supp (\rho_{Z_t}).\] }
		
		{Second, we illustrate that our results on densities of coupled forward-backward SDEs with drift that displays discontinuity in its forward component $x$, remain valid even if the discontinuity of the drift now occurs in its backward component $y$. This framework of FBSDEs was explored in \cite{ChenMaYin17}, and is motivated by a practical challenge in regime-switching term structure interest models. More precisely, we consider a regime switching term structure interest rate model in which the mean reverting term depends on an exogenous factor and switches according to the level of the factor. The dynamics of the factor is given by a BSDE and this leads to a system FBSDE with discontinuity in the drift of the forward. Existence and uniqueness results of such models were studied in \cite{ChenMaYin17}. Our results enables us to: firstly study both classical and Malliavin differentiability of the solution for such models, secondly the existence of density of such solutions.}
		
		{Finally, we investigate a classical pricing and hedging problem of contingent claims on  non-tradable underlying as discussed in \cite{ArImDR10,HuImkMuller}. The probabilistic method --as opposed to the analytic one in \cite{AnkImkPo,Davis06,Hender02}, and references therein-- describes the price and the optimal hedging in terms of forward-backward SDEs with quadratic drivers. In these papers, the forward component, characterising the non-tradable underlying, is provided by a diffusion process with smooth coefficients. In addition, it is assumed that the payoff of the contingent claim is bounded, Lipschitz continuous and twice continuously differentiable with Lipschitz continuous derivatives (see for e.g., \cite{ArImDR10,HuImkMuller}). These conditions ensure the classical and Malliavin differentiability of the solution to the generated quadratic FBSDEs. We consider an extended Black and Scholes model where the non-tradable stock pays a dividend yield that switches to a higher level when the stock value passes a certain threshold. Such assumption leads to a discontinuity in the drift coefficient in the dynamics of the stock.  In addition we can also consider contingent claims with only bounded and Lipschitz continuous payoffs. Thus the results obtained in \cite{ArImDR10} cannot be applied in this situation. Under these modest assumptions, we demonstrate that the price function is weakly differentiable with respect to the starting value of the non-tradable asset, and we establish an explicit representation of the derivative hedge in terms of the price gradient. Furthermore, we prove that the price process is also Malliavin differentiable and thus improving existing results in this direction.}

		{The rest of the paper is outlined as follows:} some notations and a non-exhaustive presentation on the theory of Malliavin calculus and density analysis are given in Section \ref{preli}. The Section \ref{solvability} is devoted to the solvability of coupled FBSDE \eqref{eqmainF} via its corresponding quasi-linear parabolic PDE \eqref{eqmainP}. In Section  \ref{comonotonicity}, we  use the results obtained in Section \ref{solvability} to derive a comonotonicity theorem for the system \eqref{eqmainF}. Section \ref{section5} is the main core of this paper, we provide conditions under which the law of solutions to coupled quadratic FBSDEs with H\"older continuous drift and with bounded and measurable drift admit respectively an absolute continuous density. Finally,  Section \ref{section7} deals with some applications on the regularity of densities for solution to coupled FBSDEs and the study of a classical pricing and hedging problems .

		.

		\section{Preliminaries}\label{preli}
		In this section, we introduce the  framework which will serve throughout this paper.
		We work on a filtered probability space $(\Omega,\mathfrak{F},\{\mathfrak{F}_t\}_{t\geqslant 0},\mathbb{P})$, on which a $d\text{-dimensional}$ Brownian motion $\{W_t\}_{t\geq 0}$ is defined and $\mathbb{F} =\{\mathfrak{F}_t\}_{t\geqslant 0}$ is the standard filtration generated by $\{W_t\}_{t\geq 0}$  augmented by all $\mathbb{P}\text{-null}$ sets of $\mathfrak{F}.$ The expectations will be denoted by $\mathbb{E}$ and $\mathbb{E}^{\mathbb{Q}}$ respectively under $\mathbb{P}$ and under any other probability measure $\mathbb{Q}$. The maturity time $T \in (0,\infty)$ is fixed  and $d \in \mathbb{N},$ $p\in [2,\infty)$. 
		\subsection{Spaces of stochastic processes and spaces of functions} {Below we define some known space of stochastic processes and functions.}
		
		\begin{itemize}
			\item $L^p$ denotes the space of $\mathfrak{F}_T\text{-adapted}$ random variables  $X$ such that $\Vert X \Vert^p_{\tiny L^p }:= \mathbb{E}|X|^p < \infty$
			\item $L^{\infty}$ denotes the space of bounded random variables with norm $\Vert X \Vert_{\tiny L^{\infty} }:= \essup_{\omega\in \Omega}|X(\omega)|$.
			
			\item$\mathcal{S}^p(\mathbb{R}^d)$ is the space of all adapted continuous $\mathbb{R}^d$-valued processes $X$ such that $\Vert X \Vert^p_{\tiny \mathcal{S}^p(\mathbb{R}^d) }:= \mathbb{E}\sup_{t\in [0,T]} |X_t|^p < \infty$
			\item $\mathcal{S}^{\infty}(\mathbb{R}^d)$ the space of continuous $\{ \mathfrak{F}_t \}_{0\leq t\leq T} $-adapted processes $Y:\Omega\times[0,T]\rightarrow \mathbb{R}^d$ such that $\Vert Y \Vert_{\tiny \infty }:= \essup_{\omega\in \Omega} \sup_{t\in [0,T]} |Y_t(\omega)| < \infty.$
			\item $\mathcal{H}^p(\mathbb{R}^d)$ stands for the space of all predictable $\mathbb{R}^d$-valued processes $Z$  such that $\Vert Z \Vert^p_{\tiny \mathcal{H}^p(\mathbb{R}^d) }:= \mathbb{E}(\int_0^T |Z_s|^2\mathrm{d}s)^{p/2}  < \infty$. 
			\item  $C^{1,2}_b([0,T]\times\mathbb{R}^d,\mathbb{R})$ is the usual space of maps $v:[0,T]\times\mathbb{R}^d\rightarrow \mathbb{R}$ that are continuously differentiable in the first variable and twice continuously differentiable w.r.t to its second variable with bounded derivatives.
			\item $W^{1,2,d+1}_{\text{loc}}([0,T[\times\mathbb{R}^d,\mathbb{R})$ denotes the Sobolev space of classes of functions $v: [0,T[\times\mathbb{R}^d \rightarrow \mathbb{R},$ such that 
			$|v|,|\partial_t v|, |\nabla_x v|, |\nabla_{x,x}^2 v| \in L^{d+1}_{\text{loc}}([0,T[\times\mathbb{R}^d,\mathbb{R}).$
			\item  Fix $\beta \in (0,1]$, $C^{0,\beta}$ is the space of H\"older continuous functions $\phi : [0,T]\times \mathbb{R}^d \rightarrow \mathbb{R}$ endowed with the norm
			\[  \|\phi\|_{C^{0,\beta}} =  \sup_{x\neq y\in \mathbb{R}^d} \frac{|\phi(x) -\phi(y)|}{|x-y|^\beta} \]
		\end{itemize}
		\subsection{The BMO space}
		The space BMO($\mathbb{P}$) is the Banach space of continuous and square integrable martingales $M$ starting at $0$ under the probability measure $\mathbb{P}$ endowed with the norm
		$$ \Vert M\Vert_{\text{BMO}(\mathbb{P})} = \sup_{\tau \in [0,T]} \Vert \mathbb{E} [\langle M \rangle_T  - \langle M \rangle_{\tau}]/\mathfrak{F}_{\tau} \Vert_{\infty}^{1/2},$$ 
		where the supremum is taken over all stopping times $\tau \in [0,T].$ While, $\mathcal{H}_{\text{BMO}}$ stands for the space of $\mathbb{R}^d\text{- valued}$ $\mathcal{H}^p\text{-integrable} $ predictable processes $(Z_t)_{t\in [0,T]}$ such that $Z*B := \int_0 Z_s\mathrm{d}B_s$ belongs to  $\text{BMO}(\mathbb{P}).$ We define $\Vert Z \Vert_{\mathcal{H}_{\text{BMO}}}:= \Vert \int Z\mathrm{d}B \Vert_{\text{BMO}(\mathbb{P})}$.
	For every martingale $M \in \text{BMO}(\mathbb{P})$, the exponential Dol\'eans-Dade process $(\mathcal{E}(M)_t)_{0\leq t\leq T}$ is defined by
	\[ \mathcal{E}(M)_t := \exp \Big(M_t -\frac{1}{2}\langle M \rangle_t\Big)  , \]
	such that $\mathbb{E}(\mathcal{E}(M)_T) = 1$. Moreover, the measure $\mathbb{Q}$ with density 
	\[  \frac{\mathrm{d}\mathbb{Q}}{\mathrm{d}\mathbb{P}}\Big|_{\mathfrak{F}_t}:= \mathcal{E}(M)_t,  \]
	defines an equivalent probability measure to $\mathbb{P}$. The continuous martingale $M \in \text{BMO}(\mathbb{P})$ satisfies the following properties: (We refer the reader to \cite{Kazamaki} for more details on the subject.)\leavevmode
	
	%
	%
	%

	\subsection{Basic notions on Malliavin calculus and density analysis}
	Here, we briefly introduce some elements in Malliavin calculus and the analysis of densities. We refer any interested reader to \cite{Nua06}, for clear exposition on the subject.
	
	Let $\mathcal{S}$ be the space of random variables $\xi$ of the form 
	\[ \xi = F\Big( (\int_{0}^{T} h_s^{1,i}\mathrm{d}W_s^1)_{1\leq i\leq n},\cdots, (\int_{0}^{T} h_s^{d,i}\mathrm{d}W_s^d)_{1\leq i\leq n}   \Big),  \]
	where $F\in C_b^{\infty}(\mathbb{R}^{n\times d}), h^1,\ldots,h^n \in L^2([0,T];\mathbb{R}^d)$ and $n\in \mathbb{N}.$ For simplicity, we assume that all $h^j$ are written as row vectors. For $\xi \in \mathcal{S},$ let $D$ be the operator defined by $D=(D^1,\cdots,D^d):\mathcal{S}\rightarrow L^2(\Omega\times [0,T])^d$ by 
	\[ 
	D^i_{\theta}\xi := \sum_{j=1}^{n}\frac{\partial F}{\partial x_{i,j}} \Big( \int_{0}^{T} h_t^{1}\mathrm{d}W_t,\cdots,\int_{0}^{T} h_t^{n}\mathrm{d}W_t   \Big)h_{\theta}^{i,j}, 0\leq \theta \leq T, 1\leq i \leq d, 
	\]
	and for $k\in \mathbb{N}$ and $\theta=(\theta_1,\cdots,\theta_k)\in [0,T]^k$ its $k\text{-fold}$ iteration is given by
	\[ D^{(k)} = D^{i_1}\cdots  D^{i_k}_{1\leq i_1,\cdots,i_k\leq d}.  \]
	For $k\in \mathbb{N}$ and $p\geq 1,$ let $\mathbb{D}^{k,p}$ be the closure of $\mathcal{S}$ with respect to the norm 
	\[ \| \xi \|^p_{k,p} = \| \xi \|^p_{L^p} + \sum_{i=1}^{k} \| |D^{(i)}\xi|\|^p_{(\mathcal{H}^p)^i}  .\]
	The operator $D^{(k)}$ is a closed linear operator on the space $\mathbb{D}^{k,p}.$ Observe that if $\xi \in \mathbb{D}^{1,2} $ is $\mathfrak{F}_t\text{-measurable}$ then $D_{\theta}\xi = 0$ for $\theta\in (t,T].$ Further denote $\mathbb{D}^{k,\infty} =\bigcap_{p>1}\mathbb{D}^{k,p}.$
	
	{The following result is the chain rule for Malliavin derivative (see  for example \cite[Proposition 1.2.4]{Nua06}). 
		\begin{prop}
			Let $\varphi: \mathbb{R}^d\rightarrow\mathbb{R}$ be a globally Lipschitz continuous  function. Suppose $\xi = (\xi^1,\cdots,\xi^d)$ is a random vector whose components belong to the space $\mathbb{D}^{1,2}$. Then $\varphi(\xi)\in \mathbb{D}^{1,2} $ and there exists a random vector $\varpi= (\varpi^1,\cdots,\varpi^d)$ bounded by the Lipschitz constant of $\varphi$ such that $D\varphi(\xi) = \sum_{i=1}^d \varpi^iD\xi^i$. In particular, if the gradient of $\varphi$ denoted by $\varphi'$ exists, then $\varpi = \varphi'(\xi).$ 
	\end{prop}}
	For a stochastic process $\mathcal{X} \in \text{Dom}(\delta)$ (not necessarily adapted to $\{\mathfrak{F}_t\}_{t\in [0,T]}$) we define the {Skorohod} integral of $\mathcal{X}$ by 
	\begin{align}\label{Skoro}
		\delta(\mathcal{X}) = \int_0^T \mathcal{X}_t\delta W_t,
	\end{align}
	which is an anticipative stochastic integral. It turns out that all $\{\mathfrak{F}_t\}_{t\in [0,T]}\text{-adapted}$ processes in $L^2(\Omega\times[0,T])$ are in the domain of $\delta$ and can be written as \eqref{Skoro}. Thus, the Skorohod integral can be viewed as an extension of the It\^{o} integral to non-adapted integrand. The following equation known as the duality formula, gives the dual relation between the Malliavin derivative and the Skorohod integral:
	\[  \mathbb{E}\Big[\xi\int_0^T \mathcal{X}_t\delta W_t\Big]= \mathbb{E}\Big[ \int_0^T \mathcal{X}_t D_t\xi \mathrm{d}t  \Big]  \]
	for any $\xi \in \mathbb{D}^{1,2}$ and $\mathcal{X} \in \text{Dom}(\delta)$. 
	
	For $F = (F^1,\cdot,F^d) \in (\mathbb{D}^{1,2})^d,$ we define the Malliavin covariance  $d\times d\text{-matrix}$ $\Gamma_{F}$ by 
	\[ 
	\Gamma_{F}^{ij}:= \langle DF^i,DF^j \rangle_{L^2[0,T]} . 
	\]
	The random vector $F$ is said to be non-degenerate if for all $p \geq 1$
	\[  
	\mathbb{E}\Big[ \Big(\det \Gamma_{F}\Big)^{-p}\Big] < \infty . 
	\]
	
	For $k\in \mathbb{N}, p\geq 1,$ denote by $\mathbb{L}_{k,p}(\mathbb{R}^m)$ the set of $\mathbb{R}^d\text{-valued}$ progressively measurable processes $\zeta=(\zeta^1,\cdots,\zeta^m)$ on $[0,T]\times \Omega$ such that 
	\begin{itemize}
		\item[(i)] For Lebesgue a.a. $t \in [0,T], \zeta(t,\cdot)\in (\mathbb{D}^{k,p})^m; $
		\item[(ii)] $(t,\omega)\rightarrow D_{s}^k \zeta(t,\omega) \in (L^2([0,T]^{1+k}))^{d\times m}$ admits a progressively measurable version;
		\item[(iii)] $\|\zeta \|_{k,p}^p = \| |\zeta| \|^p_{\mathcal{H}^p} + \sum_{i=1}^{k} \| |D^{(i)}\zeta|\|^p_{(\mathcal{H}^p)^{1+i}} < \infty.$
	\end{itemize}
	For example if a process $\zeta\in \mathbb{L}_{2,2}(\mathbb{R})$, we have 
	\begin{align*}
		\| \zeta \|_{\mathbb{L}_{1,2}}^2 &= \mathbb{E}\Big[ \int_{0}^{T}|\zeta_t|^2\mathrm{d}t + \int_{0}^{T}\int_{0}^{T}|D_{s}\zeta_t|^2\mathrm{d}s\mathrm{d}t  \Big],\\
		\| \zeta \|_{\mathbb{L}_{2,2}}^2 &= \| \zeta \|_{\mathbb{L}_{1,2}}^2 + \mathbb{E}\Big[ \int_{0}^{T}\int_{0}^{T}\int_{0}^{T}|D_{s'}D_{s}\zeta_t|^2\mathrm{d}s'\mathrm{d}s\mathrm{d}t  \Big].
	\end{align*}
	Moreover, the following follows from Jensen's inequality 
	\[  \mathbb{E}\Big[\Big(\int_{0}^{T}\int_{0}^{T} |D_s\zeta_t|^2\mathrm{d}s\mathrm{d}t\Big)^{\frac{p}{2}}\Big] \leq T^{\frac{p}{2}-1}\int_{0}^{T} \|D_s\zeta_t\|_{\mathcal{H}^p}^p\mathrm{d}s, \forall p\geq 2, \]
	and provide an alternative way to control the term of the left-hand side.
	
	Below, we recall the criterion for absolute continuity of the law of a random variable $F$ with respect to the Lebesgue measure.
	
	\begin{thm}[Bouleau-Hirsch, Theorem 2.1.2 in \cite{Nua06}]\label{criterion} Let $F$ be in $\mathbb{D}^{1,2}.$ Assume that $\| DF \|_{\mathcal{H}} > 0, \,\mathbb{P}\text{-a.s.}.$ Then $F$ has a probability distribution which is absolutely continuous with respect to the Lebesgue measure on $\mathbb{R}.$ 
	\end{thm}
	In one dimension, it is possible to give a representation of the density in the sense of Nourdin and Viens. If $\rho_F$ denotes the density of a random variable $F \in \mathbb{D}^{1,2}$ we set 
	\begin{equation}
		h_F(x):= \int_{0}^{\infty} e^{-\theta}\mathbb{E}\left[ \mathbb{E}^{*}[ \langle DF,\tilde{DF}^{\theta} \rangle_{\mathcal{H}}]|F-\mathbb{E}(F)=x \right]\mathrm{d}\theta ,
	\end{equation}
	where for any random variable $X$, we define $\tilde{X}^{\theta}:= X(e^{-\theta}W + \sqrt{1-e^{-2\theta}}W^{*})$ with $W^{*}$ an independent copy of $W$ given on a probability space $(\Omega^{*},\mathfrak{F}^{*},\mathbb{P}^{*}),$ and $\mathbb{E}^{*}$ denotes the expectation under $\mathbb{P}^{*}.$ We have
	\begin{cor}[Nourdin-Viens' formula] Under assumptions in Theorem \ref{criterion}, the support  of $\rho_F,$ denoted by $\supp(\rho_F)$ is a closed interval of $\mathbb{R}$ and for all $x \in \supp(\rho_F)$
		\[ \rho_F(x):= \frac{\mathbb{E}(|F-\mathbb{E}(F)|)}{2h_F(x-\mathbb{E}(F))}\exp\Big( -\int_{0}^{x-\mathbb{E}(F)}\frac{\theta\mathrm{d}\theta}{h_F(\theta)}    \Big).   \]	
	\end{cor}
	We will end this section by recalling the following useful result that can be found respectively in  \cite{ DungPrivaultTorrisi15}. 
	\begin{thm}[Theorem 2.4 in \cite{DungPrivaultTorrisi15}]
		Let $F \in \mathbb{D}^{1,2}$ be a random variable such that:
		\begin{align}\label{4.8}
			0 < l \leq \int_0^{\infty} D_sF \mathbb{E}\left[D_s F|\mathfrak{F}_s\right]\mathrm{d}s \leq L \quad \text{ a.s.},
		\end{align}
		where, $l$ and $L$ are positive constants. Then, $F$ possesses a density $\rho_{F}$ with respect to the Lebesgue measure. Moreover, for almost all $x \in \mathbb{R}$, the density satisfies
		\begin{align}\label{4.9}
			\frac{\mathbb{E}|F-\mathbb{E}[F]|}{2L}\exp\Big( -\frac{(x-\mathbb{E}[F])^2}{2l}\Big) \leq \rho_{F}(x) \leq 	\frac{\mathbb{E}|F-\mathbb{E}[F]|}{2l}\exp\Big(- \frac{(x-\mathbb{E}[F])^2}{2L}  \Big).
		\end{align}
		Furthermore for all $x>0$, the tail probabilities satisfy
		\begin{align}\label{tail}
			\mathbb{P}(F\geq x)\leq \exp\Big( -\frac{(x-\mathbb{E}[F])^2}{2L} \Big)\text{ and } 	\mathbb{P}(F\leq -x)\leq \exp\Big( -\frac{(x+\mathbb{E}[F])^2}{2L}\Big)
		\end{align}
	\end{thm}	
	\section{{Existence and Uniqueness of solution to quadratic FBSDEs}}\label{solvability}
	Let $(\Omega,\mathfrak{F},\mathbb{F}=\{\mathfrak{F}_s\}_{s\geqslant 0},\mathbb{P}, W_s),$  be a stochastic basis, satisfying the usual conditions. {In this section,we wish to show that the following FBSDE admits a unique solution.}
	\begin{equation}\label{eqmainF}
		\begin{cases}
			X_s^{t,x} = x + \displaystyle \int_{t}^{s} b(r,X_r^{t,x},Y_r^{t,x},Z_r^{t,x})\mathrm{d}r +  \int_{t}^{s}\sigma(r,X_r^{t,x},Y_r^{t,x}) \mathrm{d}W_r,\\
			Y_s^{t,x} = \phi(X_T^{t,x})+ \displaystyle \int_{s}^{T} f(r,X_r^{t,x},Y_r^{t,x},Z_r^{t,x})\mathrm{d}s - \int_{s}^{T} Z_r^{t,x} \mathrm{d}M_r^{X},
		\end{cases}
	\end{equation}
	where  $b,\sigma,\phi$ and $f$ are deterministic functions of appropriate dimensions taking values in appropriate spaces that will be made precise below, $\mathrm{d}M_r^{X}:= \sigma(r,X_r^{t,x},Y_r^{t,x})^{{\bf T}} \mathrm{d}W_r$ stands for the martingale part of the semi-martingale $X^{t,x}$ and the superscript $(t,x)$ will refer to its initial condition.  It is also known that, such an FBSDE provides a probabilistic interpretation for the following terminal value problem of quasi-linear partial differential equation (PDE, for short) (see for example \cite{Ma})
	\begin{equation}\label{eqmainP}
		\frac{\partial v}{\partial t}(t,x)+ \mathcal{L}v(t,x) + f(t,x,v(t,x),\nabla_x v(t,x))=0,\quad v(T,x)= \phi(x),
	\end{equation} 
	where the differential operator $\mathcal{L}$ is given by
	$$\mathcal{L}v(t,x):= \frac{1}{2}\sum_{i,j}^{d}(\sigma\sigma^{{\bf T}})_{ij}(t,x,v(t,x))\frac{\partial^2 v}{\partial x_i \partial x_j}(t,x) + \sum_{i}^{d} b_{i}(t,x,v(t,x),\nabla_x v(t,x))\frac{\partial v}{\partial x_i}(t,x).$$ 
	Moreover, both solutions for the FBSDE \eqref{eqmainF} and the PDE \eqref{eqmainP}, respectively denoted by $(X^{t,x},Y^{t,x}, Z^{t,x})$ and $v \in C^{1,2}([0,T]\times\mathbb{R}^d;\mathbb{R})$ are linked by the so called nonlinear Feynman-Kac formula:
	\begin{align}\label{FK}
		Y_s^{t,x} = v(s,X_s^{t,x});\qquad  Z_s^{t,x} = \nabla_x v(s,X_s^{t,x}),
	\end{align}
	in this case, $v$ is called the ``{\it decoupling field }" since it breaks the mutual influence between the forward and backward components in equation \eqref{eqmainF}.  
	{\begin{defi}\label{defi}
			We say that, a triple $(X^{t,x},Y^{t,x}, Z^{t,x})$ is a \textsl{weak solution} to equation \eqref{eqmainF}, if there exists a standard setup  $(\Omega,\mathfrak{F},\mathbb{F}=\{\mathfrak{F}\},\mathbb{P}, W),$ such that $(X^{t,x},Y^{t,x},Z^{t,x})\in \mathcal{S}^2(\mathbb{R}^d) \times \mathcal{S}^{\infty}(\mathbb{R}) \times \mathcal{H}_{\text{BMO}}(\mathbb{R}^d)$ and $\mathbb{P}\text{-a.s.}$ \eqref{eqmainF} is satisfied.		
			This weak solution is called a strong solution if $\mathbb{F} =\mathbb{F}^{W}.$ 
	\end{defi}}
	
	We say that the coefficients $b,\sigma,f$ and $\phi$ satisfy Assumption \ref{assum1} if there exist nonnegative constants $\Lambda,\lambda,K, K_0,$ and $\alpha_0,\beta\in(0,1)$ such that 
	\begin{assum}\label{assum1} \leavevmode
		\begin{itemize}
			\item[(AX)] 
			\begin{itemize}
				\item[(i)]$\forall t \in [0,T], \forall(x,y,z)\in \mathbb{R}^d \times \mathbb{R} \times \mathbb{R}^{d},$ $\forall \xi \in \mathbb{R}^d,$ \begin{align*}
					|b(t,x,y,z)| &\leq \Lambda (1+|y|+ |z|), \quad |\sigma(t,x,y)| \leq \Lambda (1+|y|) ,\\
					\langle \xi, a(t,x,y) \xi \rangle &\geq \lambda |\xi|^2, \quad a=\sigma\sigma^{{\bf T}}.
				\end{align*}
				\item[(ii)] For all $(t,x,y,z), (t,x,y',z') \in [0,T]\times\mathbb{R}^d\times \mathbb{R} \times \mathbb{R}^{d}$,
				\begin{align*}
					|b(t,x,y,z)-b(t,x,y',z')| &\leq K ( |y-y'| + |z-z'| ),\\
					|a(t,x,y)-a(t,x,y')| &\leq K|y-y'|,\quad |a(t,x,y)-a(t,x',y)| \leq K_0|x-x'|^{\alpha_0}.
				\end{align*}
			\end{itemize}
			\item[(AY)]
			\begin{itemize}
				\item[(i)] $\forall t \in [0,T], \forall(x,y,z)\in \mathbb{R}^d \times \mathbb{R} \times \mathbb{R}^{d},$
				\begin{align*}
					|f(t,x,y,z)| &\leq \Lambda(1 + |y|+\ell(y)|z|^2),
				\end{align*}
				where $\ell\in L^1_{\text{loc} } (\mathbb{R}, \mathbb{R}_{+})$ is locally bounded and increasing.
				\item[(ii)] For all $(t,x,y,z), (t,x,y',z') \in [0,T]\times\mathbb{R}^d\times \mathbb{R} \times \mathbb{R}^{d}$
				\begin{align*} 
					|f(t,x,y,z)-f(t,x,y',z')|&\leq K \Big(1+  \ell(|y-y'|^2)(|z|+|z'|)\Big)\Big(|y-y'| + |z-z'|\Big) 
				\end{align*}
				\item[(iii)] For all $(x,x')\in \mathbb{R}^d\times \mathbb{R}^d$
				\begin{align*}
					|\phi(x)|\leq \Lambda , \quad |\phi(x)-\phi(x')| &\leq K_0|x-x'|^{\beta}.
				\end{align*}	
			\end{itemize}
		\end{itemize}
	\end{assum}

	\begin{remark}\label{remark1}
		{When the function $\ell$ is a constant, the solvability of the \textcolor{blue}{FBSDE} \eqref{eqmainF} is derived in \cite{Delarue2}. }
		The condition (AY)(ii) will be useful to establish both the uniqueness in law and the pathwise uniqueness of solutions to the FBSDE \eqref{eqmainF}.
	\end{remark} 
	
	Below we state the main result of this section
	\begin{thm}\label{th3.3}
		Let Assumption \ref{assum1} be in force. Then the PDE \eqref{eqmainP} has a unique solution $v \in  W_{\text{Loc}}^{1,2,d+1}([0,T[\times\mathbb{R}^d,\mathbb{R} )$ such that $(t,x)\mapsto v(t,x) $ and $(t,x)\mapsto \nabla_xv(t,x) $ are continuous.
		
		Furthermore, the FBSDE \eqref{eqmainF} admits a unique weak solution $(\Omega,\{\mathfrak{F}\},\mathbb{P}),(B,\{\mathfrak{F}_t\}_{0\leq t\leq T})$, $(X,Y,Z)\in \mathcal{S}^2(\mathbb{R}^d)\times \mathcal{S}^{\infty}(\mathbb{R})\times \mathcal{H}_{\text{BMO}}(\mathbb{R}^d)$ with initial value $(t,x)\in [0,T]\times\mathbb{R}^d$ and the uniqueness is understood in the sense of probability law.
		
		Assume that $d=1$ and $\alpha_0 \geq 1/2$, for any $\delta \in (0,T)$ the FBSDE \eqref{eqmainF} has a unique strong solution for all $t \in [0,T-\delta]$.
	\end{thm}
	\begin{proof} The first statement of the Theorem requires several auxiliary results both from PDE theory and stochastic analysis. For the sake of simplicity, we are not going to reproduce the details here. We refer the reader to \cite{Delarue2} (see also \cite{RhOlivOuk}) for a {similar treatment. We start with the existence result  for the PDE  \eqref{eqmainP}} 
		
		(i).	\textbf{Solvability of PDE \eqref{eqmainP}:} {The result is proved in three steps. In the first step, we approximate the coefficients $(b,\sigma,\phi,g)$ by smooth functions (via mollifiers for instance) and derive the solution $(v_n)_{n\geq 1}$ of the regularized PDE \eqref{eqmainP}. In the second step we provide uniform bounds of $(v_n)_{n\ge1 }$ and those of their derivatives. In the last step, we use a compactness argument to construct  the decouple field of the original PDE.}

		\textbf{Step 1:} \textsl{Approximation of coefficients $(b,\sigma,f,\phi)$ by smooth functions $(b_n,\sigma_n,f_n,\phi_n).$}
		
		Let us denote by $(b_n)_{n\in \mathbb{N}},(\sigma_n)_{n\in \mathbb{N}},(\phi_n)_{n\in \mathbb{N}}$ and $ (g_n)_{n\in \mathbb{N}}$ the approximating sequences of $b, \sigma,\phi$ and $g$ respectively that can be obtained via the usual mollifiers (see for instance \cite{RhOlivOuk}). Then $b_n,\sigma_n, \phi_n$ and $g_n$ are infinitely differentiable with respect to each of their respective components, bounded with bounded derivatives of any order with compact support. Thus, from \cite{Ladyzhenskaya} the regularized PDE \eqref{eqmainP} with coefficients $(b_n,\sigma_n,f_n,\phi_n)$ in lieu of $(b,\sigma,f,\phi)$ admits a unique bounded classical solution $(v_n)_{n\geq 1}\in C^{1,2}([0,T]\times\mathbb{R}^d,\mathbb{R})$ such that $(\nabla_x v_n)_{n\geq 1}$, $(\partial v_n)_{n\geq 1}$ and $(\nabla_{xx}^2 v_n)_{n\geq 1}$ are also bounded and H\"older continuous on $[0,T]\times \mathbb{R}^d$. In addition for any prescribed probability space $(\Omega,\mathfrak{F},\mathbb{P})$ the following FBSDE:
		\begin{equation}
			\begin{cases}
				X_s^{t,x} = x +  \int_{t}^{s} b_n(r,X_r^{t,x},Y_r^{t,x},Z_r^{t,x})\mathrm{d}r +  \int_{t}^{s}\sigma_n(r,X_r^{t,x},Y_r^{t,x}) \mathrm{d}W_r,\\
				Y_s^{t,x} = \phi_n(X_T^{t,x})+  \int_{s}^{T} f_n(r,X_r^{t,x},Y_r^{t,x},Z_r^{t,x})\mathrm{d}s - \int_{s}^{T} Z_r^{t,x} \mathrm{d}M_r^{X},
			\end{cases}
		\end{equation}
		has a unique strong solution $(X^{t,x,n},Y^{t,x,n},Z^{t,x,n}) \in \mathcal{S}^2(\mathbb{R}^d)\times \mathcal{S}^{\infty}(\mathbb{R})\times \mathcal{H}_{\text{BMO}}(\mathbb{R}^d)$ such that the following relations hold (see \cite{MPY94})
		\begin{align}\label{3.3}
			Y^{t,x,n}_{\cdot} = v_n(\cdot,X^{t,x,n}_{\cdot}),\,\, Z^{t,x,n}= \nabla_x v_n(\cdot,X^{t,x,n}_{\cdot}).
		\end{align}
		
		\textbf{Step 2:} \textsl{A-priori estimates for $(v_n)_{n\geq 1}$ and for its derivatives}

		{We will start by establishing the following uniform bound of the solution $(v_n)_{n\in \mathbb{N}}$
			\begin{align}\label{3.4}
				\sup_{n\in \mathbb{N}}\sup_{\tiny (t,x)\in [0,T]\times \mathbb{R}^d} |v_n(t,x)|< \infty.
			\end{align}
			Thanks to the probabilistic representation given by \eqref{3.3}, then it is enough to prove that	
			\begin{align}\label{3.5}
				\mathbb{P}\text{-a.s.},\,\, \sup_{n\in \mathbb{N}}(\| Y^{t,x,n} \|_{\infty} + \| Z^{t,x,n} \|_{\mathcal{H}_{BMO}} )< \infty.
			\end{align} 
		}	
		To ease the notation, in the following we will omit the superscript $(t,x)$.	 	
		Let $(f_n)_{n\geq 1}$ be the usual mollifier of the function $f$. Then from Assumption \ref{assum1} (AY)  (see \cite{RhOlivOuk}), there exists $C> 0$ such that for all $(t,x,y,y',z,z')\in \mathbb{R}\times \mathbb{R}^d\times \mathbb{R}\times \mathbb{R}\times\mathbb{R}^d\times \mathbb{R}^d$
		\begin{align*} 
			|f_n(t,x,y,z)| &\leq C\Lambda (1+ |y|+ \ell_n(y)|z|^2),\\
			|f_n(t,x,y,z)-f_n(t,x,y',z')| &\leq C (1 + \ell(|y-y'|^2)(1+|z|+ |z'|))(|y-y'|+|z-z'|) ,
		\end{align*}
		where $(\ell_n)_{n\geq 1}$ stands for the mollifier of the function $\ell$. Therefore, the process $A_t^n$ defined by:
		$$|Z_t^n|^2A_t^n = \sigma_n^{-1}(t,X_t^n,Y_t^n)(f_n(t,X_t^n,Y_t^n,Z_t^n)-f_n(t,X_t^n,Y_t^n,0))Z_t^n 1_{\{Z_t^n\neq 0\}}$$
		belongs to the space $\mathcal{H}_{\text{BMO}}(\mathbb{R}^d)$, since $(Y^n,Z^n) \in\mathcal{S}^{\infty}(\mathbb{R})\times\mathcal{H}_{\text{BMO}}(\mathbb{R}^d)$ for each $n \in \mathbb{N}$. The Girsanov's theorem ensures that $W_t^n = W_t - \int_0^t A_s^n\mathrm{d}s$ is a Brownian motion under the probability measure $\mathbb{P}^n $ defined by $\mathrm{d}\mathbb{P}^n := \mathcal{E}(\int_{0}A_t^n\mathrm{d}W_t) \mathrm{d}\mathbb{P}.$  Therefore,
		\begin{align*}
			Y_t^n = \phi_n(X_T^n) + \int_{t}^{T} f_n(s,X_s^n,Y_s^n,0)\mathrm{d}s - \int_t^T Z_s^n \sigma_n^{{\bf T}}(s,X_s^n,Y_s^n)\mathrm{d}W_s^n.
		\end{align*}
		By taking the conditional expectation {on both sides of the above equation}, we deduce that
		\begin{align*}
			|Y_t^n|&\leq \mathbb{E}^{\mathbb{P}^n}\Big[ |\phi_n(X_T^n)| + \int_{t}^{T} |f_n(s,X_s^n,Y_s^n,0)|\mathrm{d}s/\mathfrak{F}_t\Big]\\
			& \leq \Lambda + C\Lambda T + C\Lambda \mathbb{E}^{\mathbb{P}^n}[ \int_{t}^{T} |Y_s^n|\mathrm{d}s/\mathfrak{F}_t ].
		\end{align*}
		The first bound in \eqref{3.5} then follows by applying the stochastic Gronwall's lemma (see for instance \cite{Wang}). Let us turn now to the second estimate in \eqref{3.5}. For this, we consider the $W_{\text{loc}}^{1,2}(\mathbb{R})\text{-function}$
		\begin{align*}  
			I(x) :=  \int_{0}^{x} K(y)\exp\Big(2 \int_{0}^{y}\ell(u)\mathrm{d}u\Big)\mathrm{d}y, 
		\end{align*}
		where $ K(y):= \int_{0}^{y} \exp\Big(-2 \int_{0}^{z}\ell(u)\mathrm{d}u\Big)\mathrm{d}z$ and the function {$\ell$} satisfies almost everywhere: $1/2I''(x)-\ell(x)I'(x) = 1/2$ (see \cite{Bahlali1}). Then, using the It\^{o}-Krylov formula for BSDE (see \cite[Theorem 2.1]{Ouknine}) we obtain that
		\begin{align*}
			I(|Y_{\tau}^n|) &= I(|Y_T^n|) + \int_{\tau}^{T} \sgn(Y_u^n)v'(|Y_u^n|)f(u,X_u^n,Y_u^n,Z_u^n) - 1/2 I''(|Y_u^n|)|Z_u^n|^2\mathrm{d}u + M_{\tau}^n\notag\\
			&\leq I(|Y_T^n|) -\frac{1}{2} \int_{\tau}^{T} |Z_u^n|^2\mathrm{d}u + \Lambda \int_{\tau}^{T}( 1 + |Y_u^n|)I'(|Y_u^n|)\mathrm{d}u +  M_{\tau}^n
		\end{align*}
		for any stopping time $\tau$ and $ M_{\tau}^n$ represents the martingale part. By taking the conditional  expectation with respect to $\mathfrak{F}_{\tau}$, we deduce that
		\begin{align*}
			\frac{1}{2} \mathbb{E}\Big[\int_{\tau}^{T} |Z_u^n|^2\mathrm{d}u/\mathfrak{F}_{\tau}\Big]	&\leq \mathbb{E}\Big[ I(|Y_T^n|)  + \Lambda\int_{\tau}^{T}(1 +|Y_u^n|)I'(|Y_u^n|)\mathrm{d}u/ \mathfrak{F}_{\tau} \Big].
		\end{align*}
		Then, the sought estimate follows from the uniform boundedness of $Y^n$.	
		
		Moreover, by using the classical machinery from the theory of SDEs along with Assumption \ref{assum1} and \eqref{3.5}, one can obtain the following 
		\begin{align}\label{3.8}
			\mathbb{P}\text{-a.s.},\,\, \sup_{n\in \mathbb{N}}\| X^{t,x,n} \|_{\mathcal{S}^2(\mathbb{R}^d)} < \infty.
		\end{align}

		On the other hand, from \eqref{3.4}, one can derive the uniform H\"{o}lder estimate for $(v_n)_{n\geq 1}$ in the same way as in \cite[Lemma 2.2]{RhOlivOuk} i.e., there exist $\beta_0\in (0,\beta], C>0$ such that for all $(t,x),(s,y) \in [0,T]\times \mathbb{R}^d$, it holds
		\begin{align}\label{holder}
			\sup_{n\in\mathbb{N}}|v_n(t,x)-v_n(s,y)| \leq C (|t-s|^{\frac{\beta_0}{2}} + |x-y|^{{\beta_0}}).
		\end{align}

		In the same spirit, there exists $\gamma > 0$, such that
		\begin{align}\label{3.6}
			\sup_{n\in \mathbb{N}}\sup_{\tiny (t,x)\in [0,T[\times \mathbb{R}^d} (T-t)^{1-\gamma}|\nabla_xv_n(t,x)|< \infty
		\end{align}
		follows and as in \cite[Lemma 2.4 and Lemma 2.5 ]{RhOlivOuk}), there exists universal constants $C$ and $\beta'$ such that  
		
		\begin{align}\label{holder1}
			\sup_{n\in\mathbb{N}}|\nabla_xv_n(t,x)-\nabla_xv_n(s,y)| \leq C (T-t)^{(-1+\beta')} (|t-s|^{\frac{\beta'}{2}} + |x-y|^{{\beta'}}),
		\end{align}

		{Note that the coefficients $b,\sigma, g$ of the FBSDE \eqref{eqmainF} are not smooth in the time variable and the terminal condition is only H\"older continuous, thus \eqref{3.6} and \eqref{holder1} cannot be obtained for $(\partial_tv_n)_{n\geq 1}$ and $(\nabla_{x,x}^2 v_n)_{n\geq 1}$. Nevertheless we can derive as in \cite{Delarue2, RhOlivOuk}), the following  Calderon-Zygmung types estimate: there exists $\alpha > 0,$ such that} 
		\begin{align}\label{3.7}
			\sup_{n\in \mathbb{N}}\int_{T-\delta}^{T}\int_{B(\zeta,R)}  [ (T-s)^{1-\alpha} (|\partial_t v_n(s,y)| + |\nabla_{x,x}^2 v_n(s,y)|)  ]^p \mathrm{d}s\mathrm{d}y < \infty,
		\end{align}
		where $B(\zeta,R)$ denotes the $\mathbb{R}^d$ ball of centre $\zeta$ and radius $R$.
		
		\textbf{Step 3:} \textsl{Arzela-Ascoli arguments} 
		
		From \eqref{3.4}, \eqref{holder}, \eqref{3.6} and \eqref{holder1}, we observe that the sequences $(v_n)_{n\geq 1}$ and $(\nabla_xv_n)_{n\geq 1}$  are uniformly bounded and equicontinuous respectively on $[0,T]\times\mathbb{R}^d$ and $[0,T-\delta]\times\mathbb{R}^d$ for every $\delta>0$.  Thus, by Arzela-Ascoli theorem,  there exist converging subsequences respectively denoted by $(v_n)_{n\geq 1}$ and $(\nabla_xv_n)_{n\geq 1}$ (still indexed by $n$) such that their limits  $v$ and $v'$ respectively are continuous. From the uniqueness of the limit we note that $v'=\nabla_x v$. On the other hand, \eqref{3.7} suggests that both $(\partial_tv_n)_{n\geq1}$ and $(\nabla_{x x}^2v_n)_{n\geq 1}$ are respectively converging in distribution for subsequences to $\partial_tv$ and $\nabla_{x x}^2v$, respectively. Therefore, $v$ is continuous and satisfies almost everywhere the PDE\eqref{eqmainP}.

		(ii).	\textbf{Weak solution to FBSDE \eqref{eqmainF}: } This follows by 
		combining the arguments developed in \cite[section 4.2]{Delarue2} along with bounds  \eqref{3.4} and \eqref{3.6}. We denote by $(\Omega,\{\mathfrak{F}\},\mathbb{P}),(W,\{\mathfrak{F}_t\}_{0\leq t\leq T}), (X,Y,Z)$ this solution with initial condition $(t,x)$ such that $X_t$ is the unique weak solution to the equation
		\begin{align}\label{SDE11}
			X_t = x + \int_0^t \tilde{b}(s,X_s)\mathrm{d}s + \int_0^t \tilde{\sigma}(s,X_s)\mathrm{d}W_s,
		\end{align}
		where $\tilde{b}(t,x)= b(t,x,v(t,x),\nabla_xv(t,x))$ and $\tilde{\sigma}(t,x)= \sigma(t,x,v(t,x))$ and the relationship \eqref{FK} holds. 
		Moreover, due to the bounds \eqref{3.5}, we have that $Z \in \mathcal{H}_{BMO}.$

		(iii).	\textbf{Uniqueness in probability law: }
		Let $(\bar{\Omega},(\bar{\mathfrak{F}}),\bar{\mathbb{P}}), (\tilde{W},\{\bar{\mathfrak{F}_t}\}_{0\leq t\leq T})),(\bar{X},\bar{Y},\bar{Z})$ be another solution to FBSDE \eqref{eqmainF} with initial condition given by $(0,x),  x\in \mathbb{R}^d$ i.e.:
		
		\begin{align}\label{eqchang}
			\begin{cases}
				\bar{X}_t = x + \displaystyle \int_{0}^{t}b(s,\bar{X}_s,\bar{Y}_s, \bar{Z}_s)\mathrm{d}s +  \int_{0}^{t}\sigma(s,\bar{X}_s,\bar{Y}_s)\mathrm{d}\tilde{W}_s, \\
				\bar{Y}_t = \phi(\bar{X}_T) + \displaystyle \int_{t}^{T} f(s,\bar{X}_s,\bar{Y}_s, \bar{Z}_s) \mathrm{d}s -\int_{t}^{T} \bar{Z}_s \mathrm{d}M^{\bar{X}}_s.
			\end{cases}
		\end{align}

		Our aim is to show that $\mathbb{P}\circ (X,Y,Z,W)^{-1} = \bar{\mathbb{P}}\circ(\bar{X},\bar{Y},\bar{Z},\tilde{W})^{-1}$, where $(X,Y,Z)$ stands for the solution we obtained in point (ii) above, with the same initial condition $(0,x)$, for all $x\in\mathbb{R}^d$. More precisely, following the weak formulation of the {\it ``four step scheme"} developed in {\cite[section 4.3]{Delarue2}, it is enough to prove that $\bar{\mathbb{P}}\text{-a.s.}$ $(\bar{Y},\bar{Z}) \equiv(\mathcal{Y},\mathcal{Z})$, where the processes $\mathcal{Y}$ and $\mathcal{Z}$ given by}
		\[ 
		\mathcal{Y}_t := v(t,\bar{X}_t),\quad \mathcal{Z}_t := \nabla_x v(t,\bar{X}_t),\forall t \in [0,T[. 
		\]

		
		Next, we first need to write $\mathcal{Y}$ as the solution of a backward SDE. We recall that the function $v$ solution to the PDE \eqref{eqmainP} belongs to ${W}_{\text{loc}}^{1,2,d+1}([0,T[\times \mathbb{R}^d,\mathbb{R})$ then the It\^o-Krylov formula can be applied in order to derive the desired backward equation, provided that the drift term from the forward equation is uniformly bounded(see \cite[Chapter 2, Section 10, Theorem 1]{Krylov1} ). Unfortunately, here the drift $b$ does not satisfied this requirement under Assumption \ref{assum1} (AX), since the process $\bar{Z}$ is only bounded in the $\mathcal{H}_{\tiny \text{BMO}}$ norm. To overcome this situation, we will use the fact that the Doleans-Dade exponential of a BMO martingale is uniformly integrable and it defines a proper density, in contrast to \cite{Delarue2}, where the authors rather used a localization arguments. More precisely, thanks to Assumption \ref{assum1} (AX) and the boundedness of $\bar{Y}$ (since $\bar{Y} \in \mathcal{S^{\infty}}$ by Definition \ref{defi} ) we observe that:
		\begin{align*}
			|\sigma^{-1}(s,\bar{X}_s,\bar{Y}_s)b(s,\bar{X}_s,\bar{Y}_s,\bar{Z}_s)| \leq C \left( 1 + |\bar{Z}_s| \right) \in \mathcal{H}_{\tiny \text{BMO}}.
		\end{align*}
		Thus, the measure $\bar{\mathbb{Q}}$ defined by 
		\[ 
		\frac{\mathrm{d}\bar{\mathbb{Q}}}{\mathrm{d}\bar{\mathbb{P}}} := \exp\Big( -\int_{0}^{t} \langle \sigma^{-1}(s,\bar{X}_s,\bar{Y}_s)b(s,\bar{X}_s,\bar{Y}_s,\bar{Z}_s), \mathrm{d}\tilde{W}_s \rangle -\frac{1}{2}\int_{0}^{t}  \big| \sigma^{-1}(s,\bar{X}_s,\bar{Y}_s)b(s,\bar{X}_s,\bar{Y}_s,\bar{Z}_s) \big|^2\mathrm{d}s  \Big),
		\]
		is uniformly integrable for all $t \in [0,T]$. Moreover, from the Girsanov theorem the process $\bar{W}$ defined for all $t\in[0,T]$ by
		\begin{align}
			\bar{W}_t := \tilde{W}_t + \int_{0}^{t}  \sigma^{-1}(s,\bar{X}_s,\bar{Y}_s)b(s,\bar{X}_s,\bar{Y}_s,\bar{Z}_s) \mathrm{d}s  
		\end{align}
		is an $\{\bar{\mathfrak{F}}\}_{0\leq t \leq T}$-Brownian motion under the measure $\bar{\mathbb{Q}}.$ Hence, the FBSDE \eqref{eqchang} takes the following form under the measure $\bar{\mathbb{Q}}$: 
		\begin{align}\label{eqchang1BM}
			\begin{cases}
				\bar{X}_t = x + \displaystyle \int_{0}^{t}\sigma(s,\bar{X}_s,\bar{Y}_s)\mathrm{d}\bar{W}_s, \\
				\bar{Y}_t = \phi(\bar{X}_T) + \displaystyle \int_{t}^{T} G(s,\bar{X}_s,\bar{Y}_s,\bar{Z}_s) \mathrm{d}s -\int_{t}^{T} \langle \bar{Z}_s ,\sigma(s,\bar{X}_s,\bar{Y}_s)\mathrm{d}\bar{W}_s\rangle,
			\end{cases}
		\end{align}
		where $G(t,x,y,z) = \langle b(t,x,y),z \rangle + f(t,x,y,z)$. Using once more Assumption \ref{assum1}, one can show that, the latter satisfies  for all $(t,x,y,y',z,z')\in [0,T]\times \mathbb{R}^d\times\mathbb{R}\times \mathbb{R}\times\mathbb{R}^{d}\times\mathbb{R}^d$
		\begin{align}
			&|G(t,x,y,z)| \leq  C\Big( 1+|y| +(1+\ell(y))|z|^2 \Big), \label{10-}\\
			&|G(t,x,y,z) - G(t,x,y',z')| \notag \\
			&\leq  C \Big( 1+|y|+2|z|+|z-z'| +\ell(|y-y'|^2)\big( 2|z|+ |z'-z|\big)\Big) \times (|y-y'| + |z-z'|).\label{10}
		\end{align}
		We further apply the It\^{o}-Krylov formula to the semimartingale $v(t,\bar{X}_t)$, to obtain for all $t\in [0,\tau]$ 
		\begin{align}\label{D}
			\mathrm{d}\mathcal{Y}_t= \left( \frac{\partial v}{\partial t} + \frac{1}{2} a\nabla_{xx}^2 v \right)(t,\bar{X}_t)\mathrm{d}t + \langle \nabla_x v(t,\bar{X}_t), \sigma(t,\bar{X},\mathcal{Y}_t) \mathrm{d}\bar{W}_t \rangle \quad \bar{\mathbb{Q}}\text{-a.s.},
		\end{align}
		where $\tau$ is  an $\bar{\mathfrak{F}}_t$ stopping time.
		Then, by using equation \eqref{eqmainP} with the fact that $\mathcal{Z}_t := \nabla_x v(t,\bar{X}_t)$, the dynamics of the process $(\bar{Y}_t- \mathcal{Y}_t)_{0\leq t\leq \tau}$ is given by
		\begin{align}
			\mathrm{d}(\bar{Y}- \mathcal{Y})_t =& -\frac{1}{2} \left( a(t,\bar{X}_t,\bar{Y}_t)- a(t,\bar{X}_t,\mathcal{Y}_t) \right) \nabla_{x x}^2 v(t,\bar{X}_t)\mathrm{d}t \notag \\
			&   -\Big( G(t,\bar{X}_t,\bar{Y}_t,\bar{Z}_t) -G(t,\bar{X}_t,\mathcal{Y}_t,\mathcal{Z}_t) \Big)\mathrm{d}t + \langle \bar{Z}_t-\mathcal{Z}_t, \sigma(t,\bar{X}_t,\bar{Y}_t)  \mathrm{d}\bar{W}_t  \rangle.
		\end{align}   	
		We recall that the processes $\mathcal{Y}$ and $\bar{Y}$ are both {uniformly bounded, hence we define $L_0$  by $L_0 := 2(\Vert \mathcal{Y} \Vert_{\infty}^2 + \Vert \bar{Y} \Vert_{\infty}^2).$} Let us consider the following positive and locally bounded function $h \in L^{1}_{\text{Loc}}(\mathbb{R})$: \[ h(y):=\kappa_1+\kappa_2\ell(y),\,\,\,\kappa_1,\kappa_2\in\mathbb{R}_+.\]

		For all $y\in [0,L_0]$, consider the function $\Phi_h \in W_{\text{loc}}^{1,2}(\mathbb{R})$ defined by:
		\[ \Phi_h (z):= \int_{0}^{z}\!\! \exp\left( \kappa  \int_{0}^{y}\!\! h(t)\mathrm{d}t\right) \mathrm{d}y, \]
		with $\kappa$ a free nonnegative parameter. The function $\Phi_h$ satisfies almost everywhere the differential equation $\Phi_h''(z) - \kappa h(z)\Phi_h'(z) = 0 $. Moreover for $|z|\leq L_0$, we have $0\leq z\Phi_h'(z) \leq \exp\left(  \kappa \|h\|_{L^1([0,L_0])} \right)\Phi_h(z)$.

		Then from the It\^{o}'s formula applied the function $\Phi_h$  (see  for example \cite[Theorem 2.1]{Ouknine}) we deduce for all $t\in [0,\tau]$:
		\begin{align}\label{eqIK1}
			&\Phi_h(|\bar{Y}_t- \mathcal{Y}_t|^2)\notag\\
			=& \Phi_h(|\bar{Y}_{\tau}-\mathcal{Y}_{\tau}|^2) + \int_{t}^{\tau} \Phi_h'(|\bar{Y}_s-\mathcal{Y}_s|^2)(\bar{Y}_s-\mathcal{Y}_s)\left( a(s,\bar{X}_s,\bar{Y}_s)- a(s,\bar{X}_s,\mathcal{Y}_s) \right) \nabla_{x x}^2 v(s,\bar{X}_s)\mathrm{d}s \notag\\
			& +2 \int_{t}^{\tau} \Phi_h'(|\bar{Y}_s-\mathcal{Y}_s|^2)(\bar{Y}_s-\mathcal{Y}_s)\left( G(t,\bar{X}_s,\bar{Y}_s,\bar{Z}_s) -G(s,\bar{X}_s,\mathcal{Y}_s,\mathcal{Z}_s) \right)\mathrm{d}s \notag\\
			&-2 \int_{t}^{\tau} \Phi_h'(|\bar{Y}_s-\mathcal{Y}_s|^2)(\bar{Y}_s-\mathcal{Y}_s)\langle \bar{Z}_s -\mathcal{Z}_s, \sigma(s,\bar{X}_s,\bar{Y}_s)  \mathrm{d}\bar{W}_s  \rangle \notag\\
			&- \int_{t}^{\tau} \Phi_h'(|\bar{Y}_s-\mathcal{Y}_s|^2)\langle \bar{Z}_s-\mathcal{Z}_s, a(s,\bar{X}_s,\bar{Y}_s)(\bar{Z}_s-\mathcal{Z}_s) \rangle \mathrm{d}s \notag\\
			&-2 \int_{t}^{\tau} \Phi_h''(|\bar{Y}_s-\mathcal{Y}_s|^2)\langle \bar{Z}_s-\mathcal{Z}_s, a(s,\bar{X}_s,\bar{Y}_s)(\bar{Z}_s-\mathcal{Z}_s) \rangle \mathrm{d}s -\int_{t}^{\tau} \mathrm{d}\mathcal{M}_s,
		\end{align}
		where $ \mathrm{d}\mathcal{M}_s= 2{1}_{\{s\leq {\tau}\}}  \Phi_h'(|\bar{Y}_s-\mathcal{Y}_s|^2)(\bar{Y}_s-\mathcal{Y}_s)\langle \bar{Z}_s-\mathcal{Z}_s, \sigma(s,\bar{X}_s,\bar{Y}_s)  \mathrm{d}\bar{W}_s  \rangle$ is a square integrable martingale.
		Using Assumption \ref{assum1} and \eqref{10}, we obtain that
		\begin{align}\label{eqIK2}
			&\Phi_h(|\bar{Y}_t- \mathcal{Y}_t|^2)\notag\\
			\leq& \Phi_h(|\bar{Y}_{\tau}-\mathcal{Y}_{\tau}|^2) + K \int_{t}^{\tau} \Phi_h'(|\bar{Y}_s-\mathcal{Y}_s|^2)|\bar{Y}_s-\mathcal{Y}_s|^2 |\nabla_{xx}^2 v(s,\bar{X}_s)|\mathrm{d}s\notag\\
			&+ 2C \int_{t}^{\tau} \Phi_h'(|\bar{Y}_s-\mathcal{Y}_s|^2)|\bar{Y}_s-\mathcal{Y}_s|\Big(1+|\mathcal{Y}_s|+2|\mathcal{Z}_s|+|\bar{Z}_s-\mathcal{Z}_s|
			\notag\\
			&+ \ell(|\bar{Y}_s-\mathcal{Y}_s|^2)\Big\{2|\mathcal{Z}_s|+|\bar{Z}_s-\mathcal{Z}_s|\Big\}\Big)\Big(|\bar{Y}_s-\mathcal{Y}_s|+|\bar{Z}_s-\mathcal{Z}_s|\Big)\mathrm{d}s\notag\\
			&-\lambda \int_{t}^{\tau}2\kappa|\bar{Y}_s-\mathcal{Y}_s|^2 |\bar{Z}_s-\mathcal{Z}_s|^2  h(|\bar{Y}_s-\mathcal{Y}_s|^2)\Phi_h'(|\bar{Y}_s-\mathcal{Y}_s|^2)\mathrm{d}s\notag\\
			& -\lambda \int_{t}^{\tau} \Phi_h'(|\bar{Y}_s-\mathcal{Y}_s|^2) |\bar{Z}_s-\mathcal{Z}_s|^2 \mathrm{d}s-\int_{t}^{\tau} \mathrm{d}\mathcal{M}_s.
		\end{align}
		From the local boundedness and positivity of $\ell$, there exists $M_1>0$ such that  $\ell(|\bar{Y}_s-\mathcal{Y}_s|^2)\leq M_1$. Using  \eqref{3.6} and applying repeatedly the Young inequality, \eqref{eqIK2} becomes
		\begin{align}\label{eqIK5}
			&	\Phi_h(|\bar{Y}_t- \mathcal{Y}_t|^2)+\lambda \int_{t}^{\tau} \Phi_h'(|\bar{Y}_s-\mathcal{Y}_s |^2) |\bar{Z}_s-\mathcal{Z}_s|^2\Big(1+2\kappa \kappa_1 |\bar{Y}_s-\mathcal{Y}_s |^2+ 2\kappa \kappa_2 |\bar{Y}_s-\mathcal{Y}_s |^2\ell(|\bar{Y}_s-\mathcal{Y}_s |^2)\Big)\mathrm{d}s\notag\\
			\leq& \Phi_h(|\bar{Y}_{\tau}-\mathcal{Y}_{\tau}|^2) + K_{\epsilon_1,\epsilon_2, M_1,\epsilon_5, L_0}(1+L_0) \int_{t}^{\tau} \Phi_h'(|\bar{Y}_s-\mathcal{Y}_s|^2)|\bar{Y}_s-\mathcal{Y}_s|^2 (1+(T-s)^{-1+\gamma} +|\nabla_{xx}^2 v(s,\bar{X}_s)|)\mathrm{d}s\notag\\
			&+ K(\epsilon_1 +\frac{1}{\epsilon_3})\int_{t}^{\tau} \Phi_h'(|\bar{Y}_s-\mathcal{Y}_s|^2)|\bar{Y}_s-\mathcal{Y}_s |^2|\bar{Z}_s-\mathcal{Z}_s|^2\mathrm{d}s \notag\\
			&+K(\epsilon_2+\epsilon_3+M\epsilon_5+M\epsilon_6) \int_{t}^{\tau} \Phi_h'(|\bar{Y}_s-\mathcal{Y}_s |^2)|\bar{Z}_s-\mathcal{Z}_s|^2\mathrm{d}s
			\notag\\
			&+K(\epsilon_4+\frac{1}{\epsilon_6})\int_{t}^{\tau}\Phi_h'(|\bar{Y}_s-\mathcal{Y}_s |^2)|\bar{Y}_s-\mathcal{Y}_s |^2\ell(|\bar{Y}_s-\mathcal{Y}_s |^2)|\bar{Z}_s-\mathcal{Z}_s|^2\mathrm{d}s -\int_{t}^{\tau} \mathrm{d}\mathcal{M}_s,
		\end{align}
		where $K$ is a positive constant that may change from line to line.
		Choosing  $\epsilon_2=\epsilon_3=\frac{\lambda}{16K}, \epsilon_5=\epsilon_6=\frac{\lambda}{16KM_1}$, $\kappa, \kappa_1$ and $\kappa_2$ such that $2\kappa\kappa_1= \epsilon_1 +\frac{1}{\epsilon_3}$ and $2\kappa\kappa_2=\epsilon_4 +\frac{1}{\epsilon_6}$ and using the relation  $ z\Phi_h'(z) \leqslant \exp(\kappa|| h ||_{L^1([0,L_0])})\Phi_h(z) $ for all $z \in [0,L_0]$, we get	
		\begin{align}\label{eqIK6}
			&	\Phi_h(|\bar{Y}_t- \mathcal{Y}_t|^2)+\frac{3}{4}\lambda \int_{t}^{\tau} \Phi_h'(|\bar{Y}_s-\mathcal{Y}_s |^2) |\bar{Z}_s-\mathcal{Z}_s|^2\mathrm{d}s\notag\\
			\leq& \Phi_h(|\bar{Y}_{\tau}-\mathcal{Y}_{\tau}|^2) + K\int_{t}^{\tau} \Phi_h(|\bar{Y}_s-\mathcal{Y}_s|^2) (1+(T-s)^{-1+\gamma} +|\nabla_{xx}^2 v(s,\bar{X}_s)|)\mathrm{d}s-\int_{t}^{\tau} \mathrm{d}\mathcal{M}_s,
		\end{align}
		where $K$ is a constant depending on $\epsilon_1,\epsilon_2, M_1,\epsilon_5, L_0$ and $\|h\|_{L^1([0,L_0])}$. Taking the conditional expectation on both sides of \eqref{eqIK6}, we have
		\begin{align}\label{eqIK7}
			&\boldmath{1}_{ \{t\leq \tau\}}\Phi_h(|\bar{Y}_t- \mathcal{Y}_t|^2) +  \frac{3\lambda}{4} \bar{\mathbb{E}}^{\bar{\mathbb{Q}}}\Big[ \boldmath{1}_{ \{t\leq \tau\}} \int_{t}^{\tau} \Phi_h'(|\bar{Y}_s-\mathcal{Y}_s |^2) |\bar{Z}_s-\mathcal{Z}_s|^2\mathrm{d}s \Big| \bar{\mathfrak{F}}_t\Big] \notag\\
			\leq & \bar{\mathbb{E}}^{\bar{\mathbb{Q}}}\Big[ \boldmath{1}_{ \{t\leq \tau\}}\Phi_h(|\bar{Y}_{\tau}-\mathcal{Y}_{\tau}|^2)\Big|\mathfrak{F}_t\Big]\notag\\
			& + K \bar{\mathbb{E}}^{\bar{\mathbb{Q}}}\Big[\boldmath{1}_{ \{t\leq \tau\}}\int_{t}^{\tau} \Phi_h(|\bar{Y}_s-\mathcal{Y}_s|^2) (1+(T-s)^{-1+\gamma} +|\nabla_{xx}^2 v(s,\bar{X}_s)|)\mathrm{d}s \Big|\bar{\mathfrak{F}}_t\Big].
		\end{align}

		The bound \eqref{eqIK7} is analogous to the one obtained in \cite[Section 4.3.5, equation(4.14)]{Delarue2}. We underline the fact that the second order derivatives of $v$ appearing on the right hand side of \eqref{eqIK7} are not uniformly controlled on $[0,T]\times\mathbb{R}^d$ (see \eqref{3.7}) and this prevent one to directly applying the classical Gronwall inequality in \eqref{eqIK7}.
		Thus, one needs to apply the {Krylov's inequality to control efficiently the latter (see \cite[Lemma 4.1]{Delarue2}) and then using} the non trivial discrete Gronwall's lemma developed in \cite[Section 4.3.5]{Delarue2} to derive : 
		$$ \essup_{\omega\in \Omega}\Phi_h(|\bar{Y}_t- \mathcal{Y}_t|^2) = 0, \,\,  \text{ for all } t \in [0,T].$$
		Since $\Phi_h$ is nonnegative, and thanks to the continuity of $\bar{Y}$ and $\mathcal{Y}$  we deduce $\bar{\mathbb{P}}\text{-a.s.}$
		$\bar{Y}_t = \mathcal{Y}_t $  \text{ for all } $t \in [0,T]$.  Moreover, using the fact $\Phi_h'$ is bounded from below by a positive constant and from \eqref{eqIK7} we derive
		\[   \bar{\mathbb{E}}^{\bar{\mathbb{Q}}}\Big[ {1}_{\{t\leq \tau\} } \int_{t}^{\tau} |\bar{Z}_s-\mathcal{Z}_s|^2\mathrm{d}s \Big| \bar{\mathfrak{F}}_t \Big] = 0, \]
		which implies that $\bar{Z}_t = \mathcal{Z}_t$  $\mathrm{d}t\otimes \bar{\mathbb{P}}\text{-a.s.}$ This completes the proof.

		(iv).	\textbf{Strong solution to FBSDE \eqref{eqmainF}} {We follow the idea developed in \cite{RhOlivOuk}. We restrict ourselves to $d=1$ and set $\alpha_0\ge 1/2$. Using properties of local time, we will prove that the solution to SDE \eqref{SDE11} is pathwise unique.}
		Recall that, if $X$ is a semimartingale, its local time $(t,a)\mapsto L_t^{a}$ at time $t\in [0,T]$ and level $a \in \mathbb{R}$, can be defined through the density occupation formula:
		\[ \int_0^t f(X_s)\mathrm{d}\langle X \rangle_s = \int_{\mathbb{R}}f(a)L_t^{a}(X_t)\mathrm{d}a,  \]
		for every bounded and measurable function $f:\mathbb{R}\rightarrow \mathbb{R}$.

		Let $X_t$ and $X'_t$ denote two {weak} solutions to SDE \eqref{SDE11} with the same underlying Brownian motion $\{W_t\}_{t\geq 0}$ for all $t\in [0,T]$ {we can show} as in \cite[Theorem 4.1]{RhOlivOuk} that $L_t^0(X-X') = 0$ (where $L_t^0$ stands for the local time at level $0$). {For the sake of completeness, we will briefly reproduce the proof here. Suppose there is $t \in [0,T], \varepsilon >0$ and a set $\mathcal{A} \in \mathfrak{F}$, with $\mathbb{P}(\mathcal{A}) > 0$ such that $L_t^{0}(X^{1} - X^{2})(\omega) > \varepsilon$ for $\omega \in \mathcal{A}$. Since the map $\alpha \mapsto L_t^{\alpha}(X^{1} - X^{2})$ is right continuous, then there exists $\tilde\delta >0$ such that for all $\alpha \in [0,\tilde\delta],~ L_t^{\alpha}(X^{1} - X^{2}) \geq \varepsilon/2,$ on $ \mathcal{A}$. Therefore, using the occupation-times formula we deduce that:
			\begin{align}\label{6.1}
				\int_{0}^{t} \frac{\mathrm{d}\langle X^{1} - X^{2} \rangle_s }{|X^{1}_s - X^{2}_s|^{2\alpha_0}} = \int_{0}^{+\infty} \frac{1}{|\alpha|^{2\alpha_0}}L_t^{\alpha}(X^{1} - X^{2})\mathrm{d}\alpha \geq \frac{\varepsilon}{2} \int_{0}^{\tilde\delta} \frac{1}{\alpha^{2\alpha_0}}\mathrm{d}\alpha = +\infty, ~~ \text{ on } \mathcal{A}. 
			\end{align}
			On the other hand:
			\begin{align}\label{6.3}
				\int_{0}^{t} \frac{\mathrm{d}\langle X^{1} - X^{2} \rangle_s }{|X^{1}_s - X^{2}_s|^{2\alpha_0}} = \int_{0}^{t} \frac{ (\tilde{\sigma}(s,X^1_s) - \tilde{\sigma}(s,X^2_s))^2 }{|X^{1}_s - X^{2}_s|^{2\alpha_0}}\mathrm{d}s, ~~ \text{ on } \mathcal{A}, 
			\end{align}
			the latter is bounded $\text{on } \mathcal{A}$ provided
			\[ |\tilde{\sigma}(t,X^1_t) - \tilde{\sigma}(t,X^2_t)| \leq C |X^1_t-X^2_t|^{\alpha_0} , \]
			and this is true only for all $t\in [0,T-\delta]$.
			Thus, $\mathbb{P}(\mathcal{A}) = 0,$ which is a contradiction.} Therefore, the pathwise uniqueness holds for SDE \eqref{SDE11} (see \cite[Theorem 4.2]{RhOlivOuk}).
		
		Let $(X,Y,Z)$ and $(X',Y',Z')$ be two weak solutions to FBSDE \eqref{eqmainF} with the same underlying stochastic basis $(\Omega,\mathfrak{F},(\mathfrak{F})_{t\leq s \leq T}, W,\mathbb{P}).$ Since, the solution is unique in law (see the above subsection), it holds $\mathbb{P}\circ (X,Y,Z,W)^{-1} = \mathbb{P}\circ(X',Y',Z',W)^{-1}.$ Moreover, using the pathwise uniqueness of \eqref{SDE11}, we have that $ X_t = X'_t, \forall t\in [0,T]$ $\mathbb{P}\text{-a.s.}$  The continuity of the function $v$ and relation \eqref{FK}, gives 
		\begin{align}
			\begin{cases}
				Y'_t = v(t,X'_t) = v(t,X_t) = Y_t \quad \forall t\in [0,T],\mathbb{P}\text{-a.s.}  \\
				Z'_t = \nabla_x v(t,X'_t) = \nabla_xv(t,X_t) = Z_t \quad \mathrm{d}\mathbb{P}\otimes \mathrm{d}t\text{-a.e.}
			\end{cases}
		\end{align}
		This concludes the proof.
	\end{proof}
	
	{\begin{remark}
			In the remainder of the paper, one can relax the uniform Lipschitz condition of  the drift $b$ in $y$ in  Assumption \ref{assum1} to a locally locally Lipschitz condition that is, a Lipschitz condition in the region $[0,T]\times \mathbb{R}^d\times\{ |y|\leq \Gamma \}\times \mathbb{R}^d$, where the constant $\Gamma >0$ can be taken as the uniform bound coming from \eqref{3.5}.
	\end{remark}}
	
	\begin{remark}\label{remk 3.4}
		Note that, as in \cite[Theorem 2.9]{Delarue1}, under Assumption \ref{assum1} with $\beta = 1$ and $\sigma = \mathbb{I}_{d\times d}$, there exists $\Upsilon >0$ only depending on the coefficients appearing in the assumptions such that the gradient of $v$ solution to PDE \eqref{eqmainP} satisfies:
		\begin{align*}
			\forall (t,x)\in [0,T]\times \mathbb{R}^d,\, |\nabla_x v(t,x)| \leq \Upsilon
		\end{align*}
	\end{remark}

	\section{A comonotonicity theorem for FBSDE \eqref{eqmainF}}\label{comonotonicity} In this section we establish a type of comparison theorem for the control component of the solution to the coupled FBSDE \eqref{eqmainF}. This is known as the comonotonicity theorem. It was first introduced in \cite{ChenKulWei05} in the Lipschitz framework and then extended in the quadratic framework in \cite{DosDos13} for decoupled FBSDEs. The comonotonicity theorem finds interesting applications in the context of economic models of equilibrium pricing in the framework of forward-backward SDEs (\cite{ChenKulWei05},\cite{DosDos13}). 
	
	Below, we recall the meaning of two functions being comonotonic.
	\begin{defi}
		Two functions $h_1$ and $h_2$ are said to be comonotonic, if both $h_1$ and $h_2$ are of the same monotonicity, that is, if $h_1$ is increasing(or decreasing), so is $h_2.$ Moreover, $h_1$ and $h_2$ are said to be strictly comonotonic if $h_1$ and $h_2$ are strictly monotonic.
	\end{defi}
	
	In this section, we assume that $d=1$ and we recall that from Theorem \ref{th3.3}, the equation \eqref{eqmainF} admits a unique strong solution $(X^x,Y^x,Z^x)$ such that the function $v\in W_{\text{loc}}^{1,2,2}([0,T[\times\mathbb{R},\mathbb{R})$ solves the quasi-linear PDE \eqref{eqmainP} and the following relation hold
	\begin{equation}\label{rep2}
		Y_s^{t,x} = v(s, X_s^{t,x}),\quad Z_s^{t,x} = v'(s, X_s^{t,x}),
	\end{equation}
	Moreover, the FBSDE \eqref{eqmainF} can be written as 
	\begin{equation}\label{eqmain6}
		\begin{cases}
			X_s^{t,x} = x + \displaystyle \int_{t}^{s} \tilde{b}(r,X_r^{t,x})\mathrm{d}r + \int_{t}^{s} \tilde{\sigma}(r,X_r^{t,x})\mathrm{d}W_r,\\
			Y_s^{t,x} = \phi(X_T^{t,x})+ \displaystyle \int_{s}^{T} f(r,X_r^{t,x},Y_r^{t,x},Z_r^{t,x})\mathrm{d}r - \int_{s}^{T} Z_r^{t,x} \mathrm{d}M_r^X,
		\end{cases}
	\end{equation}
	where, $\tilde{b}(\cdot,\cdot)= b(\cdot,\cdot,v(\cdot, \cdot), v'(\cdot,\cdot)),\quad \tilde{\sigma}(\cdot,\cdot)= \sigma(\cdot,\cdot,v(\cdot,\cdot)).$
	
	We are now in position to give the main result of this section. It can be viewed as an extension of those in \cite{ChenKulWei05, DosDos13} to the case of coupled FBSDEs under {weaker} assumptions on the coefficients and more general type of quadratic drivers. We will consider the following additional assumption:
	\begin{itemize}
		\item[(AX')] The function $(t,x,\cdot,\cdot)\mapsto b(t,x,\cdot,\cdot)$ is continuous on $[0,T]\times\mathbb{R}^d$ for all $y$ and $z$.
	\end{itemize}
	\begin{thm}\label{Main1}
		Let assumptions of Theorem \ref{th3.3} hold for $d=1$ and $\alpha_0 \geq 1/2$. Assume further that (AX') is valid. For all $(t,x) \in [0,T) \times \mathbb{R},$  let $(X^{t,x,i}, Y^{t,x,i}, Z^{t,x,i})$ be the solution to FBSDE \eqref{eqmainF}  with drift $b_i$, generators $f_i$, terminal value $\phi_i$ and the same dispersion coefficient $\sigma$, $i \in \{ 1,2\}.$ Suppose that $x\mapsto \phi_i(x)$ and $x\mapsto f_i(\cdot,x,\cdot,\cdot)$ are comonotonic for all $i \in \{ 1,2\}$. 
		Then $\text{for any } (t,x) \in [0,T) \times \mathbb{R} \text{ and } s \in [t, T)$
		\begin{equation}\label{rep4}
			{Z}_s^{t,x,1} \cdot {Z}_s^{t,x,2} \geq 0, \quad \mathbb{P}\text{-a.s.}
		\end{equation}
	\end{thm}
	
	The proof of the theorem is based on a careful application of the comparison theorem both for SDEs and quadratic BSDEs, since it is known that such results amount to a kind of monotonicity of the solutions $X$ and $Y$, respectively. One of the difficulties here, comes from the non-Lipschitz continuity of the drift $b$ and a comparison theorem result in this frame is still an open question in the literature. However, thanks to the additional assumptions (AX') and the continuity of the process $X$ in $T$(see Theorem \ref{th3.3}), we can invoke the refine comparison theorem for SDEs stated in [\cite{KaShr88}, Proposition 5.2.18],  to claim that, for fixed $t$ and $T$ the mapping $x\mapsto X_T^{t,x}$ is increasing.
	
	\begin{proof} Since $\phi_1$ and $\phi_2$ are comonotonic, thus for fixed $t$ and $T,$ the functions $\phi_1(X_T^{t,x,1})$ and $\phi_2(X_T^{t,x,2})$ are almost surely comonotonic. A same reasoning yields $x\mapsto f_1(\cdot,X_{\cdot}^{t,x,1},\cdot,\cdot)$ and $x\mapsto f_2(\cdot,X_{\cdot}^{t,x,2},\cdot,\cdot)$ are almost surely comonotonic.
		
		{Under the assumptions of the theorem, a comparison theorem for quadratic BSDEs is valid (see Theorem \ref{thap}). Combining this with} the monotonocity (and comonotonicity) of $x\mapsto \phi_i(X_T^{t,x,i})$ and $x\mapsto f_{i}(\cdot,X_T^{t,x,i},\cdot,\cdot)$ we obtain that $x\mapsto Y^{t,x,i}$ is almost surely monotone. From the comonotonicity of $x\mapsto \phi_1(X_T^{t,x,1}),$ $x\mapsto \phi_2(X_T^{t,x,2}),$ $x\mapsto f_1(\cdot,X^{t,x,1},\cdot,\cdot),$ and $x\mapsto f_2(\cdot,X^{t,x,2},\cdot,\cdot),$ we invoke once more Theorem \ref{thap} to conclude that the mappings $x\mapsto Y^{t,x,1}$ and $x\mapsto Y^{t,x,2}$ are also comonotonic a.s. Therefore from \eqref{rep2}, we deduce that 
		\begin{align*}
			Z_s^{t,x,1} \cdot Z_s^{t,x,2} &=(\partial_x v_1)(s,X_s^{t,x,1})\cdot(\partial_x v_2)(s,X_s^{t,x,2}) \geq 0, \,\, \mathbb{P-} a.s.
		\end{align*} 
		This conclude the proof.
	\end{proof}
	The next result is a refinement of the preceding one. {It is derived without the additional assumption (AX') (continuity of the drift $b$ in $t$ and $x$) and assuming further that the diffusion coefficient $\sigma=1$. In this context, the system of interest is given by 
		\begin{equation}
			\begin{cases}\label{mono}
				X_s^{t,x} = x + \displaystyle \int_{t}^{s} \tilde{b}(r,X_r^{t,x})\mathrm{d}r + W_s-W_t,\\
				Y_s^{t,x} = \phi(X_T^{t,x})+ \displaystyle \int_{s}^{T} f(r,X_r^{t,x},Y_r^{t,x},Z_r^{t,x})\mathrm{d}r - \int_{s}^{T} Z_r^{t,x} \mathrm{d}M_r^X,
			\end{cases}
		\end{equation}
		and the drift coefficient $\tilde b$ is now allowed to be discontinuous both in its temporal and spatial variables.
		\begin{thm}\label{comon}
			Let Assumption \ref{assum1} be in force for $d=1$. Define for all $(t,x) \in [0,T) \times \mathbb{R},$ the solution to FBSDE \eqref{mono} by $(X^{t,x,i}, Y^{t,x,i}, Z^{t,x,i})$ for $i \in \{ 1,2\}.$ Suppose that $x\mapsto \phi_i(x)$ and $x\mapsto f_i(\cdot,x,\cdot,\cdot)$ are comonotonic for all $i \in \{ 1,2\},$ and furthermore, that $\phi_1,\phi_2$ are also comonotonic.
			Then $\text{for any } (t,x) \in [0,T) \times \mathbb{R} \text{ and } s \in [t, T),$
			\begin{equation}\label{rep4}
				{Z}_s^{t,x,1} \cdot {Z}_s^{t,x,2} \geq 0, \quad \mathbb{P}\text{-a.s.}
			\end{equation}
		\end{thm}
		\begin{proof}
			Using the comparison theorem for SDEs provided in \cite[Theorem 2]{Nakao83}, {we can then apply the machinery developed in the proof of Theorem \ref{Main1}}. This ends the proof.
		\end{proof}
		The next result provides conditions under which one can conclude on the positivity or negativity of the control process $Z$ solution to a scalar coupled FBSDE with quadratic growth. The main ingredient in the proof of this result concerns the differentiability with respect to the starting point $x$ of the forward process. {In our frame, this is provided} in Section \ref{section5}, Proposition \ref{th4.4} . The complete proof of this corollary can be performed as in \cite[Corollary 3.5]{DosDos13}, we will not reproduce it here.
		\begin{cor}\label{Cor}
			Let Assumptions of Theorem \ref{comon} be in force. Let $(X^x,Y^x,Z^x)$ be the solution to FBSDE \eqref{mono} with initial value $(t,x)\in [0,T)\times\mathbb{R}$. Then, if $x\mapsto\phi(x)$ and $x\mapsto f(\cdot,x,\cdot,\cdot)$ are increasing (resp. decreasing) functions, then $Z_t \geq 0 \,\, \mathbb{P}\text{-a.s.}$(resp. $Z_t \leq 0 \,\, \mathbb{P}\text{-a.s.}$ ) for all $t \in [0,T)$.
		\end{cor}
			As pointed out in \cite[Remark 4.3]{Mastrolia}, a strict comonotonicity condition on the data in Corollary \ref{Cor}  is not sufficient to conclude that $Z_t > 0\,\,\, \mathbb{P} \text{-a.s.}$, hence $Y_t$ will have an absolute continuous density with respect to the Lebesgue measure from Bouleau-Hirsch's criterion (Theorem \ref{criterion}). Finding the right conditions to ensure the latter result will constitute the main objective in the following section.

		\section{Density analysis for quadratic FBSDEs with rough coefficients}\label{section5}
		In this section we analyse the density of the solution $\bold{X}^x_{\cdot} = (X^x_{\cdot},Y^x_{\cdot},Z^x_{\cdot})$ to the equation \eqref{eqmainF}. The notation $\mathcal{L}(X_T^x)$ will stand for the law of the process $X_T^x$. Sometimes, we will omit the superscript.
		
		\subsection{The H\"{o}lder continuous case}\label{Hcase}
		{
			In this subsection, we carry out the analysis of densities of the FBSDE \eqref{eqmainF}, by assuming further that the coefficients are H\"{o}lder continuous in their time and space variables respectively and the dispersion $\sigma$ will only depend on the forward process $X$(see equation \eqref{eqmain36} below). We are going to provide sufficient conditions under which the forward component $X$ and {the first component solution of the} Backward one $Y$ admit respectively an absolutely continuous law with respect to the Lebesgue measure, despite the roughness of the drift $b$, the driver $f$ and of the terminal value function $\phi$.}
		
		An important ingredient in this analysis relies on the It\^o-Tanaka trick, as developed in \cite{FlanGubiPrio10} and further extended in \cite{FlanGubiPrio} and \cite{WeiLvWang}. This method (originated from \cite{Zvon74}) consists in constructing a one to one transformation of a phase space that allows one to pass from an SDE with bad drift coefficient to a diffusion process with smooth coefficients.}
	
	{More precisely, we will consider the following system 
		\begin{equation}\label{eqmain36}
			\begin{cases}
				X_s^{t,x} = x + \displaystyle \int_{t}^{s} {b}(r,X_r^{t,x},Y_r^{t,x},\sigma^{-1}(r,X_r^{t,x})Z_r^{t,x})\mathrm{d}r + \int_{t}^{s} {\sigma}(r,X_r^{t,x})\mathrm{d}W_r,\\
				Y_s^{t,x} = \phi(X_T^{t,x})+ \displaystyle \int_{s}^{T} f(r,X_r^{t,x},Y_r^{t,x},\sigma^{-1}(r,X_r^{t,x})Z_r^{t,x})\mathrm{d}r - \int_{s}^{T} Z_r^{t,x} \mathrm{d}W_r,
			\end{cases}
		\end{equation}
		which is an equivalent version of \eqref{eqmainF} (via an obvious change of variables) and we will assume further that the coefficients $b,\sigma,\Phi$ and $f$ satisfy the the following set of assumptions:
		\begin{assum}\label{assum53} \leavevmode
			\begin{itemize}
				\item[{\bf(A1)}] Assumption \ref{assum1} is valid with $\sigma(t,x,y)\equiv \sigma(t,x)$. The coefficients $b,\sigma,f$ are H\"{o}lder continuous in $(t,x) \in [0,T]\times \mathbb{R}^d$ uniformly in $y$ and $z$. In particular, there exists a constant $\theta \in (0,1)$ such that $b(t,\cdot,y,z)\in C^{\theta}(\mathbb{R}^d;\mathbb{R})$, $(t,y,z) \in [0,T]\times\mathbb{R}\times \mathbb{R}^d$.
				\item[{\bf(A2)}] In addition, $\sigma(t,\cdot) \in C^3_b(\mathbb{R}^d,\mathbb{R})$, $t\in [0,T]$, $\sup_{t\in[0,T]}\|\sigma(t,\cdot)\|_{C^3_b(\mathbb{R},\mathbb{R})} < \infty$ and for all $(t,x)\in [0,T]\times\mathbb{R}$, $\|\sigma^{-2}\|_0:= \sup_{x\in \mathbb{R}^d,t\in [0,T]}|\sigma^{-2}(t,x)|< \infty$.
			\end{itemize}
		\end{assum}
			In the above context, the FBSDE \eqref{eqmain36} has a weak solution which is unique in probability law (see Theorem \ref{th3.3}) and the relation \eqref{FK} writes 
			\begin{align}\label{FK1}
				Y_s^{t,x} = v(s,X_s^{t,x});\qquad  Z_s^{t,x} = \sigma(t,X_s^{t,x})\nabla_x v(s,X_s^{t,x}),
			\end{align}
			where we will maintain $v$ as the unique solution to the PDE\eqref{eqmainP}\footnote{this holds with a slightly different second order operator $\mathcal{L}$, $f(t,x,y,\sigma^{-1} z)$ instead of $f(t,x,y,z)$ and $b(t,x,y,\sigma^{-1} z)$ in lieu of $b(t,x,y,z)$.}
		}

		{It is worth noting that, under the aforementioned assumptions, the weak solution $v$ to the PDE \eqref{eqmainP} turns out to be a classical one i.e. $v \in C^{1,2}([0,T]\times\mathbb{R}^d;\mathbb{R})$ (see Section 8 in \cite{Delarue2}). Moreover, there is a suitable constant $\gamma >0$ such that the following pointwise estimate of $\nabla_{xx}^2v(t,x)$ holds
			\begin{equation}\label{remk 5.2}
				\forall (t,x)\in [0,T[\times \mathbb{R}^d,\,	(T-t)^{1-\gamma}|\nabla_{x,x}^2v(t,x)| < \infty.
			\end{equation} 
		In addition, the FBSDE \eqref{eqmain36} can be written as follows
		\begin{equation}\label{eqmain37}
			\begin{cases}
				X_s^{t,x} = x + \displaystyle \int_{t}^{s} \tilde{b}(r,X_r^{t,x})\mathrm{d}r + \int_{t}^{s} {\sigma}(r,X_r^{t,x})\mathrm{d}W_r,\\
				Y_s^{t,x} = \phi(X_T^{t,x})+ \displaystyle \int_{s}^{T} f(r,X_r^{t,x},Y_r^{t,x},\sigma^{-1}(r,X_r^{t,x})Z_r^{t,x})\mathrm{d}r - \int_{s}^{T} Z_r^{t,x} \mathrm{d}W_r,
			\end{cases}
		\end{equation}
		where, $\tilde{b}(t,x)= b(t,x,v(t,x), \sigma^{-1}(t,x)\nabla_x v(t,x))$ and $\nabla_x v(t,x)$ denotes the gradient of the classical solution $v$ to the PDE \eqref{eqmainP}.
	}
	
	{The following lemma, provides a smoothness result gains by the transformed drift $\tilde b$ under the additional Assumption  \ref{assum53} . 
		\begin{lemm}
			Under Assumption \ref{assum53}, the transformed drift $\tilde b$ is $\theta'\text{-H\"{o}lder}$ continuous for all $t<T$.
		\end{lemm}
		\begin{proof}
			Let $(t,x,x',y,z)\in [0,T)\times\mathbb{R}^d\times\mathbb{R}^d\times\mathbb{R}\times \mathbb{R}^d$. Then from the assumptions of the Lemma we obtain that
			\begin{align*}
				&|\tilde{b}(t,x)-\tilde{b}(t,x')|\\
				&=|b(t,x,v(t,x),\sigma^{-1}(t,x)v'(t,x))-b(t,x',v(t,x'),\sigma^{-1}(t,x')\nabla_x v(t,x'))|\\
				&\leq \Lambda_{b}|x-x'|^{\theta} + K|v(t,x)-v(t,x')| + K|\sigma^{-1}(t,x)\nabla_x v(t,x)-\sigma^{-1}(t,x')\nabla_x v(t,x')|\\
				&\leq \Lambda_{b}|x-x'|^{\theta} + K|x-x'|^{\beta_0} + K|\sigma^{-1}(t,x)||x-x'|^{\beta'} + K|\nabla_x v(t,x')||\sigma^{-1}(t,x)-\sigma^{-1}(t,x')| \\
				&\leq 4\max(\Lambda_{b},K) |x-x'|^{\theta'},
			\end{align*}
			where we used \eqref{holder} and \eqref{holder1} to derive the third inequality above and $\theta'= \min(\theta,\beta_0,\beta').$
		\end{proof}
		\subsubsection{Strong Solvability, Pathwise uniqueness and Regularity of FBSDE \eqref{eqmain36}} From Assumption \ref{assum53}, we notice that the transformed drift $\tilde b$ is H\"older continuous and bounded and the diffusion coefficient $\sigma$ is Lipschitz continuous and bounded. Then from \cite[Theorem 4]{Zvon74} the forward equation in \eqref{eqmain37} has a unique strong solution. Hence, the solution of the FBSDE built in Theorem \ref{th3.3} is also strong for all $d\geq 1$. We further notice that, the transformed drift $\tilde b$ and the diffusion $\sigma$ fulfil the requirements of \cite[ Remark 10]{FlanGubiPrio10}. Then, from  \cite[Theorem 7]{FlanGubiPrio} (see also \cite[Theorem 5]{FlanGubiPrio10}), the forward equation in \eqref{eqmain37} admits a stochastic flow of diffeomorphisms of class $C^{1,\theta''}$ where $\theta'' \in (0,\theta')$. So, the variational differentiability in the sense of Malliavin of the solution process $X^{t,x}$ follows as well. Indeed, following the same strategy as in \cite{FlanGubiPrio10} and \cite{FlanGubiPrio}, we know that $X^{t,x}$ is linked to the process $\tilde X^{t,\zeta}$ solution of the following SDE
		\begin{align}\label{eq5.38}
			\tilde X_s^{t,\zeta} = \zeta + \displaystyle \int_{t}^{s} \tilde{b}_1(r,\tilde X_r^{t,\zeta})\mathrm{d}r + \int_{t}^{s} \tilde{\sigma}_1(r,\tilde X_r^{t,\zeta})\mathrm{d}W_r,
		\end{align}
		via the transformation
		\begin{equation}\label{eq5.6}
			X^{t,x}_s= \Psi^{-1}(s,\tilde X_s^{t,\zeta}),
		\end{equation}
		where $\Psi^{-1}$ denotes the inverse of the non-singular diffeomorphisms $\Psi$ of class $C^2$  defined by (see \cite[Lemma 6]{FlanGubiPrio10})
		\begin{equation}\label{eq5.7}
			\Psi(t,x) = x + U(t,x),
		\end{equation}
		where $U$ is the solution to the following Backward Kolmogorov equation 
		\begin{align}\label{2.8}
			\begin{cases}
				\displaystyle	\partial_t U(t,x) + \frac{1}{2}\Tr\left[(\sigma \sigma^{\bold{T}})(t,x) \Delta U(t,x)\right] + (\tilde b\cdot DU)(t,x)- \mu U(t,x) = -\tilde b(t,x), \\
				U(T,x) = 0, \quad (t,x) \in [0,T)\times \mathbb{R}^d.
			\end{cases}
		\end{align}
		with $\mu > 0$ is fixed. The drift coefficient $\tilde{b}_1(t,\zeta) = \mu U(t,\Psi^{-1}(t,\zeta)) $ belongs to $L^{\infty}(0,T;C_b^{2+\theta}(\mathbb{R};\mathbb{R}))$, the diffusion $\tilde{\sigma}_1(t,\zeta)= [I+D\Psi(t,\Psi^{-1}(t,\zeta))]\sigma(t,\Psi^{-1}(t,\zeta))$ obviously inherits some of the smoothness properties of $\sigma$ and $U$. Thus $\tilde b_1$ and $\tilde \sigma_1$ are both Lipschitz continuous, we can then invoke the result from \cite{Nua06} to claim that the SDE \eqref{eq5.38} is Malliavin differentiable and so is also $X^{t,x}$ via the above representation and a straightforward application of the chain rule for Malliavin calculus.} 
	
	We also notice that, the uniqueness of $\tilde{X}$ gives the pathwise uniqueness of solution $(X,Y,Z)$ to the equation \eqref{eqmain36}. Indeed, if $(\bar{X},\bar{Y},\bar{Z})$ is another weak solution to \eqref{eqmain36} defines on the same underlying stochastic basis $(\Omega,\mathfrak{F},(\mathfrak{F})_{t\leq s \leq T}, W,\mathbb{P})$, then $\bar{X}$ for instance, will give rise to the process $\underline{X}= \Psi(t,\bar{X})$ solution to \eqref{eq5.38}, then $\underline{X}=\tilde{X}$ and this implies that $X=\bar{X}$. On the other hand, thanks to the relation \eqref{FK1}, gives 
	\begin{align}
		\begin{cases}
			\bar{Y}_t = v(t,\bar{X}_t) = v(t,X_t) = Y_t \quad \forall t\in [0,T],\mathbb{P}\text{-a.s.}  \\
			\bar{Z}_t = \sigma(t,\bar{X}_t)\nabla_x v(t,\bar{X}_t) = \sigma(t,X_t) \nabla_xv(t,X_t) = Z_t \quad \mathrm{d}\mathbb{P}\otimes \mathrm{d}t\text{-a.e.}
		\end{cases}
	\end{align}
	
	{To conclude, thanks to the relation \eqref{FK1} again, and using the chain rules (classical and Malliavin calculus), we deduce that the couple solution $(Y,Z)$ is differentiable in the classical and Malliavin sense. Moreover, the following representations hold for all $0\leq s \leq t < T$
		\begin{align}
			D_sY_t = \nabla_x v(t,X_t)D_sX_t, \quad D_sZ_t = \Big(\sigma(t,X_t)\nabla_{x,x}^2 v(t,X_t) + \sigma_x(t,X_t)\nabla_x v(t,X_t)\Big)D_sX_t.
		\end{align}
		\begin{remark}
				Due to the lack of regularity of the drift $b$, the equation satisfied by the Malliavin derivative or the classical derivative with respect to the initial value for the triple solution $(X,Y,Z)$ still missing in the literature. However, in the one dimensional case, we can obtain the following for all $0\leq s \leq t < T$
				\begin{align}\label{Mal57}
					D_s X_t &= D\Psi^{-1}(t,\tilde X_t)\tilde \sigma_1(s,\tilde X_s)  \exp\left(\int_s^t\Big\{\tilde b_1'(r,\tilde X_r) -\frac{1}{2}(\tilde \sigma_1'(r,\tilde X_r))^2 \Big\}\mathrm{d}r + \int_s^t \tilde \sigma_1'(r,\tilde X_r)\mathrm{d}W_r\right)\notag\\
					&= D\Psi^{-1}(t,\tilde X_t)\tilde \sigma_1(s,\tilde X_s) \tilde \xi_t (\tilde \xi_s)^{-1}
				\end{align}
			\end{remark}
			\subsubsection{Existence of Density of the forward $X$} 
			\begin{thm}\label{th5.5}
				Let Assumption \ref{assum53} be in force. Then, the law of $X_t$ solution to FBSDE \eqref{eqmain36} has an absolutely continuous law with respect to the Lebesgue measure.
			\end{thm}
			\begin{proof}
				We first note that the diffusion $\tilde{a}:=\tilde\sigma_1\tilde\sigma_1^{\bold{T}}$ is non degenerate , i.e., there is a constant $C>0$ such that $\forall \xi \in \mathbb{R}^d$, $\langle \xi, \tilde a(t,\zeta)\xi \rangle \geq C |\xi|^2$
				where $\tilde\sigma_1(t,\zeta)= D\Psi(t,\Psi^{-1}(t,\zeta))\sigma(t,\Psi^{-1}(t,\zeta))$. Indeed, by choosing $\mu$ as in \cite[Lemma 4]{FlanGubiPrio10}, such that $|DU| \leq 1/2$ we obtain from Assumption \ref{assum1} and the triangle inequality that
				\begin{align*}
					\langle \xi, \tilde a(t,\zeta)\xi \rangle &= \langle \xi,\big( D\Psi(t,\Psi^{-1}(t,\zeta))\big)^2\sigma\sigma^{\bold{T}}(t,\Psi^{-1}(t,\zeta))\xi \rangle\\
					&\geq \lambda|\big( D\Psi(t,\Psi^{-1}(t,\zeta))\big)^2||\xi|^2\\
					&= \lambda|I_{d} + DU(t,\Psi^{-1}(t,\zeta))|^2|\xi|^2\\
					&\geq \lambda (1-| DU(t,\Psi^{-1}(t,\zeta))|)^2|\xi|^2\geq \lambda/2 |\xi|^2.
				\end{align*}
				Therefore, from \cite{Nua06}, the solution process $\tilde X^{t,\zeta}$ to equation \eqref{eq5.38} has a continuous law with respect to the Lebesgue measure that will be denoted by $\rho_{\tilde X_{t}(\zeta)}$. Hence, for every bounded and continuous function $\varphi: \mathbb{R}^d\rightarrow \mathbb{R}$ the following is well defined
				\begin{align*}
					\mathbb{E}\varphi(X^{x}_t) &:= \mathbb{E}\varphi\big(\Psi^{-1}(t,\tilde X^{\zeta}_t)\big)
					= \int_{\mathbb{R}^d} \varphi\big(\Psi^{-1}(t,z)\big)\cdot\rho_{\tilde X_t(\zeta)}(z) \mathrm{d}z\\
					&= \int_{\mathbb{R}^d} \varphi(z)\big( \det J\Psi(t,z) \big)\cdot\rho_{\tilde X_t(\zeta)}(\Psi(t,z)) \mathrm{d}z,
				\end{align*} 
				where $J$ stands for the Jacobian with respect to the spatial variable $x$. Hence, the process $X^{x}$ has a density $\rho_{X_{\cdot}(x)}$ with respect to the Lebesgue measure given by 
				\[ \rho_{X_{t}(x)} =  \big( \det J\Psi \big)\cdot\rho_{\tilde X_t(\zeta)}(\Psi).   \]
			\end{proof}
			\subsubsection{Density of the backward component $Y$} In this subsection, we provide some conditions under which the backward component $Y$ has an absolute continuous density with respect to the Lebesgue measure. To do this, we consider the following additional set of assumptions
			\begin{assum}\leavevmode\label{assum58}
				\begin{itemize}
					\item[{\bf(A3)}] Assumption \ref{assum53} is valid in for d=1 and $\beta =1$ i.e., the terminal value $\phi$ is Lipschitz continuous. There exist $\Lambda,\lambda >0$ such that for all $(t,x)\in [0,T]\times \mathbb{R}$, $\lambda \leq \sigma(t,x) \leq \Lambda$.
					\item[{\bf(A4)}] $(x,y,z)\mapsto f(t,x,y,z)$ is continuously differentiable for every $t\in [0,T]$ and for all $(t,x,y,z)\in [0,T]\times \mathbb{R}^3$ and $\alpha \in (0,1)$
					\begin{align*}
						|f_x'(t,x,y,z)| &\leq \Lambda (1+ |y| + \ell(|y|)|z|^{\alpha})\\
						|f_y'(t,x,y,z)| &\leq \Lambda (1+ |z|^{\alpha}) \\
						|f_z'(t,x,y,z)| &\leq \Lambda (1+\ell(y)|z|)
					\end{align*}
					\item[{\bf(A5)}] $\phi$ is differentiable $\mathcal{L}(X_T^x)\text{-a.e.}$
				\end{itemize}
			\end{assum}
			The method consists on proving that the Malliavin covariance matrix of $Y_t$ is invertible, i.e. for all $0\leq s \leq t < T$
			\[ \Gamma_{Y_t} = \langle D_sY_t, D_sY_t \rangle_{L^2([0,t])} > 0,\quad \mathbb{P}\text{-a.s.}  \]
			For this, one needs first to derive the linear equation (BSDE) satisfied by a version $(D_sY_t,D_sZ_t)$ of the Malliavin derivatives of the solution to equation \eqref{eqmain3}. 
			\begin{prop}
				Under Assumption \ref{assum58}, a version of $(D_sY_t,D_sZ_t)$ satisfies 
				\begin{align}\label{Mal58}
					D_sY_t &= 0, \quad D_sY_t = 0, \quad t< s\leq T,\notag\\			D_sY_t &= \phi'(X_T)D_sX_T+ \displaystyle \int_{t}^{T}  f_x'(r,\bold{X}_r)D_s{X}_r + f_y'(r,\bold{X}_r)D_s{Y}_r  \mathrm{d}r \notag\\
					&\qquad + \int_{t}^{T} \Big( Z_r\sigma^{-1}_x(r,X_r)D_s{X}_r + \sigma^{-1}(r,X_r)D_sZ_r \Big) f_z'(r,\bold{X}_r)  \mathrm{d}r- \int_{t}^{T} D_sZ_r \mathrm{d}W_r
				\end{align}
				here $\bold{X}_{\cdot} = (X_{\cdot},Y_{\cdot},\sigma^{-1}(\cdot,X_{\cdot})Z_{\cdot})$ and $D_sX^x_t$ is given by \eqref{Mal57}.
				Moreover, $(D_tY_t)_{0\leq t\leq T}$ is a continuous version of $(Z_t)_{0\leq t\leq T}$. 
			\end{prop}
			\begin{proof}  For the representation, we appeal the result from \cite[Theorem 5.12]{ImkRhossOliv}\footnote{Note that the result is still valid with non constant $\sigma$ satisfying the conditions (A1) and (A2)}, since $\tilde{b}$ is uniformly bounded and H\"{o}lder continuous for all $t\in [0,T)$. {The proof is completed}.
			\end{proof}
			In the sequel, we will adopt the following notations as in \cite{F.Antonelli} (see also \cite{Mastrolia} ):  For any $\mathcal{A}\in \mathcal{B}(\mathbb{R}),$ and $t\in [0,T]$ such that $\mathbb{P}(X_T \in \mathcal{A}|\mathfrak{F}_t) > 0:$
			\begin{align}
				\begin{cases}
					\overline{\phi}:= \sup_{x\in \mathbb{R}}\phi'(x),\qquad \overline{\phi}^{\mathcal{A}}:= \sup_{x\in \mathcal{A}}\phi'(x),&\\
					\underline{\phi}:= \inf_{x\in \mathbb{R}}\phi'(x),\qquad \underline{\phi}^{\mathcal{A}}:= \inf_{x\in \mathcal{A}}\phi'(x),&
				\end{cases}
			\end{align}
			\begin{align}
				\overline{f}(t):= \sup_{s\in[t,T],(x,y,z)\in \mathbb{R}^3}f_x'(s,x,y,z),\qquad
				\underline{f}(t):= \inf_{s\in[t,T],(x,y,z)\in \mathbb{R}^3}f_x'(s,x,y,z).
			\end{align}
			Below is stated the main result of this subsection.
			\begin{thm}\label{th5.9}
				Suppose Assumption \ref{assum58} holds. Suppose in addition there is $\mathcal{A}\in \mathcal{B}(\mathbb{R})$ such that $\mathbb{P}(X_T \in \mathcal{A}|\mathfrak{F}_t) > 0$ and one of the following assumptions holds 
				\begin{itemize}
					\item[(A+)] $\phi' \geq 0,$ $\phi'_{|\mathcal{A}}>0, \mathcal{L}(X_T)\text{-a.e.}$ and $\underline{f}\geq 0$, $\sigma_x^{-1}f_z'\geq 0$,
					\item[(A--)] $\phi' \leq 0,$ $\phi'_{|\mathcal{A}}<0, \mathcal{L}(X_T)\text{-a.e.}$ and $\overline{f} \leq 0$, $\sigma_x^{-1}f_z'\leq 0 $.
				\end{itemize}
				Then, $Y_t$ possesses an absolute continuous law with respect to the Lebesgue measure on $\mathbb{R}$.
		\end{thm}}
		{\begin{proof}
				From Assumption \ref{assum58}, we deduce that
				\begin{align}\label{eq511}
					|\sigma^{-1}(s,X_s)f_z'(s,\bold{X}_s)| \leq C \big(1+ \ell(Y_s)|Z_s|\big) \in \mathcal{H}_{\text{BMO}}
				\end{align}
				{Let $\mathbb{\tilde P}$ be the probability measure defined by} 
				\begin{align}
					\frac{\mathrm{d}\mathbb{\tilde P}}{\mathrm{d}\mathbb{P}}\Big|_ {\mathfrak{F}_t} = \mathcal{E}\Big( \sigma^{-1}(t,X_t)f_z'(t,\bold{X}_t)* W\Big)_t= \mathcal{E}_t.
				\end{align}
				{It follows from \eqref{eq511} that $\mathcal{E}_t$} is uniformly integrable. Therefore, {taking the conditioning expectation with respect to $\mathbb{\tilde P}$ on both sides of \eqref{Mal58} gives for all $0\leq s \leq t < T$}
				\begin{align}
					D_s Y_t=& \mathbb{E}^{\mathbb{\tilde P}} \Big( \phi'(X_T)D_sX_T \notag\\
					&+ \int_t^T \left[ \Big( f_x'(r,\bold{X}_r) + Z_r\sigma^{-1}_x(r,X_r) f_z'(r,\bold{X}_r) \Big)D_s X_r + f_y'(r,\bold{X}_r)D_s Y_r \right] \mathrm{d}r    \Big| \mathfrak{F}_t\Big)
				\end{align}
				{Using} \eqref{Mal57} and the well known linearisation method, $D_sY_t$ can be rewritten:
				\begin{align*}
					&D_s Y_t\\
					=& \mathbb{E}^{\mathbb{\tilde P}} \Big[ e^{\int_{t}^{T}f_y'(u,\bold{X}_u)\mathrm{d}u}\phi'(X_T) D\Psi^{-1}(T,\tilde X_T)\tilde \sigma_1(s,\tilde X_s) \tilde \xi_T (\tilde \xi_s)^{-1} \\
					& + \int_t^T e^{\int_{t}^{r}f_y'(u,\bold{X}_u)\mathrm{d}u}  \Big( f_x'(r,\bold{X}_r) + Z_r\sigma^{-1}_x(r,X_r) f_z'(r,\bold{X}_r) \Big) D\Psi^{-1}(r,\tilde X_r)\tilde \sigma_1(s,\tilde X_s) \tilde\xi_r (\tilde \xi_s)^{-1}\mathrm{d}r    \Big| \mathfrak{F}_t \Big]\\
					=& \mathbb{E} \Big[ \tilde \psi_T \phi'(X_T)\tilde\xi_T\\ 
					&+ \int_t^T \tilde \psi_r  \Big( f_x'(r,\bold{X}_r) + Z_r\sigma^{-1}_x(r,X_r) f_z'(r,\bold{X}_r) \Big) D\Psi^{-1}(r,\tilde X_r) \tilde \xi_r \mathrm{d}r \Big| \mathfrak{F}_t  \Big] (\tilde \psi_t)^{-1}(\tilde \xi_s)^{-1}\tilde \sigma_1(s,\tilde X_s),
				\end{align*}
				where the last equality is due to Bayes' rule with $\tilde \psi$ given by	
				\[ \tilde \psi_t := \exp\Big(  \int_{0}^{t}(f_y'(u,\bold{X}_u)-\frac{1}{2}|\sigma^{-1}(u,X_u)f_z'(u,\bold{X}_u)|^2)\mathrm{d}u+ \int_{0}^{t}\sigma^{-1}(u,X_u)f_z'(u,\bold{X}_u)\mathrm{d}W_u   \Big),\]
				and the expression of $\tilde\xi$ is given by \eqref{Mal58}.
				Therefore, from the above computations, the Malliavin covariance $\Gamma_{Y_t}$ of $Y_t$ is given by 
				\begin{align*}
					\Gamma_{Y_t}&= (\tilde \psi_t^{-1})^2\int_{0}^{t} (\tilde \sigma_1(s,\tilde X_s)\tilde\xi_s^{-1})^2\mathrm{d}s\\
					& \Big(\mathbb{E}\Big[ \tilde \psi_T  \phi'(X_T)\tilde \xi_T+ \int_{t}^{T}\tilde \psi_r \Big( f_x'(r,\bold{X}_r) + Z_r\sigma^{-1}_x(r,X_r) f_z'(r,\bold{X}_r) \Big) D\Psi^{-1}(r,\tilde X_r)\tilde\xi_r \mathrm{d}r \Big|\mathfrak{F}_t \Big]\Big)^2 .
				\end{align*} 
				To obtain the desired result, we show  under assumption (A+) that 
				$$\mathbb{E}\Big[ \tilde \psi_T  \phi'(X_T)\tilde \xi_T+ \int_{t}^{T}\tilde \psi_r \Big( f_x'(r,\bold{X}_r) + Z_r\sigma^{-1}_x(r,X_r) f_z'(r,\bold{X}_r) \Big) D\Psi^{-1}(r,\tilde X_r)\tilde\xi_r \mathrm{d}r \Big|\mathfrak{F}_t \Big]\neq 0.
				$$
				Using It\^{o}'s formula, the product $\tilde \psi_t \tilde\xi_t $ can be rewritten as
				\begin{align*}
					\tilde \psi_t \tilde\xi_t &= \exp\Big\{\int_0^{t}\big[ \tilde b'_1(u,\tilde X_u)\mathrm{d}u + f_y'(u,\bold{X}_u) + \tilde \sigma'_1(u,\tilde X_u)\sigma^{-1}(u,X_u)f_z'(u,\bold{X}_u) \big]\mathrm{d}u \Big\}\\
					&\quad \times \exp\Big\{ \int_{0}^{t} \Big(\tilde \sigma'_1(u,\tilde X_u)+\sigma^{-1}(u,X_u)f_z'(u,\bold{X}_u)\Big)\mathrm{d}W_u\\ &\qquad\quad\quad-\frac{1}{2} \int_{0}^{t}\Big( \tilde \sigma'_1(u,\tilde X_u)+\sigma^{-1}(u,X_u)f_z'(u,\bold{X}_u)\Big)^2\mathrm{d}u  \Big\} = \tilde H_t\times \tilde M_t
				\end{align*}
				From the assumptions of the theorem, the stochastic integral $\int_{0}^{T} \Big(\tilde \sigma'_1(u,\tilde X_u)+\sigma^{-1}(u,X_u)f_z'(u,\bold{X}_u)\Big)\mathrm{d}W_u$ is a BMO martingale. Hence, the measure $\mathbb{Q}$ defined by $\frac{\mathrm{d}\mathbb{ Q}}{\mathrm{d}\mathbb{P}}\Big|_ {\mathfrak{F}_t}= \tilde M_t$ is an equivalent probability measure to $\mathbb{P}$. Therefore, using the the Bayes' rule once more we obtain that
				\begin{align*}
					&\mathbb{E}\Big[ \tilde \psi_T  \phi'(X_T)\tilde \xi_T+ \int_{t}^{T}\psi_r \Big( f_x'(r,\bold{X}_r) + Z_r\sigma^{-1}_x(r,X_r) f_z'(r,\bold{X}_r) \Big) D\Psi^{-1}(r,\tilde X_r)\tilde\xi_r \mathrm{d}r \Big|\mathfrak{F}_t \Big]\\
					=&\mathbb{E}^{\mathbb{Q}}\Big[ \phi'(X_T)\tilde H_T+ \int_{t}^{T} \Big( f_x'(r,\bold{X}_r) + Z_r\sigma^{-1}_x(r,X_r) f_z'(r,\bold{X}_r) \Big) D\Psi^{-1}(r,\tilde X_r)\tilde H_r \mathrm{d}r \Big|\mathfrak{F}_t \Big]\tilde M_t.
				\end{align*}
				Let us remark that $D\Psi^{-1}(t,\xi)= 1 + \sum_{k\geq 1}(-DU(t,\Psi^{-1}\xi))^k \geq I + \sum_{k\geq 1}(-1/2)^k \geq 0 $ for all $\xi \in \mathbb{R}$. On the other hand, we also notice from the assumptions of the theorem and thanks to \cite[Corollary 3.5]{DosDos13} that $Z_r\tilde \sigma_1(r,\tilde X_r) \geq 0$, where $\tilde X$ stands for the solution to equation \eqref{eq5.38}. Hence,
				\begin{align*}
					&\mathbb{E}\Big[ \tilde \psi_T  \phi'(X_T)\tilde \xi_T+ \int_{t}^{T}\psi_r \Big( f_x'(r,\bold{X}_r) + Z_r\sigma^{-1}_x(r,X_r) f_z'(r,\bold{X}_r) \Big) D\Psi^{-1}(r,\tilde X_r)\tilde\xi_r \mathrm{d}r \Big|\mathfrak{F}_t \Big]\\
					&\geq \mathbb{E}^{\mathbb{Q}}\Big[{1}_{\{X_T \in \mathcal{A}\}}\Big( \underline{\phi} \tilde H_T + \int_{t}^{T} \Big(\underline{f}+ Z_r\tilde \sigma_1(r,\tilde X_r) \tilde \sigma_1^{-1}(r,\tilde X_r)\sigma^{-1}_x(r,X_r) f_z'(r,\bold{X}_r)\Big)D\Psi^{-1}(r,\tilde X_r)\tilde H_r \mathrm{d}r  \Big)\Big|\mathfrak{F}_t \Big]\tilde M_t\\
					&\geq \mathbb{E}^{\mathbb{Q}}\Big[{1}_{\{X_T \in \mathcal{A}\}}\Big(e^{-KT} \underline{\phi}^{\mathcal{A}} e^{\int_{0}^{T} f_y'(u,\bold{X}_u)\mathrm{d}u -K \int_{0}^{T}f_z'(u,\bold{X}_u)\mathrm{d}u}\\ 
					&\quad + \int_{t}^{T} \Big(\underline{f}+ Z_r\tilde \sigma_1(r,\tilde X_r) \tilde \sigma_1^{-1}(r,\tilde X_r)\sigma^{-1}_x(r,X_r) f_z'(r,\bold{X}_r)\Big)D\Psi^{-1}(r,\tilde X_r)\tilde H_r \mathrm{d}r  \Big)\Big|\mathfrak{F}_t \Big]\tilde M_t\\
					&\geq \mathbb{E}^{\mathbb{Q}}\Big[{1}_{\{X_T \in \mathcal{A}\}}\Big(e^{-KT} \underline{\phi}^{\mathcal{A}} e^{-\Lambda\int_{0}^{T}(2+|Z_u|^2)\mathrm{d}u} e^{-K\sqrt{T} \sqrt{\int_{0}^{T}|f_z'(u,\bold{X}_u)|^2\mathrm{d}u}}\\ 
					&\quad + \int_{t}^{T} \Big(\underline{f}+ Z_r\tilde \sigma_1(r,\tilde X_r) \tilde \sigma_1^{-1}(r,\tilde X_r)\sigma^{-1}_x(r,X_r) f_z'(r,\bold{X}_r)\Big)D\Psi^{-1}(r,\tilde X_r)\tilde H_r \mathrm{d}r  \Big)\Big|\mathfrak{F}_t \Big]\tilde M_t
				\end{align*}
				where we used the bound $|f_y'(s,\bold{X}_s)|\leq \Lambda (1+ |Z_s|^{\alpha})$, the Cauchy-Schwarz inequality and the fact that $\tilde \sigma_1(r,\tilde X_r) \geq 0$ (from its definition and the fact that $D\Psi >0$).
				Provided that the law of $X_T$ is absolutely continuous with respect to the Lebesgue measure, we deduce that $$ \mathbb{E}\Big[ \tilde \psi_T  \phi'(X_T)\tilde \xi_T+ \int_{t}^{T}\psi_r \Big( f_x'(r,\bold{X}_r) + Z_r\sigma^{-1}_x(r,X_r) f_z'(r,\bold{X}_r) \Big) D\Psi^{-1}(r,\tilde X_r)\tilde\xi_r \mathrm{d}r \Big|\mathfrak{F}_t \Big] >0.$$ This concludes the proof.
			\end{proof}
			\begin{remark}\leavevmode
				\begin{itemize}
					\item Even though we consider a fully coupled FBSDE under much weaker conditions, the strict monotonicity of the terminal condition and the driver on a particular Borel set, guarantee the existence of a density for the backward component solution $Y$ of the equation (compare with \cite{Mastrolia} for the case of decoupled FBSDE under Cauchy-Lipschitz drift). 
					\item  Note that, the extra assumption $\sigma_x^{-1}f_z' \geq 0$ or $\sigma_x^{-1}f_z' \leq 0$ here follows from of the special structure of the backward equation considered. These assumptions naturally disappeared when one considers FBSDE with the same structure as for instance in \cite{F.Antonelli}, \cite{Mastrolia}.
					\item A similar result was obtained in \cite{Olivera} for coupled FBSDEs with non constant diffusion coefficient but under rather strong  smoothness requirements on the coefficients. For instance, the drift is assumed H\"older continuous in $t$, differentiable and uniformly Lipschitz continuous in $x$ with derivative in $x$ being bounded and H\"older continuous in all its components $t,x,y$ and $z$. 
				\end{itemize}
			\end{remark}
		}
		
		{\subsubsection{Gaussian-type bounds of densities for FBSDE \eqref{eqmain36}} Here, we will provide upper and lower bounds for densities of the forward $X$ and the backward $Y$ components solution to \eqref{eqmain36}, respectively. We will further assume that 
			\begin{itemize}
				\item[{\bf(A6)}] Assumption \ref{assum58} holds and the drift $b:= b(t,x,y,z)$ is weakly differentiable in $x$ such that $b_x'(t,\cdot,y,z) \in L^{\infty}(\mathbb{R})$ for all $t,y$ and $z$, $\partial_t \sigma$ exists and it is uniformly bounded.
		\end{itemize}}

		{\begin{lemm}\label{bounds}
				Let assumption {\bf(A6)} be in force. Then for all $(t,x)\in [0,T)\times \mathbb{R}$ the {transformed drift $\tilde{b}$} is bounded and Lipschitz continuous in $x$. 
			\end{lemm}	
			\begin{proof}
				Let $(t,x,x')\in [0,T)\times \mathbb{R}\times\mathbb{R}$ we have:
				\begin{align*}
					|\tilde{b}(t,x)-\tilde{b}(t,x')| &= |{b}(t,x,v(t,x),v'(t,x))-{b}(t,x',v(t,x'),v'(t,x'))| \\
					&\leq |b_x'| |x-x'| + K\sup_{x}|v'(t,x)| |x-x'| + K \sup_{x}|v''(t,x)| |x-x'|\\
					&\leq \Lambda_{\tilde b}|x-x'|,
				\end{align*}
				where, we used the mean value theorem, the bounds of $\nabla_x b \in L^{\infty}(\mathbb{R})$, $v'$ (see \eqref{3.7}), $v''$ (see \eqref{remk 5.2}) and $\Lambda_{\tilde {b}}:= \max\{ \|b_x'\|_{\infty}, K\|v'\|_{\infty},K\|v''\|_{\infty} \}$.
		\end{proof}}
		{\begin{remark}\label{rmk 5.9}
				Under the Assumption {\bf(A6)}, the FBSDE \eqref{eqmain36} has a unique strong classical and Malliavin differentiable solution such that from \cite[Theorem 2.1]{DGR05}, the Malliavin derivative of the forward $X_t$ is given explicitly for al $0\leq s \leq t$ by
				\begin{equation}\label{517}
					D_sX_t = \sigma(s,X_s)\exp\Bigg[ \int_s^t \Big( \partial_x\tilde b - \frac{\tilde b\partial_x \sigma + \partial_t \sigma}{\sigma}-\frac{1}{2}(\partial_{xx}^2\sigma)\sigma \Big)(r,X_r)\mathrm{d}r \Bigg].
				\end{equation}
				Moreover, for all $p \geq 1$
				\begin{align*}
					\mathbb{E}\Big[ \sup_{0\leq s\leq t \leq T} \big(|\nabla_x X_t|^p + |(\nabla_x X_t)^{-1}|^p + |D_s X_t|^{p} \big)\Big]< \infty.
				\end{align*}
				In particular, using the well-known representation $\nabla_x Y_t \nabla_x X_s = D_s Y_t$ for all $0< s\leq t$, one can obtain  for all $p \geq 1$
				\[ \mathbb{E}\big[ \sup_{0\leq t \leq T} |Z_t|^p \big]< \infty.\]
		\end{remark}}
		
		{\begin{thm}\label{prop 4.10}
				Let assumption {\bf(A6)} and one of the assumptions (A+) or (A--) be in force. If there is $\mathcal{A}\in \mathcal{B}(\mathbb{R})$ such that $\mathbb{P}(X_T \in \mathcal{A}|\mathfrak{F}_t) > 0$.  Then:
				\begin{itemize}
					\item[(i)] The probability density and the tail probabilities of the {forward component} $X_t$ solution to FBSDE \eqref{eqmain36} satisfy for all $t \in [0,T)$ and $x>0$
					\begin{align}\label{5.15}
						\frac{\mathbb{E}|X_t-\mathbb{E}[X_t]|}{2C(\Lambda)t e^{2Kt}}\exp\Big( -\frac{(x-\mathbb{E}[X_t])^2}{2C(\Lambda,\lambda)t e^{-2Kt}}\Big) \leq \rho_{X_t}(x) \leq 	\frac{\mathbb{E}|X_t-\mathbb{E}[X_t]|}{2C(\Lambda,\lambda)t e^{-2Kt}}\exp\Big(- \frac{(x-\mathbb{E}[X_t])^2}{2C(\Lambda)t e^{2Kt}}  \Big),
					\end{align}
					and 
					\begin{align}\label{tail1}
						\mathbb{P}(X_t\geq x)\leq \exp\Big( -\frac{(x-\mathbb{E}[X_t])^2}{2C(\Lambda)t e^{2Kt}} \Big)\text{ and } 	\mathbb{P}(X\leq -x)\leq \exp\left( -\frac{(x+\mathbb{E}[X_t])^2}{2C(\Lambda,\lambda)t e^{2Kt}}\right),
					\end{align}
					respectively.
					\item[(ii)] The density of the process $Y_t$ solution to FBSDE \eqref{eqmain36}, satisfy for all $t \in [0,T)$ and $y>0$
					\begin{align}\label{5.17}
						\frac{\mathbb{E}|Y_t-\mathbb{E}[Y_t]|}{2t \left(\Upsilon e^{Kt}\right)^2}\exp\Big( -\frac{(x-\mathbb{E}[Y_t])^2}{2t \left( \alpha(t)e^{-Kt}\right)^2}\Big) \leq \rho_{Y_t}(y) \leq 	\frac{\mathbb{E}|Y_t-\mathbb{E}[Y_t]|}{2t \left( \alpha(t)e^{-Kt}\right)^2}\exp\Big(- \frac{(y-\mathbb{E}[Y_t])^2}{2t \left(\Upsilon e^{Kt}\right)^2}  \Big).
					\end{align}
					Furthermore for all $y>0$ the tail probabilities satisfy
					\begin{align}\label{tail2}
						\mathbb{P}(Y_t\geq y)\leq \exp\Big( -\frac{(y-\mathbb{E}[Y_t])^2}{2C(\Lambda)t \left(\Upsilon e^{Kt}\right)^2} \Big)\text{ and } 	\mathbb{P}(Y_t\leq -y)\leq \exp\Big( -\frac{(y+\mathbb{E}[Y_t])^2}{2C(\Lambda)t \left(\Upsilon e^{Kt}\right)^2}\Big),
					\end{align}
					where the constant $\Upsilon$ can be derived from \eqref{3.6}.
				\end{itemize}
			\end{thm}
			\begin{proof}It is {enough} to establish {a similar bound as \eqref{4.8} for each of the processes $X_t$ and $Y_t$ for all $t\in [0,T)$}. We will start with the forward process $X$. Under the assumption of the theorem, $X_t$ is Malliavin differentiable and thanks to Lemma \ref{bounds}, we deduce that there is a constant $K:= K({\Lambda_{\tilde b}},\|\partial_x \sigma\|_{\infty},\|\partial_t \sigma\|_{\infty},\|\partial_{xx}^2 \sigma\|_{\infty})$
				\begin{align}\label{4.11}
					0< \frac{\lambda}{\Lambda}e^{-Kt} \leq D_sX_t \leq \Lambda e^{Kt}, \text{ for all } s\leq t.
				\end{align}
				Therefore, 
				\begin{align*}
					0< C(\lambda,\Lambda)te^{-2Kt} \leq \int_0^t D_sX_t\mathbb{E}[D_sX_t|\mathfrak{F}_s]\mathrm{d}s \leq C(\Lambda)t e^{2Kt}.
				\end{align*}
				This implies that the probability density and the tail probability of $X_t$ satisfy \eqref{5.15} and \eqref{tail} respectively.
				Let us turn now to the backward process $Y$, and we recall that for all $t\in [0,T]$, $Y_t = v(t, X_t)$, where $v$ still denotes the classical solution to PDE \eqref{eqmainP}. Under the assumptions of the theorem, we deduce that the process $Y_t$ has an absolute continuous law with respect to the Lebesgue measure (see Theorem \ref{th5.9}). Then for all $0\leq s\leq t < T$
				\begin{align*}
					\Gamma_{Y_t} = \langle D_sY_t, D_sY_t\rangle_{L^2([0,t])} = (v'(t,X_t))^2\langle D_sX_t, D_sX_t\rangle_{L^2([0,t])} > 0\quad \mathbb{P}\text{-a.s.},
				\end{align*}
				which imply that $(v'(t,X_t))^2 > 0$ and $\langle D_sX_t, D_sX_t\rangle_{L^2([0,t])} > 0\quad \mathbb{P}\text{-a.s.}$ But the latter is always true, since all the terms inside the exponential in \eqref{517} are finite and $\sigma$ is bounded from below by $\lambda >0$.
				Then, there is a function $\alpha(t)>0$ such that the following holds
				$$v'(t,X_t) \geq \alpha(t) \text{ or } {v'(t,X_t) \leq -\alpha(t)} \quad \mathbb{P}\text{-a.s.} \text{ for all } t\in [0,T).$$ {First let us assume that $v'(t,X_t) \geq \alpha(t) \, \mathbb{P}\text{-a.s.}$} From  \eqref{4.11} we deduce that
				\begin{align}\label{4.12}
					0< \frac{\lambda}{\Lambda}\alpha(t)e^{-Kt} \leq D_sY_t = v'(t,X_t)D_sX_t \leq \Lambda \Upsilon e^{Kt},
				\end{align} 
				where the constant $\Upsilon\geq 0$ is an upper bound of $v'$ that can be derived from \eqref{3.6}. Therefore,
				\begin{align}\label{5.21}
					0< C(\lambda,\Lambda)t \left( \alpha(t)e^{-Kt}\right)^2 \leq \int_0^t D_sY_t \mathbb{E}[D_sY_t|\mathfrak{F}_s]\mathrm{d}s\leq C(\Lambda)t \left(\Upsilon e^{Kt}\right)^2.
				\end{align}
				{On the other hand, let us assume now $v'(t,X_t) \leq -\alpha(t) \quad \mathbb{P}\text{-a.s.}$. Using the bound \eqref{4.11} once more, we deduce that
					\begin{align}\label{5.22}
						0>- \alpha(t)e^{-Kt} \geq D_sY_t = v'(t,X_t)D_sX_t \geq v'(t,X_t) e^{Kt}.
					\end{align}
					From the linearity of the expected value, we have
					\begin{align}\label{5.23}
						0>- \alpha(t)e^{-Kt} \geq \mathbb{E}[ D_sY_t |\mathfrak{F}_s] \geq v'(t,X_t) e^{Kt}.
					\end{align}
					Combining \eqref{5.22} and \eqref{5.23} give \eqref{5.21}. Then, one can show that the density and the tail probability of the backward component $Y_t$ satisfy the bounds \eqref{5.17} and \eqref{tail2}.
				}
				The proof is completed. 
		\end{proof}}
		
		\subsection{Existence of density for $Z$ with weakly differentiable drift} 
		The major aim in this part is to establish a result on the existence of a density for the control process $Z$ under minimal regularity assumptions on the drift. The first conclusion in this approach was achieved in \cite[Theorem 4.2]{Mastrolia} for decoupled FBSDEs with quadratic drivers (assuming the correct smoothness constraints on the drift and diffusion of the forward equation as given for instance in \cite{ImkDos}). In the context of coupled FBSDEs, this result was obtained in \cite{Olivera} by assuming that the drift is twice continuously differentiable with bounded second derivatives in all spatial variables with non trivial diffusive coefficient. We claim that, the result obtained in \cite[Theorem 3.6]{Olivera} is quite optimal in the sense that it will be difficult to obtain a refinement of this result when the diffusion coefficient is non constant.
		
		We consider in particular the following system
		\begin{equation}\label{eqmain3}
			\begin{cases}
				X_t^x = x + \displaystyle \int_{s}^{t} b(r,X_r^x,Y_r^x,Z_r^x)\mathrm{d}r + W_t-W_s ,\\
				Y_t^x = \phi(X_T^x)+ \displaystyle \int_{t}^{T} f(r,X_r^x,Y_r^x,Z_r^x)\mathrm{d}r - \int_{t}^{T} Z_r^x\mathrm{d}W_r,
			\end{cases}
		\end{equation}
		and we assume the drift $b$ is weakly differentiable in its forward component $x$ such that the weak derivative $b_x'(t,\cdot,y,z)$ is bounded uniformly in $t,y$ and $z$. Note that, this case is not covered by \cite[Theorem 3.6]{Olivera}.
		
		{In this context, the associated PDE \eqref{eqmainP} (which remain a quasi-linear parabolic PDE) takes the particular form
			\begin{equation}\label{eqmainP1}
				\frac{\partial v}{\partial t}+ \frac{1}{2}\sum_{i,j}^{d}\frac{\partial^2 v}{\partial x_i \partial x_j} + \sum_{i}^{d} b_{i}(t,x,v,\nabla_x v)\frac{\partial v}{\partial x_i} + f(t,x,v,\nabla_x v)=0,
			\end{equation}
			where $v:= v(t,x)$ and $v(T,x)= \phi(x)$ for all $(t,x)\in [0,T]\times\mathbb{R}^d$.}
		
		We will assume further that:
		\begin{assum}\label{assum3} \leavevmode
			\begin{itemize}
				\item[{\bf(A7)}] $(x,y,z)\mapsto f(t,x,y,z)$ is twice continuously differentiable for every $t\in [0,T]$. There exists an adapted process $(\mathcal{K}_t)_{0\leq t\leq T}$ belonging to $\mathcal{S}^{2p}(\mathbb{R})$ for all $p\geq 1$ such that for any $t\in[0,T]$ all second order derivatives of $f$ at $(t,\bold{X}_t)=(t,X_t,Y_t,Z_t)$ are bounded by $\mathcal{K}_t$ a.s.
				\item[{\bf(A8)}] $\phi$ is  twice continuously differentiable $L(X_T^x)\text{-a.e.}$ and $\phi''$ has a polynomial growth.
			\end{itemize}
		\end{assum}
		We will start to establish a result on the second {order} Malliavin differentiability of coupled FBSDEs with quadratic {type} drivers. The setup proposed here covers the one in \cite{ImkDos}, since the coefficients satisfy weaker assumptions.
		\begin{thm}\label{secondMal}
			Let assumptions {\bf(A6)},{\bf(A7)} and {\bf(A8)} hold. Then, the solution process $\bold{X} =(X,Y,Z)$ of the FBSDE \eqref{eqmain3} is twice Malliavin differentiable, i.e., for each $s,t \in[0,T)$ the processes $(D_sX_t,D_sY_t,D_sZ_t)\in \mathbb{D}^{2,p}\times \mathbb{L}_{1,2}\times\mathbb{L}_{1,2}$, for all $p\geq1$. Moreover, a version of $(D_{s'}D_sY_t,D_{s'}D_sZ_t)_{0\leq s'\leq s\leq t\leq T}$ satisfies:
			\begin{align}
				D_{s'}D_sY_t =& D_{s'}D_s\xi - \int_t^T D_{s'}D_sZ_r\mathrm{d}W_r \notag\\
				&+ \int_t^T \left( [D_{s'}\bold{X}_r]^{{\bf T}}[Hf](r,\bold{X}_r)D_s\bold{X}_r + \langle \nabla f(r,\bold{X}_r), D_{s'}D_s \bold{X}_r \rangle \right)\mathrm{d}r, \label{secondY}
			\end{align}
			where $A^{{\bf T}}$ is the transpose of the matrix $A$, the symbol $[Hf]$ denotes the Hessian matrix of the function $f$, $\xi = \phi(X_T)$, $D_{s'}D_s \bold{X}_r=(D_{s'}D_sX_r,D_{s'}D_sY_r,D_{s'}D_sZ_r)$ and $D_{s'}D_sX_r$ stands for the second order Malliavin derivative of the forward process $X$.
			
			Moreover, $(D_tD_sY_t)_{0\leq s\leq t\leq T}$ is a version of $(D_sZ_t)_{0\leq s\leq t\leq T}$.
		\end{thm}
		
		\begin{remark} {From {\bf(A6)} and Lemma \ref{bounds}, the transformed drift $\tilde b$ is bounded and Lipschitz continuous, thus satisfies the linear growth condition with bounded weak derivative. Hence, the second order Malliavin differentiability of the forward process $X_t$ follows by similar technique as in \cite[Theorem 3.4]{BanosNilssen16}}. This result can be viewed as an extension of the latter theorem to the case of forward SDEs coupled with solution to a BSDE.
			
			Moreover, by means of compactness arguments, one can show that the second order Malliavin derivative of $X_t$ has moments of any order i.e. for all $p >1$
			\begin{align}
				\sup_{0\leq s',s\leq t \leq T} \mathbb{E}|D_{s'}D_s X_t|^{p} < \infty.
			\end{align}
			{Furthermore, Theorem \ref{secondMal} is an optimal result in the sense that if the drift $b$ is bounded in the space variable $x$, Lipschitz continuous in $y$ and $z$ with bounded (weak) derivative in $x$, the driver $f$ and the terminal value $\phi$ satisfy assumptions (A7) and (A8) respectively, then $(X_t,Y_t,Z_t)\in \mathbb{D}^{2,p}\times \mathbb{L}_{2,2} \times \mathbb{L}_{2,2}$ for all $p\geq 1$ and $(X_t,Y_t,Z_t) \notin \mathbb{D}^{3,p}\times \mathbb{L}_{2,2} \times \mathbb{L}_{2,2}$ for all $p\geq 1$.}
		\end{remark}
		{\begin{remark}\label{remlocdeco1}
				Assuming further that the drift $b$ is continuously differentiable in $t$ with bounded derivatives, the following explicit representation of the second order Malliavin derivative of the forward process $X_t$ is valid for all $0\leq s'\leq s \leq t$
				\begin{align*}
					D_{s'}D_sX_t =& 2D_sX_t\Big(\tilde b(t,X_t)D_{s'}X_t -\tilde b(s,X_s)D_{s'}X_s -\tilde b(s',X_{s'}) -\int_{s\vee s'}^t {\partial_u}\tilde b(u,X_u)D_{s'}X_u\mathrm{d}u\\ 
					&-\int_{s\vee s'}^t {\partial_x}\tilde b(u,X_u)\cdot b(u,X_u) D_{s'}X_u\mathrm{d}u -\int_{s\vee s'}^t {\partial_x}\tilde b(u,X_u) D_{s'}X_u\mathrm{d}W_u \Big).
				\end{align*}
				Indeed, let
				\[  F(t,x) = \int_0^x \tilde b(t,y)\mathrm{d}y - \tilde b(t,0).\]
				Then,
				\[  {\partial_t} F(t,x) = \int_0^x {\partial_t} \tilde b(t,y)\mathrm{d}y - \frac{\partial}{\partial t} \tilde b(t,0). \]
				Notice that for all $0\leq s\leq t$ the Malliavin derivative $D_sX_t$ of $X_t$ writes
				\[ D_sX_t = \exp\left( \int_s^t \partial_x\tilde b(u,X_u)\mathrm{d}u  \right)  \]
				By the It\^o's formula:
				\begin{align*}
					\int_s^t \tilde \partial_x\tilde b(u,X_u)\mathrm{d}u= 2(F(t,X_t) - F(s,X_s)) -2 \int_s^t {\partial_u}F(u,X_u)\mathrm{d}u -2\int_s^t \tilde b(u,X_u)\mathrm{d}X_u.
				\end{align*}
				Hence the sought equation is derived by a direct computation.
		\end{remark}}

		{In order to prove Theorem \ref{secondMal}}, we will follow the strategy developed in \cite[Theorem 4.1]{ImkDos} which is based on a careful application of \cite[Lemma 2.1]{ImkDos} with a slight modification. In contrast to \cite{ImkDos}, here we consider a more general setting. The system \eqref{eqmain3} is (weakly) coupled (whereas \cite{ImkDos} considered a decoupled FBSDE), the drift $b$ has a bounded weak derivative in the spatial variable $x$ instead of being twice continuously differentiable as in \cite{ImkDos} and $f_y'$ and $ f_z'$ are bounded by $\Lambda(1+|z|^{\alpha})$ and $\Lambda(1+\ell(y)|z|)$, respectively (and uniformly bounded by a constant $C>0$ in \cite{ImkDos} ). The latter means that the linear BSDE \eqref{MalY} has non-uniform Lipschitz coefficients on both the backward and the control components. We then approximate the BSDE \eqref{MalY} by truncating the stochastic constants appearing in both components. {Note in addition that the second Malliavin derivative of the forward process $X$ is not continuous and thus does not satisfy the standard requirement  $\mathbb{E}[\sup_{0\leq s',s \leq T}|D_{s'}D_s X_t|^p] < \infty$ as that is often in the literature.}

		Subsequently, let us consider the following family of truncated BSDEs: For $0\leq s \leq t \leq T$ and $n\in \mathbb{N}$
		\begin{align}\label{appro}
			U_{s,t}^n = D_s\xi + \int_t^T G^n(r,\Theta^n_{s,r})\mathrm{d}r -\int_t^T V^n_{s,r}\mathrm{d}W_r, \quad \Theta^n_{s,r} = (D_sX_r, U_{s,r}^n,V_{s,r}^n ),
		\end{align}
		where, $G^n: [0,T]\times \mathbb{R}\times\mathbb{R}\times\mathbb{R}\rightarrow \mathbb{R}$ is given by
		\begin{align}\label{driverG}
			G^n(t,x,u,v) =  f_x'(t,\bold{X}_t)x + f_y'(t,\bold{X}_t^n)u + f_z'(t,\bold{X}_t^n)v ,\quad \bold{X}_t^n = ( X_t,Y_t,\rho_n(Z_t)),
		\end{align}
		and $\rho_n: \mathbb{R}\rightarrow \mathbb{R} $ is continuously differentiable and satisfies the following properties\footnote{ Note that, it is always possible to build such functions see for instance \cite{ImkDos}}
		\begin{itemize}
			\item[(i)] $({\rho}_n)_{n\in \mathbb{N}}$ converges uniformly to the identity. For all $n \in \mathbb{N}$ and $x \in \mathbb{R}$ 
			\begin{equation}\label{trunc}
				\rho_n(x) =	\begin{cases}
					n+1,& \quad x > n+2,\\
					x,& \quad |x|\leq n,\\
					-(n+1),& \quad x < -(n+2).
				\end{cases}
			\end{equation}
			In addition $|{\rho}_n(x)| \leq |x|$ and $ |{\rho}_n(x)| \leq n+1 $.
			\item[(ii)] The derivative $\nabla {\rho}_n$ is absolutely uniformly bounded by 1, and converges to 1 locally uniformly.
		\end{itemize}
		\begin{remark}\label{remk 4.11}
			{One can show using} the properties of the functions $\rho_n$ and $\ell$ that:
			\begin{align*}
				|f_y'( X_t,Y_t,\rho_n(Z_t)) | &\leq \Lambda(1+ |Z_t|^{\alpha})\\
				|f_z'( X_t,Y_t,\rho_n(Z_t))| &\leq \Lambda(1+ \ell(|Y_t|)|Z_t|).
			\end{align*}
			Moreover, the boundedness of $Y_t$ gives 
			\[ \sup_{n\in \mathbb{N}}\|| f_z'(t,\bold{X}_t^n)|*W\|_{BMO} \leq \Lambda\|(1+ \ell(|Y_t|)|Z_t|)*W \|_{BMO} < \infty. \]
			{Hence}, the family of drivers $(G_n)_{n\in \mathbb{N}}$ given by \eqref{driverG} satisfies the same conditions as the one of the BSDE \eqref{MalY}.
		\end{remark}
		
		We have the following a-priori estimates
		\begin{lemm}
			{Suppose Assumption \ref{assum3} is valid and} $(U^n,V^n)$ solves the BSDE \eqref{appro}. Then for any $p>1$ we have
			\begin{align}\label{4.13}
				\sup_{n\in\mathbb{N}} \sup_{0\leq s\leq T} \mathbb{E} \Big(\int_0^T |U_{s,t}^n|^2 + |V_{s,t}^n|^2 \mathrm{d}t\Big)^{p} < \infty.
			\end{align}
		\end{lemm}
		\begin{proof}
			From Remark \ref{remk 4.11}, we know that the family of drivers $(G^n)_{n\in\mathbb{N}}$ satisfies {a stochastic Lipschitz type condition}. Then, from \cite[Lemma 3.8]{ImkRhossOliv} there exist a constant $q \in (1, \infty)$ which only depends on the BMO norm of $Z*W$ and a constant $C$ independent of $n \in \mathbb{N}$ such that
			\begin{align*}
				\sup_{n\in\mathbb{N}} \sup_{0\leq s\leq T} \mathbb{E} \Big[ \Big(\int_0^T |U_{s,t}^n|^2\mathrm{d}t\Big)^{p} + \Big(\int_0^T |V_{s,t}^n|^2 \mathrm{d}t\Big)^{p} \Big] \leq C \mathbb{E} \Big[ |D_s \xi|^{2pq} + \Big( \int_0^T|f_x'(t,\bold{X}_t)D_sX_t| \mathrm{d}t \Big)^{2pq}   \Big]^{\frac{1}{q}}.
			\end{align*}
			For any $\gamma >2$, we have 
			\begin{align*}
				\mathbb{E}|D_s\xi|^\gamma &= \mathbb{E}|\phi'(X_T)D_sX_T|^{\gamma}\leq C(1+ \mathbb{E}|X_T|^{2\gamma}) + C\mathbb{E}|D_sX_T|^{2\gamma}< \infty.
			\end{align*}
			{On the other hand, using {\bf(A2)}, Remark \ref{rmk 5.9} and the fact that $Y$ is uniformly bounded and $\ell$ is locally bounded}, we deduce that
			\begin{align*}
				\mathbb{E}\Big( \int_0^T|f_x'(t,\bold{X}_t)D_sX_t| \mathrm{d}t \Big)^{\gamma} \leq C \mathbb{E}\Big(\sup_{t\in [0,T]}|D_s X_t|^{2\gamma} + \Big(\int_0^T (1+ |Z_t|^2)\mathrm{d}t \Big)^{2\gamma}\Big)< \infty.
			\end{align*}
			Combining the above bounds, provide the desired result.
		\end{proof}
		The next lemma gives an existence, uniqueness and Malliavin differentiability {result} of solution processes to BSDE \eqref{appro}. It is as an extension of \cite[Lemma 4.2]{ImkDos}
		\begin{lemm}\label{lem 4.13}
			For each $n\in \mathbb{N}$, the BSDE \eqref{appro} has a unique solution $(U^n,V^n) \in \mathcal{S}^{2p}([0,T]\times [0,T]) \times\mathcal{H}^{2p}([0,T]\times [0,T])$ for any $p>1$. Moreover for $0\leq s \leq t \leq T$ the random variables $(U^n_{s,t},V^n_{s,t})$ are Malliavin differentiable and a version $(D_{s'}U^n_{s,t},D_{s'}V^n_{s,t})_{0\leq s'\leq s \leq t \leq T}$ satisfies
			\begin{align}\label{appro1}
				D_{s'}U^n_{s,t} &= D_{s'}D_s\xi - \int_t^T D_{s'}V^n_{s,r}\mathrm{d}W_r \notag\\
				&+ \int_t^T \Big[ \mathcal{R}_{s',s,r}^n + f_y'(r,\bold{X}_r^n)D_{s'}U^n_{s,r} + f_z'(r,\bold{X}_r^n)D_{s'}V^n_{s,r}  \Big]\mathrm{d}r,
			\end{align}
			where $\mathcal{R}_{s',s,t}^n = (D_{s'}G^n)(t,\Theta_{s,t}^n) + f_x'(t,\bold{X}_t)D_{s'}D_sX_t$ and
			$\Theta_{s,r}^n = (D_sX_r,U^n_{s,r},V^n_{s,r})$. Furthermore for any $p>1$
			\begin{align}\label{4.15}
				\sup_{n\in \mathbb{N}}\int_0^T\mathbb{E}\left\{ \|D_{s'}U^n\|_{\mathcal{S}^{2p}([0,T]\times[0,T])}^{2p} + \|D_{s'}V^n\|_{\mathcal{H}^{2p}([0,T]\times[0,T])}^{2p}  \right\}\mathrm{d}s' < \infty.
			\end{align}
		\end{lemm}
		
		The first term in the expression of $\mathcal{R}_{s',s,t}^n$ writes:
		\begin{align*}
			(D_{s'}G^n)(t,\Theta_{s,t}^n) = D_{s'}[f_x'(t,\bold{X}_t)]D_sX_t + D_{s'}[f_y'(t,\bold{X}_t^n)]U_{s,t}^n + D_{s'}[f_z'(t,\bold{X}_t^n)]V_{s,t}^n,
		\end{align*}
		and each term is explicitly given by: 
		
		\begin{align}\label{5.33}
			\begin{cases} 
				D_{s'}[f_x'(t,\bold{X}_t)]D_sX_t = f_{xx}^{''}(t,\bold{X}_t)D_{s'}X_tD_sX_t + f_{xy}^{''}(t,\bold{X}_t)D_{s'}Y_tD_sX_t \\
				\qquad \qquad \qquad \qquad \qquad+ f_{xz}^{''}(t,\bold{X}_t)D_{s'}Z_tD_sX_t, \\
				D_{s'}[f_y'(t,\bold{X}_t^n)]U_{s,t}^n= f_{xy}^{''}(t,\bold{X}_t^n)D_{s'}X_tU_{s,t}^n + f_{yy}^{''}(t,\bold{X}_t^n)D_{s'}Y_tU_{s,t}^n\\
				\qquad \qquad \qquad \qquad \qquad + f_{yz}^{''}(t,\bold{X}_t^n)\rho_n'(Z_t)D_{s'}Z_tU_{s,t}^n,\\
				D_{s'}[f_z'(t,\bold{X}_t^n)]V_{s,t}^n = f_{xz}^{''}(t,\bold{X}_t^n)D_{s'}X_tV_{s,t}^n + f_{yz}^{''}(t,\bold{X}_t^n)D_{s'}Y_tV_{s,t}^n\\
				\qquad \qquad \qquad  \qquad \qquad + f_{zz}^{''}(t,\bold{X}_t)\rho_n'(Z_t)D_{s'}Z_tV_{s,t}^n.
			\end{cases}
		\end{align}

		\begin{proof} {It is enough} to prove that the coefficients of the BSDE \eqref{appro1} satisfy the assumptions (H1)--(H4) {in Appendix \ref{apdix3} and the result will follow by applying} Theorem \ref{thapendix}. Conditions (H1) and (H2) are straightforward from the assumptions of the Lemma. More precisely, the first term in the BSDE \eqref{appro1} is given by:
			\begin{align*}
				D_{s'}D_s\xi = D_{s'}(\phi'(X_T)D_sX_T)= \phi''(X_T)D_sX_T + \phi'(X_T)D_{s'}D_sX_T,
			\end{align*}
			{and using the assumptions of the lemma, we have}
			\begin{align}\label{4.17}
				\mathbb{E}|D_{s'}D_s\xi|^{\gamma} \leq C \mathbb{E}\Big( 1+ |X_T|^{2\gamma}  + |D_sX_T|^{2\gamma} + |D_{s'}D_sX_T|^{2\gamma} \Big) < \infty.
			\end{align}
			On the other hand, note that for each $n \in \mathbb{N}$ the family of drivers $(G^n)_{n\in \mathbb{N}}$ is Lipschitz continuous in its third and fourth components, since $f_y'(\cdot,\cdot,\cdot,\rho_n(z))$ and $f_z'(\cdot,\cdot,\cdot,\rho_n(z))$ are bounded. In addition, using the growth of $f_x'$ (see Assumption \ref{assum3} (A2)), the bound of $Y$ and Remark \ref{rmk 5.9}, we deduce for any $p >1:$
			\begin{align*}
				\sup_{0\leq s \leq T}\mathbb{E}\Big[\sup_{0\leq t \leq T}|f_x'(t,\bold{X}_t)D_sX_t|^p\Big] \leq \Lambda \sup_{0\leq s \leq T}\mathbb{E}\Big[(1+\sup_{0\leq t \leq T}|Z_t|^2)^{2p}\Big]^{\frac{1}{2}}\mathbb{E}\Big[\sup_{0\leq t \leq T}|D_sX_t|^{2p}\Big]^{\frac{1}{2}}< \infty.
			\end{align*}
			This implies that for each $n \in \mathbb{N}$, $(G^n)_{n\in \mathbb{N}}$ satisfies (H3) in Appendix \ref{apdix3}. We need to prove that: 
			\begin{align}\label{4.18}
				\sup_{n\in\mathbb{N}} \sup_{0\leq s\leq T} \mathbb{E} \Big[ \Big( \int_0^T |\mathcal{R}_{s',s,t}^n|\mathrm{d}t\Big)^{\gamma}\Big] < \infty.
			\end{align} 
			Let us remark that
			\begin{align*}
				\mathbb{E} \Big( \int_0^T |\mathcal{R}_{s',s,t}^n|\mathrm{d}t\Big)^{\gamma} \leq C \mathbb{E}\Big[ \Big( \int_0^T |(D_{s'}G^n)(t,\Theta_{s,t}^n)|\mathrm{d}t\Big)^{\gamma} + \Big( \int_0^T |f_x'(t,\bold{X}_t)D_{s'}D_sX_t|\mathrm{d}t\Big)^{\gamma} \Big]= I_1 + I_2.	
			\end{align*}
			By using once more the growth of $f_x'$, the boundedness of $Y$ and Remark \ref{rmk 5.9}, we derive the following estimate for $I_2$
			\begin{align*}
				I_2 & \leq C \mathbb{E}\Big( \int_0^T (1+ |Y_t| +\ell(Y_t)|Z_t|^{\alpha})^2\mathrm{d}t + \int_0^T |D_{s'}D_sX_t|^2\mathrm{d}t\Big)^{\gamma}\\
				&\leq C \Big(1 + \mathbb{E}\sup_{0\leq t \leq T}|Z_t|^{4\gamma}\Big) + C \mathbb{E}\Big(\int_0^T |D_{s'}D_sX_t|^2\mathrm{d}t\Big)^{\gamma} < \infty,
			\end{align*}
			where $C$ is a nonnegative constant independent of $n$. As for the term $I_1$, it is sufficient to prove that each terms in \eqref{4.15} is bounded uniformly in $n$. This can be achieved similarly as in \cite[P.364]{ImkDos}. Thus, (H4) is satisfied. The bound \eqref{4.15},  follows from \cite[Lemma 3.8]{ImkRhossOliv}, combined with \eqref{4.17} and \eqref{4.18}. This ends the proof.
		\end{proof}
		\begin{proof}[Proof of Theorem \ref{secondMal}] As stated above, the proof follows by applying \cite[Lemma 2.1]{ImkDos}. We know from Lemma \ref{lem 4.13} that the BSDE \eqref{appro} has a unique Malliavin differentiable solution $(U^n_{s,\cdot},V^n_{s,\cdot})$ such that its $(\mathbb{L}_{1,2}\times \mathbb{L}_{1,2} )\text{-norm}$ is bounded uniformly in $n$. 
			It remains to prove the $\mathcal{H}^2\text{-convergence}$ of the process $(U^n_{s,\cdot},V^n_{s,\cdot})$ to $(D_sY_{\cdot},D_sZ_{\cdot})$. Using \cite[Lemma 3.8]{ImkRhossOliv} we deduce that, there exist $q \in (1,\infty)$ only depending on $\sup_{n\in \mathbb{N}}\|f_z'(\cdot,\cdot,\cdot,\rho_n(Z_{\cdot}))*W\|_{BMO}$ such that:
			\begin{align*}
				&\sup_{0\leq s \leq T}\mathbb{E}\Big[ \int_0^T |D_sY_t-U_{s,t}^n|^2\mathrm{d}t + \int_0^T |D_sZ_t-V_{s,t}^n|^2\mathrm{d}t \Big]\\
				&\leq C \sup_{0\leq s \leq T} \mathbb{E}\Big[\Big(\int_0^T |f_y'(t,\bold{X}_t)- f_y'(t,\bold{X}_t^n)||U_{s,t}^n| + |f_z'(t,\bold{X}_t)-f_z'(t,\bold{X}_t^n)||V_{s,t}^n| \mathrm{d}t\Big)^{2pq}\Big]^{\frac{1}{q}}\\
				&\leq C \sup_{0\leq s \leq T} ( II_1 + II_2). 
			\end{align*}
			{We only show that $II_1$ is bounded uniformly in $n$ and the boundedness of $II_2$ follows similarly}. Using H\"older's inequality, we have
			\begin{align*}
				II_1 &= \mathbb{E}\Big[\Big(\int_0^T |f_y'(t,\bold{X}_t)-f_y'(t,\bold{X}_t^n)||U_{s,t}^n|\mathrm{d}t\Big)^{2pq}\Big]^{\frac{1}{q}}\\
				&\leq \mathbb{E}\Big[\Big(\int_0^T |f_y'(t,\bold{X}_t)-f_y'(t,\bold{X}_t^n)|^2\mathrm{d}t\Big)^{2pq}\Big]^{\frac{1}{2q}}\sup_{0\leq s \leq T}\mathbb{E}\Big[\Big(\int_0^T |U_{s,t}^n|^2\mathrm{d}t\Big)^{2pq}\Big]^{\frac{1}{2q}} .
			\end{align*}
			Recall that $\bold{X}_{\cdot} = (X_{\cdot},Y_{\cdot},Z_{\cdot})$ and $\bold{X}_{\cdot}^n = (X_{\cdot},Y_{\cdot},\rho_n(Z_{\cdot}))$. From the assumptions of the theorem(continuity of $f_y'$) combined with the boundedness of $Y_t$, Remark \ref{rmk 5.9} and Remark \ref{remk 4.11}, we deduce that the first term of the above inequality is finite, while the second one is finite thanks to \eqref{4.13}. The result follows by applying the Dominated convergence theorem. Then, similar arguments as in \cite{ImkDos} can be applied here to show that $(D_{s'}D_sY,D_{s'}D_sZ)$ is the unique solution to the linear BSDE \eqref{secondY} as well for $D_tD_sY_t$ being a version of $D_sZ_t$. {We omit the details}. The proof is completed.  
		\end{proof}
		
		We are now in position to establish the main result of this section
		\begin{thm}\label{densZ} {Suppose} that assumptions of Theorem \ref{secondMal} are valid. Then the law of the {control} component $Z_t$ solution to FBSDE \eqref{eqmain3} has a density which is continuous with respect to the Lebesgue measure if :
			\begin{itemize}
				\item[(i)] There exists $c$ s.t., $0\leq D_{s'}D_sX_t \leq c, $ for all $0<s', s<t\leq T.$
				\item[(ii)]  $f_x',f_y',f_{xx}^{''},f_{yy}^{''},f_{zz}^{''}\geq 0$ and $f_{xz}^{''}=f_{yz}^{''} = 0$.
				\item[(iii)] $f_{xy}^{''} = 0$ or $(f_{xy}^{''}\geq 0 \text{ and } \phi' \geq 0, \mathcal{L}(X_T)\text{a.e.}).$
			\end{itemize}
			If there exists $A \in \mathcal{B}(\mathbb{R})$ such that $\mathbb{P}(X_T\in A|\mathfrak{F}_t) > 0,$ and such that 
			\[ (1_{\{\underline{\phi''}<0\}}\underline{\phi''}e^{2\Lambda_bt} + \underline{\phi'} 1_{\{\underline{\phi'}<0\}}c) + (1_{\{\underline{\phi''}\geq 0\}}\underline{\phi''}+ \underline{f_{xx}^{''}}(t)(T-t))e^{-2\Lambda_bt} \geq 0, \]
			and 
			\[ 1_{\{\underline{\phi''}^{A}<0\}}\underline{\phi''}^{A}e^{2\Lambda_bt} + \underline{\phi'}^{A} 1_{\{\underline{\phi'}^{A}<0\}}c + (1_{\{\underline{\phi''}^{A}\geq 0\}}\underline{\phi''}^{A}+ \underline{f_{xx}^{''}}(t)(T-t))e^{-2\Lambda_bt} > 0. \]
		\end{thm}
		\begin{proof}
			Under the assumptions of the Theorem and by the standard linear technique we obtain 
			\begin{align*}
				D_{s'}D_sY_t =&\mathbb{E}^{\mathbb{Q}}\Big[ e^{\int_{t}^{T}f_y'(u,\bold{X}_u)\mathrm{d}u}\left( \phi''(X_T)D_{s'}X_TD_sX_T + \phi'(X_T)D_{s'}D_sX_T\right)\\
				&+ \int_t^T e^{\int_{t}^{r}\nabla_yf(u,\bold{X}_u)\mathrm{d}u}\Big( f_x'(r,\bold{X}_r)D_{s'}D_sX_r + f_{xx}^{''}(r,\bold{X}_r)D_{s'}X_rD_sX_r \\
				&+ f_{xy}^{''}(r,\bold{X}_r)D_{s'}X_rD_sX_r +  f_{xy}^{''}(r,\bold{X}_r)D_{s'}X_rD_sY_r \\
				&+ f_{yy}^{''}(r,\bold{X}_r)D_{s'}Y_rD_sY_r + f_{zz}^{''}(r,\bold{X}_r)D_{s'}Z_rD_sZ_r  \Big)\mathrm{d}r\Big|\mathfrak{F}_t \Big],
			\end{align*}
			where the probability measure $\mathbb{Q}$ is given by \eqref{4.5}.	Then, one can use a similar technique as in the proof of  \cite[Theorem 4.2]{Mastrolia} to deduce the result. This ends the proof.	
		\end{proof}
		{\begin{cor}\label{cor520}
				Under the assumptions of Theorem \ref{densZ} , the probability density of the {control component} $Z_t$ solution to FBSDE \eqref{eqmain3} satisfy for all $t \in [0,T)$ and $z>0$
				\begin{align*}
					\frac{\mathbb{E}|Z_t-\mathbb{E}[Z_t]|}{2t (\Upsilon^{(1)}e^{-Kt})^2}\exp\Big( -\frac{(z-\mathbb{E}[Z_t])^2}{2t e^{-2Kt}}\Big) \leq \rho_{Z_t}(z) \leq 	\frac{\mathbb{E}|Z_t-\mathbb{E}[Z_t]|}{2t (\omega(t)e^{-Kt})^2}\exp\Big(- \frac{(z-\mathbb{E}[Z_t])^2}{2t e^{2Kt}}  \Big).
				\end{align*}
				Moreover, its tail probabilities satisfy the following Gaussian-type bounds
				\begin{align*}
					\mathbb{P}(Z_t\geq z)\leq \exp\Big( -\frac{(z-\mathbb{E}[Z_t])^2}{2t e^{2Kt}} \Big),\,\,\, 	\mathbb{P}(Z_t\leq -z)\leq \exp\left( -\frac{(z+\mathbb{E}[Z_t])^2}{2t e^{2Kt}}\right).
				\end{align*}
			\end{cor}
			\begin{proof} It is enough to establish similar bound as \eqref{4.8}. We recall that $Z_t = v'(t, X_t)$, for all $t\in [0,T)$ where $v'$ stands for the derivative with respect to the spatial component of the classical solution $v$ to PDE \eqref{eqmainP}. From the above theorem, we also know that the law of $Z_t$ has an absolute continuous density with respect to the Lebesgue measure, since $\Gamma_{Z_t} > 0$ $\mathbb{P}\text{-a.s.}$ and this implies that $(v''(t,X_t))^2 > 0$ and $\langle D_sX_t, D_sX_t\rangle_{L^2([0,t])} > 0\quad \mathbb{P}\text{-a.s.}$ The latter is always true then, there is a function $\omega(t)>0$ such that 
				$$v''(t,X_t) \geq \omega(t) \text{ or } {v''(t,X_t) \leq -\omega(t)} \quad \mathbb{P}\text{-a.s.} \text{ for all } t\in [0,T).$$ {We assume that $v''(t,X_t) \geq \omega(t) \, \mathbb{P}\text{-a.s.}$}, since the other case can be treated similarly. From  \eqref{4.11} with $\sigma = 1$, we deduce that
				\begin{align*}
					0< \omega(t)e^{-Kt} \leq D_sZ_t = v''(t,X_t)D_sX_t \leq \Lambda \Upsilon^{(1)} e^{Kt},
				\end{align*} 
				where the constant $\Upsilon^{(1)}\geq 0$ is an upper bound of $v''$ that can be derived from \eqref{remk 5.2}. Therefore,
				\begin{align*}
					0< t \left( \omega(t)e^{-Kt}\right)^2 \leq \int_0^t D_sZ_t \mathbb{E}[D_sZ_t|\mathfrak{F}_s]\mathrm{d}s\leq t \Big(\Upsilon^{(1)} e^{Kt}\Big)^2.
				\end{align*}
		\end{proof}}
		\begin{remark}\leavevmode
			\begin{itemize}
				\item[(1)] It is worth mentioning that the results in Section 5.1 and 5.2 can be recovered if we remove the H\"older continuity of the drift $b$ in $(t,x)$. However the price to pay will be to further assume that $\sigma=1$, the drift $b$ does not depend on the control variable $z$ i.e. $b(t,x,y,z)=b(t,x,y)$ and all the other assumptions remaining the same.
				\item[(2)] 	Comparing our results to those obtained in \cite{AbouraBourguin, F.Antonelli}, the bounds exhibit a dependence with respect to driver of the BSDE and the solution to the quasi linear PDE associated to the FBSDE. 
			\end{itemize}
		\end{remark}
		
		\subsubsection*{Example}
		{Let $b(t,x,y,z)=0,\, \sigma(t,x)=1,\, f(t,x,y,z)=y+z$ and $ \phi(x)=\frac{e^x}{1+e^x}$. 
			Then the conditions of Theorem \ref{prop 4.10} are satisfied and thus the corresponding FBSDEs has a unique strong solution. In addition the density of the solution have the following bounds: for all $t\in[0,T)$ and $y>0$ it holds:
			\begin{align*}
				\frac{\mathbb{E}|Y_t-\mathbb{E}[Y_t]|}{2t \left(\Upsilon e^{\Lambda_{{b}}t}\right)^2}\exp\Big( -\frac{(x-\mathbb{E}[Y_t])^2}{2t \left( \alpha(t)e^{-\Lambda_{2,\tilde{b}}t}\right)^2}\Big) \leq \rho_{Y_t}(y) \leq 	\frac{\mathbb{E}|Y_t-\mathbb{E}[Y_t]|}{2t \left( \alpha(t)e^{-\Lambda_{{b}}t}\right)^2}\exp\Big(- \frac{(y-\mathbb{E}[Y_t])^2}{2t \left(\Upsilon e^{\Lambda_{2,\tilde{b}}t}\right)^2}  \Big),
			\end{align*}
			where $\Lambda_{2,\tilde{b}}=\max( 1, \|v'\|_{\infty},\|v''\|_{\infty} )$ and $v$ is the solution to the following quasilinear PDE
			\begin{equation}\label{eqmainPrac1}
				\begin{cases}
					\frac{\partial v}{\partial t}(t,x)+ \frac{1}{2}\frac{\partial^2 v}{\partial x^2}(t,x) =- v(t,x)-\nabla_x v(t,x)\\
					v(T,x)= \phi(x).
				\end{cases}
			\end{equation} 
		}

		\subsection{The discontinuous case}\label{section6}
		In this subsection, by considering a trivial diffusive coefficient, the goal is to relax {\bf(A1)} by removing the H\"older continuity assumptions on the coefficients. We propose a refinement of Theorem \ref{th5.5} and Theorem \ref{th5.9} as well, when the  drift coefficient is allowed to be merely discontinuous(see Assumption \ref{assum1} {(AX)}). The FBSDE of interest is given by
		\begin{equation}\label{540}
			\begin{cases}
				X_t^x = x + \displaystyle \int_{s}^{t} b(r,X_r^x,Y_r^x,Z_r^x)\mathrm{d}r + W_t-W_s ,\\
				Y_t^x = \phi(X_T^x)+ \displaystyle \int_{t}^{T} f(r,X_r^x,Y_r^x,Z_r^x)\mathrm{d}r - \int_{t}^{T} Z_r^x\mathrm{d}W_r.
			\end{cases}
		\end{equation}
		We first start by studying the regularities of solution to \eqref{540} in this context.
		\subsubsection{Smoothness of solutions of QFBSDEs with discontinuous coefficients }
		This subsection is devoted to the study of {both} the classical and the Malliavin differentiability of solutions to equations \eqref{540}.
		
		
		Next, we first study the regularity of $x\mapsto (X_t^x,Y_t^x)$ in the Sobolev sense. 
		Let us consider the weighted Sobolev space $W^{1,p}(\mathbb{R}^d,\nu)$ which consists of functions $\mu \in L^p(\mathbb{R}^d,\nu)$ admitting weak derivative of first order denoted by $\partial_{x_i}\mu$, equipped with the following complete norm 
		\[ \| \mu\|^p_{W^{1,p}(\mathbb{R}^d,\nu)} := \| \mu\|^p_{L^p(\mathbb{R}^d,\nu)} + \sum_{i=1}^{d} \| \partial_{x_i}\mu \|^p_{L^p(\mathbb{R}^d,\nu)} .\]
		The space $L^p(\mathbb{R}^d,\nu)$ is the weighted Lebesgue space of measurable functions $\mu: \mathbb{R}^d \rightarrow \mathbb{R}^d$ such that 
		$\| \mu\|^p_{L^p(\mathbb{R}^d,\nu)}:= \int_{\mathbb{R}^d} |\mu(x)|^p\nu(x)\mathrm{d}x < \infty$ and
		the function $\nu$ stands for the weight possessing a $p\text{th}$ moment with respect to the Lebesgue measure on $\mathbb{R}^d$, i.e., $\nu : \mathbb{R}^d\rightarrow [0,\infty)$ such that $\int_{\mathbb{R}^d} (1+ |x|^p)\nu(x)\mathrm{d}x < \infty$ for some $p > 1$.
		
		\begin{prop}\label{th4.4} Let Assumption \eqref{assum1} be in force with $\sigma \equiv \mathbb{I}_{d\times d}$.
			For every $\delta >0,$ and for any bounded open set $\mathcal{O}\subset \mathbb{R}^d$, the solution $(X^x,Y^x,Z^x)$ to equation \eqref{540} satisfies:
			\begin{equation}
				(X_t^x, Y_t^x) \in L^2(\Omega, W^{1,p}(\mathbb{R}^d,\nu)) \times L^2(\Omega, W^{1,p}(\mathcal{O})) \text{ for all } t\in [0, T-\delta].
			\end{equation}
			
			Moreover for $d =1$, and for all $0\leq t < T,$ the following explicit representations hold $ \mathrm{d}t\otimes\mathrm{d}\mathbb{P}\text{-a.s.}$
			\begin{align}
				\xi_t^x := \frac{\partial}{\partial x}X_t^x&=\exp\Big({-\int_0^{t}\int_{\mathbb{R}}\tilde b\left(u,z\right)L^{X^{x}}(\mathrm{d}u,\mathrm{d}z)}\Big),\label{sobderX}\\
				\eta_t^x:=\frac{\partial}{\partial x}Y_t^x &=v'(t,X_t^x) \xi_t^x ,\label{sobderY}
			\end{align}
			$L^{X^{x}}(\mathrm{d}u,\mathrm{d}z)$ stands for the integration in space and time with respect to the local time of $X^x$ {and $v'$ is the spatial partial derivative of $v$ solution to PDE \eqref{eqmainP1}}.
		\end{prop}
		
		\begin{proof}
			We recall that, Theorem \ref{th3.3}  ensures the unique weak solvability of FBSDE \eqref{540} and there is a {\it "decoupling field"} $v \in W^{1,2,d+1}_{\text{loc}}([0,T[\times \mathbb{R}^d,\mathbb{R})$ solution to PDE \eqref{eqmainP1} such that the forward equation in \eqref{eqmain3}  can be written as:
			\begin{align}\label{eqfor1}
				X_t^x= x+\int_s^t\tilde b(r,X_r^x)\mathrm{d}r + W_t -W_s,
			\end{align}
			where $\tilde{b}(t,x):= b(t,x,v(t,x),\nabla_xv(t,x)).$ {Using \eqref{3.4} , \eqref{3.6} and the assumptions of Proposition \ref{th5.1}}  the drift $\tilde{b}$ is bounded for all $t \in [0,T-\delta]$. Thus, {using} \cite[Theorem 3]{NilssenProske} the forward process $X^x$ is Sobolev differentiable. On the other hand, the function $v$ is Lipschitz continuous for all $t < T$ (see \eqref{3.6} ) then for all $\omega$ outside every $\mathbb{P}\text{-null}$ set, the process $Y^x_t(\omega) = v(t,X_t^x(\omega))$ belongs to $W^{1,1}_{Loc}(\mathbb{R})$ (see \cite{LeoniMoreni05}). Then \eqref{sobderY} follows by applying the chain rule developed in the same paper \cite[Theorem 1.1]{LeoniMoreni05}. This concludes the proof.
		\end{proof}

		In the following, we are now studying the variational differentiability of solution to \eqref{540} in the sense of Malliavin.	
		Without loss of generality in what follows we assume that equation \eqref{540} has initial value given by $(0,x).$ 
		\begin{prop}[Malliavin differentiability]\label{th5.1} Let assumptions of {Proposition \ref{th5.1}} be in force.  Then for every $\delta >0,$ the FBSDE \eqref{540} has a unique strong Malliavin differentiable solution for all $t\in [0, T-\delta], T>0.$ 
			
			In particular, for $d=1$, and for all $0<s\leq t < T,$ a version of $(D_sX_t^x, D_sY_t^x)$ has the following explicit representation
			\begin{align}
				D_sX_t^x&=\exp\Big({-\int_s^{t}\int_{\mathbb{R}}\tilde b\left(u,z\right)L^{X^{x}}(\mathrm{d}u,\mathrm{d}z)}\Big),\label{malderX}\\
				D_sY_t^x &=v'(t,X_t^x) D_sX_t^x,\label{malderY}
			\end{align}
			where	$L^{X^{x}}(\mathrm{d}u,\mathrm{d}z)$ stands for the integration in space and time with respect to the local time of $X^x$ {and $v'$ is the spatial partial derivative of $v$ solution to the PDE \eqref{eqmainP1}}
		\end{prop}
		\begin{proof} 
			The proof follows as in \cite[Theorem 4.4]{RhOlivOuk} by using the result of Malliavin differentiability of SDEs with irregular drift from \cite{MMNPZ13}. The explicit representation \eqref{malderX} is well known (see for example \cite{BMBPD17, MeTa19}), while \eqref{malderY} follows by applying the chain rule for Malliavin calculus (see \cite[Proposition 1.2.4]{Nua06}) by noticing that $v$ is Lipschitz continuous for all $t< T$ (see \eqref{3.6}). This ends the proof.
		\end{proof}
		
		\begin{remark}\label{repr1}
			The following relation between the Sobolev and the Malliavin derivatives of the forward and backward components hold respectively $\mathrm{d}t\otimes\mathbb{P}\text{-a.s.}$
			\begin{align*}
				D_s X_t^x &= \xi_t^x (\xi_s^x)^{-1} ,\\
				D_s Y_t^x &= \eta_t^x (\xi_s^x)^{-1}
			\end{align*} 
			for any $s,t \in [0,T), s\leq t.$
		\end{remark}
		\subsubsection{Density analysis of FBSDEs with discontinuous drift case} The main objective in this section is to conduct the analysis of densities obtained in the preceding section by assuming now 
		 
		\begin{assum}\label{assum2} \leavevmode
			\begin{itemize}
				\item[{\bf(A9)}] Assumption \ref{assum1} is valid in one dimension $(d=1)$ with $\sigma\equiv 1$ and $b$ bounded in the {forward and the} control variable i.e., $\forall (t,x,y,z)\in [0,T]\times\mathbb{R}\times \mathbb{R}\times \mathbb{R}, |b(t,x,y,z)|\leq \Lambda (1+|y|)$.
				\item[{\bf(A10)}] Assumptions {\bf(A4)} and {\bf(A5)} hold.
				
			\end{itemize}
		\end{assum}
		\begin{thm}[Existence of a Density for $X_t$]\label{densX}
			Under {\bf(A9)}, the finite dimensional laws of the strong solution to the forward SDE \eqref{eqfor1} are absolutely continuous with respect to the Lebesgue measure. {Precisely} for fixed $t\in (0,T],$ the law of $X_t^x$ denoted by $\mathcal{L}(X_t^x)$ has a density with respect to the Lebesgue measure. 
		\end{thm}	
		\begin{proof}
			This follows by showing that  the exponential term in equation \eqref{malderX} is finite almost surely. The latter is obtained by using Girsanov theorem and Lemma \ref{lemmaexpoloc}. This ends the proof.
		\end{proof}

		\subsubsection*{Existence of a Density for the backward component $Y_t$} 
		We first derive the linear BSDE satisfied by a version $(D_sY_t^x,D_sZ_t^x)$ of the Malliavin derivatives of the solution to \eqref{540}.
		
		\begin{prop}\label{prop 3.4}
			Under Assumption \ref{assum2}, a version of $(D_sY_t,D_sZ_t)$ satisfies 
			\begin{align}\label{MalY}
				D_sY_t &= 0, \quad D_sY_t = 0, \quad t< s\leq T,\notag\\
				D_sY_t &= \phi'(X_T)D_sX_T+ \int_{t}^{T} \langle \nabla f(r,\bold{X}_r),D_s\bold{X}_r   \rangle\mathrm{d}r - \int_{t}^{T} D_sZ_r\mathrm{d}W_r,
			\end{align}
			where $D_sX_t$ is given by \eqref{malderX}.
			Moreover, $(D_tY_t)_{0\leq t\leq T}$ is a continuous version of $(Z_t)_{0\leq t\leq T}$. 
		\end{prop}
		\begin{proof} The first part of the proof follows from Theorem \ref{th5.1}. As for the representation, we appeal \cite[Theorem 5.3]{ImkRhossOliv} since $\tilde{b}$ is uniformly bounded for all $t\in [0,T]$. {The proof is completed}.
		\end{proof}
		
		%

		Below is stated the main result of this subsection.
		\begin{thm}\label{densY}
			Suppose Assumption \ref{assum2} is in force. {Suppose in addition} there is $\mathcal{A}\in \mathcal{B}(\mathbb{R})$ such that $\mathbb{P}(X_T \in \mathcal{A}|\mathfrak{F}_t) > 0$ and one of the following assumptions holds 
			\begin{itemize}
				\item[(H+)] $\phi' \geq 0,$ $\phi'_{|\mathcal{A}}>0, \mathcal{L}(X_T)\text{-a.e.}$ and $\underline{f} \geq 0,$
				\item[(H--)] $\phi' \leq 0,$ $\phi'_{|\mathcal{A}}<0, \mathcal{L}(X_T)\text{-a.e.}$ and $\overline{f} \leq 0.$
			\end{itemize}
			Then, $Y_t$ possesses an absolute continuous law with respect to the Lebesgue measure on $\mathbb{R}$.
		\end{thm}
		\begin{proof} From Assumption \ref{assum2}, we deduce that
			\begin{align}  \label{5.5}
				|f_z'(s,\bold{X}_s)| \leq \Lambda \big(1+ \ell(Y_s)|Z_s|\big) \in \mathcal{H}_{\text{BMO}}.
			\end{align}
			{Let $\mathbb{Q}$ be the probability measure defined by} 
			\begin{align}\label{4.5}
				\frac{\mathrm{d}\mathbb{Q}}{\mathrm{d}\mathbb{P}}\Big|_ {\mathfrak{F}_t} = \mathcal{E}\Big( f_z'(t,\bold{X}_t)* W\Big)_t=M_t.
			\end{align}
			{It follows from \eqref{5.5} that $M_t$} is uniformly integrable. Therefore, {taking the conditioning expectation with respect to $\mathbb{Q}$ on both sides of \eqref{MalY} gives for all $0\leq s \leq t < T$}
			\begin{align}
				D_s Y_t = \mathbb{E}^{\mathbb{Q}} \Big( \phi'(X_T)D_sX_T + \int_t^T \left( f_x'(r,\bold{X}_r)D_s X_r + f_y'(s,\bold{X}_r)D_s Y_r \right) \mathrm{d}r    \Big| \mathfrak{F}_t\Big)
			\end{align}
			{Using} Remark \ref{repr1} and the well known linearisation method, $D_s Y_t$ can be rewritten:
			\begin{align*}
				D_s Y_t &= \mathbb{E}^{\mathbb{Q}} \Big[ e^{\int_{t}^{T}f_y'(u,\bold{X}_u)\mathrm{d}u}\phi'(X_T)\xi_T + \int_t^T e^{\int_{t}^{r}f_y'(u,\bold{X}_u)\mathrm{d}u}  f_x'(r,\bold{X}_r)\xi_r \mathrm{d}r    \Big| \mathfrak{F}_t \Big](\xi_s)^{-1}\\
				&= \mathbb{E} \Big[ \psi_T \phi'(X_T)\xi_T + \int_t^T \psi_r f_x'(r,\bold{X}_r) \xi_r \mathrm{d}r \Big| \mathfrak{F}_t  \Big] (\psi_t)^{-1}(\xi_s)^{-1},
			\end{align*}
			where the last equality is due to the Bayes' rule with $\psi$ is given by	
			\[ \psi_t := \exp\Big(  \int_{0}^{t}(f_y'(u,\bold{X}_u)-\frac{1}{2}|f_z'(u,\bold{X}_u)|^2)\mathrm{d}u+ \int_{0}^{t}f_z'(u,\bold{X}_u)\mathrm{d}W_u   \Big),\]
			and the expression of $\xi$ is given in \eqref{sobderX}.
			Therefore, from the computations above, the Malliavin covariance $\Gamma_{Y_t}$ of $Y_t$ is given by 
			\begin{equation*}
				\Gamma_{Y_t} = \Big(\mathbb{E}\Big[ \psi_T  \phi'(X_T)\xi_T+ \int_{t}^{T}\psi_r f_x'(r,\bold{X}_r)\xi_r \mathrm{d}r \Big|\mathfrak{F}_t \Big]\Big)^2 \times (\psi_t^{-1})^2\int_{0}^{t}(\xi_s^{-1})^2\mathrm{d}s.
			\end{equation*} 
			It is enough to prove that $\mathbb{E}\Big[ \psi_T  \phi'(X_T)\xi_T+ \int_{t}^{T}\psi_r f_x'(r,\bold{X}_r)\xi_r \mathrm{d}r \big|\mathfrak{F}_t \Big]\neq 0$ under assumption (H+) to deduce the desired result. Let us remark that, the product $\psi_T \xi_T$ can be rewritten as follow:
			\begin{align*}
				\psi_T \xi_T &= \exp\Big(-\int_0^{T}\int_{\mathbb{R}}\tilde b\left(u,z\right)L^{X^{x}}(\mathrm{d}u,\mathrm{d}z) + \int_0^T f_y'(u,\bold{X}_u)\mathrm{d}u \Big)\\
				&\quad \times \exp\Big( \int_{0}^{T}f_z'(u,\bold{X}_u)\mathrm{d}W_u -\frac{1}{2} \int_{0}^{T}|f_z'(u,\bold{X}_u)|^2\mathrm{d}u  \Big) = H_T\times M_T
			\end{align*}
			{ The Bayes' rule once more yields}
			\begin{align*}
				\mathbb{E}\Big[ \psi_T  \phi'(X_T)\xi_T+ \int_{t}^{T}\psi_r f_x'(r,\bold{X}_r)\xi_r \mathrm{d}r \Big|\mathfrak{F}_t \Big]
				=\mathbb{E}^{\mathbb{Q}}\Big[ \phi'(X_T)H_T + \int_{t}^{T}f_x'(r,\bold{X}_r) H_r  \mathrm{d}r \Big|\mathfrak{F}_t \Big]M_t
			\end{align*}
			Using the fact that $\phi' + f_x' \geq \underline{\phi} + \underline{f}$, we deduce:
			\begin{align*}
				\mathbb{E}\Big[ \psi_T  \phi'(X_T)\xi_T+ \int_{t}^{T}\psi_r f_x'(r,\bold{X}_r)\xi_r \mathrm{d}r \Big|\mathfrak{F}_t \Big]\geq&  \mathbb{E}^{\mathbb{Q}}\Big[{1}_{\{X_T \in \mathcal{A}\}}\Big( \underline{\phi} H_T + \underline{f}\int_{t}^{T} H_r \mathrm{d}r  \Big)\Big|\mathfrak{F}_t \Big]M_t
			\end{align*}
			Using the bound $|f_y'(s,\bold{X}_s)|\leq \Lambda (1+ |Z_s|^{\alpha})$ {and Young inequality $|z|^{\alpha} \leq c(1+|z|^2)$, for $\alpha \in (0,1)$} we deduce
			\begin{align*}
				\mathbb{E}\Big[ \psi_T  \phi'(X_T)\xi_T + \int_{t}^{T}\psi_r f_x'(r,\bold{X}_r)\xi_r \mathrm{d}r \Big|\mathfrak{F}_t \Big]
				\geq& \mathbb{E}^{\mathbb{Q}}\Big[ {1}_{\{X_T \in \mathcal{A}\}}\Big( \underline{\phi}^{\mathcal{A}} e^{-\int_0^{T}\int_{\mathbb{R}}\tilde b\left(u,z\right)L^{X^{x}}(\mathrm{d}u,\mathrm{d}z)+ \int_{0}^{T} f_y'(u,\bold{X}_u)\mathrm{d}u} \notag\\
				&+ \underline{f}\int_t^T e^{-\int_0^{r}\int_{\mathbb{R}}\tilde b\left(u,z\right)L^{X^{x}}(\mathrm{d}u,\mathrm{d}z)}e^{\int_{0}^{r} f_y'(u,\bold{X}_u)\mathrm{d}u}\mathrm{d}r \Big) |\mathfrak{F}_t\Big]M_t\\
				\geq  &   \mathbb{E}^{\mathbb{Q}}\Big[{1}_{\{X_T \in \mathcal{A}\}}e^{-\Lambda T}\Big( \underline{\phi}^{\mathcal{A}} e^{-\int_0^{T}\int_{\mathbb{R}}\tilde b\left(u,z\right)L^{X^{x}}(\mathrm{d}u,\mathrm{d}z)}e^{-\Lambda\int_{0}^{T} |Z_u|^2\mathrm{d}u} \notag\\
				&+ \underline{f} \int_t^T e^{-\int_0^{r}\int_{\mathbb{R}}\tilde b\left(u,z\right)L^{X^{x}}(\mathrm{d}u,\mathrm{d}z)}e^{-\Lambda\int_{0}^{T}|Z_u|^2\mathrm{d}u}\Big) |\mathfrak{F}_t\Big]M_t.
			\end{align*} 
			Provided that the law of $X_T$ is absolutely continuous with respect to the Lebesgue measure, and that $-\int_0^{T}\int_{\mathbb{R}}\tilde b\left(u,z\right)L^{X^{x}}(\mathrm{d}u,\mathrm{d}z)<\infty,\,\,\mathbb{Q}$-a.s.,  we deduce that $\mathbb{E}\Big( \psi_T  \phi'(X_T)\xi_T+ \int_{t}^{T}\psi_r f_x'(r,\bold{X}_r)\xi_r \mathrm{d}r \Big|\mathfrak{F}_t \Big) >0$. {The absolute continuity of the law of with respect to the Lebesgue measure follows from assumption. It then remains to show that:}
			$$
			-\int_0^{T}\int_{\mathbb{R}}\tilde b\left(u,z\right)L^{X^{x}}(\mathrm{d}u,\mathrm{d}z)<\infty,\,\,\mathbb{Q}\text{-a.s.}
			$$
			Thanks to Girsanov's theorem, the decomposition \eqref{eqslocalt1} and H\"older inequality we have
			\begin{align}\label{eqfor3}
				&	\mathbb{E}^{\mathbb{Q}}\Big[-\int_0^{T}\int_{\mathbb{R}}\tilde b\left(s,z\right)L^{X^{x}}(\mathrm{d}s,\mathrm{d}z)\Big]\notag\\
				=&\mathbb{E}\Big[-M_T\int_0^{T}\int_{\mathbb{R}}\tilde b\left(s,z\right)L^{X^{x}}(\mathrm{d}s,\mathrm{d}z)\Big]\leq\mathbb{E}\Big[M_T^2\Big]^{\frac{1}{2}}\mathbb{E}\Big[\Big(-\int_0^{T}\int_{\mathbb{R}}\tilde b\left(s,z\right)L^{X^{x}}(\mathrm{d}s,\mathrm{d}z)\Big)^2\Big]^{\frac{1}{2}}\notag\\
				\leq & C\mathbb{E}\Big[e^{ 2\int_{0}^{T}  \tilde b(s,W_s^x)\mathrm{d}W_s - \int_{0}^{T} (\tilde b(s,W_s^x))^2\mathrm{d}s }\Big(-\int_0^{T}\int_{\mathbb{R}}\tilde b\left(s,z\right)L^{W^{x}}(\mathrm{d}s,\mathrm{d}z)\Big)^2\Big]^{\frac{1}{2}}\notag\\
				\leq &C \mathbb{E}\Big[e^{ 2\int_{0}^{T}  \tilde b(s,W_s^x)\mathrm{d}W_s - \int_{0}^{T} (\tilde b(s,W_s^x))^2\mathrm{d}s }\Big(-
				\int_0^T \tilde b(s,W^{x}_s)\diffns W_s\notag\\
				&-\int_{0}^T \tilde b (T-s,\widehat{W}^{x}_s)\diffns B(s)+\int_{0}^T \tilde b (T-s,\widehat{W}^{x}_s)\frac{\widehat{W}_s}{T-s}\diffns s
				\Big)^2\Big]^{\frac{1}{2}}\notag\\
				\leq & C \mathbb{E}\Big[e^{ 4\int_{0}^{t}  \tilde b(s,W_s^x)\mathrm{d}W_s -2 \int_{0}^{t} (\tilde b(s,W_s^x))^2\mathrm{d}s }\Big]^{1/4}\Big\{\mathbb{E}\Big[\Big(
				\int_0^T \tilde b (s,W^{x}_s)\diffns W_s\Big)^4\Big]^{1/4}\notag\\
				&+\mathbb{E}\Big[\Big(\int_{0}^T \tilde b (T-s,\widehat{W}^{x}_s)\diffns B_s\Big)^4\Big]^{1/4}+\mathbb{E}\Big[\Big(\int_{0}^T \tilde b (T-s,\widehat{W}^{x}_s)\frac{\widehat{W}_s}{T-s}\diffns s
				\Big)^4\Big]^{1/4}\Big\}\notag\\
				\leq &C I_1(I_2+I_3+I_4).
			\end{align}

				Using Cauchy-Schwartz inequality, we have that 
				$$
				I_1\leq 
				\mathbb{E}\Big[e^{ 4\int_{0}^{t} \tilde  b(s,W_s^x)\mathrm{d}W_s -8 \int_{0}^{t} (\tilde b(s,W_s^x))^2\mathrm{d}s }\Big]^{1/8} \mathbb{E}\Big[e^{ 6 \int_{0}^{t} (\tilde b(s,W_s^x))^2\mathrm{d}s }\Big]^{1/8}.
				$$
				Applying Girsanov theorem to the martingale $4\int_{0}^{t}  \tilde b(s,W_s^x)\mathrm{d}W_s$, the first term in the bound above is equal to $1$. The boundedness of $\tilde{b}$ yields the boundedness of the second term.
				
				Using It\^o isometry (or Burkholder-Davis-Gundy), the boundedness of $\tilde b$ implies the boundedness of $I_2$, $I_3$ and $I_4$ (see for e.g. \cite{BMBPD17}). 
				The proof is completed. 
			\end{proof}
					%
			{\begin{remark}
					Another interpretation of the result in Theorem \ref{densY} is that the spatial derivative {$v'$ of $v$} (solution to quasi-linear PDE \eqref{eqmainP1}) is sucht that $v'(t,x) > 0 \text{ a.e.}$ (or $v'(t,x) < 0 \text{ a.e.}$). This follows from the fact that $D_tY_t$ is a continuous version of $Z_t = v'(t,x)$ for all $t\in (0,T)$ (see \cite[Remark 3.4]{F.Antonelli}).
					{We emphasize that the solution $v$ here is understood in the weak sense. In other words, we obtain conditions that guarantee the strict monotonicity of weak solutions of parabolic PDE by means of purely probabilistic tools. To our knowledge, such investigation has always been performed only for classical solution of parabolic or elliptic PDEs.}

					%
					
			\end{remark}} 
			\begin{remark}	
				We underline that the results obtained in this paper cover the case of Lipschitz continuous drivers with discontinuous drift. Hence, following the scheme employed above, one can prove for instance that Theorem 3.1 and Theorem 3.2 in \cite{Mastrolia} remain valid when $\sigma= 1$ and the drift is only bounded and measurable in $x$.
			\end{remark}	
			\section{Some applications}\label{section7}	The main aim in this section is to provide some applications of the results obtained in the previous sections of the paper.  We will first study the regularity of densities of solution to coupled FBSDEs with quadratic drivers under minimal conditions. 
			In particular, we exploit Theorem \ref{secondMal} to prove that the densities of  $X$, $Y$ and $Z$, are indeed H\"{o}lder continuous despite the non-differentiability of the drift coefficient. By considering a practical example from \cite{ChenMaYin17}, we will illustrate that the method derive here on existence and smoothness of densities, is also effective in that framework with the specificity of the discontinuity of the drift. 
			\subsection{Smoothness of densities}
			{We consider the following FBSDE
				\begin{equation}\label{aa}
					\begin{cases}
						X_t = x + \displaystyle \int_{0}^{t} b(s,X_s,Y_s,Z_s)\mathrm{d}s + W_t ,\\
						Y_t = \phi(X_T)+ \displaystyle \int_{t}^{T} f(s,X_s,Y_s,Z_s)\mathrm{d}s - \int_{t}^{T} Z_s\mathrm{d}W_s.
					\end{cases}
				\end{equation}
				Recall that under assumption {\bf (A6)} and if $\phi'\geq 0$, $\underline{f}\geq 0$ (resp. $\phi'\leq 0$ and $\underline{f}\leq 0$)  then for all $0< s\leq t <T,$ $D_sY_t \geq 0$ (resp. $D_sY_t \leq 0$ ). Moreover, if there is a Borel set $\mathcal{A}$ such that $\mathbb{P}(X_T\in \mathcal{A}/\mathfrak{F}_t)>0$ and  $\phi'_{|\mathcal{A}}>0$ (resp. $\phi'_{|\mathcal{A}}<0$) then $D_sY_t > 0$ (resp. $D_sY_t< 0$ ). In addition, if {\bf (A7)} and {\bf (A8)} are also valid, then the processes $X,Y$ and $Z$ solutions to \eqref{aa} are twice Malliavin differentiable.}
			
			{Let $(G_k)_{0\leq k\leq n}$ be a random variable such that $G_0= 1$ and for all $0\leq s\leq t \leq T$
				\[ G_k = \delta\Big( G_{k-1}\Big(\int_0^T D_sX_t\mathrm{d}s\Big)^{-1}\Big), \text{ if } 1\leq k \leq n+1, \]
				with $\delta$ denoting the Skorohod integral introduced in \eqref{Skoro}.
				From \cite[P.115]{Nua06}, if $X_t\in \mathbb{D}^{1,2}(\Omega)$, $\int_0^T D_sX_t\mathrm{d}s \neq 0$ $\mathbb{P}\text{-a.s.}$ and $G_{k}\Big(\int_0^T D_sX_t\mathrm{d}s\Big)^{-1}$ belongs to $\text{Dom}(\delta)$ then the probability density $\rho_{X_t}$ of the forward equation $X_t$ is of class $C^{n}(\mathbb{R})$ and \begin{align}
					\frac{d^n}{dx^n}\rho_{X_t}(x)= (-1)^n\mathbb{E}\left[1_{\{X_t >x\}}G_{n+1}\right]
			\end{align}}
			
			{In the above context, the Malliavin derivative of $X_t$ is explicitly given by
				\[ D_sX_t = \exp\Big( \int_s^t \tilde b'(u,X_u)\mathrm{d}u \Big), \text{ for all } 0\leq s \leq t,\]
				where $\tilde b(t,x) = b(t,x,v(t,x),v'(t,x))$ and $\tilde b'$ denotes the weak derivative of $\tilde b$ with respect to the variable $x$. Hence for any $t \in (0,T]$ there is $\epsilon > 0$ such that $\int_0^T D_sX_t \mathrm{d}s \geq \epsilon >0.$ Using the smoothness of the function $y\mapsto y^{-1}$ on $(\epsilon,\infty)$ and the fact that $X_t \in \mathbb{D}^{2,2}(\Omega)$ we deduce that $F:= \Big(\int_0^T D_sX_t\mathrm{d}s\Big)^{-1} \in \mathbb{D}^{1,2}(\Omega)$, so $F \in \text{Dom}(\delta)$ and $G_1 = FW(T) - \int_0^T D_sF\mathrm{d}s$ which implies that $G_1 \in \mathbb{D}^{1,2}(\Omega)$, then $G_1F \in \mathbb{D}^{1,2}(\Omega)$ and $G_1F \in \text{Dom}(\delta)$.}
			
			{Similarly, let $(H_k)_{0\leq k\leq n}$ be a random variable such that $H_0= 1$ and for all $0\leq s\leq t \leq T$
				\[ H_k = \delta\Big( H_{k-1}\Big(\int_0^T D_sY_t\mathrm{d}s\Big)^{-1}\Big), \text{ if } 1\leq k \leq n+1. \]
				Since for all $0\leq s\leq t$, $D_sY_t > 0$ $\mathbb{P}\text{-a.s.}$, we deduce that, there is a constant $\epsilon_1>0$ such that $\int_0^T D_sY_t\mathrm{d}s \geq \epsilon_1$. Using now the smoothness of the function $y\mapsto y^{-1}$ on $(\epsilon_1,\infty)$ and the fact that $Y_t \in \mathbb{D}^{2,2}(\Omega)$ we deduce that $F_1:= \Big(\int_0^T D_sY_t\mathrm{d}s\Big)^{-1} \in \mathbb{D}^{1,2}(\Omega)$ so $H_1F_1 \in \text{Dom}(\delta)$.}
			
			{We have shown that 
				\begin{prop}
					Under {\bf(A6)},{\bf(A7)},{\bf(A8)} and assuming further that either (A+) or (A-) is valid. Then, the densities $\rho_{X_t}$ and $\rho_{Y_t}$ of the processes $X_t$ and $Y_t$ solution to \eqref{aa}, respectively belong to $C^0(\mathbb{R})$ such that 
					\begin{align*}
						\rho_{X_t}(x) &= \mathbb{E}\left( 1_{\{X_t >x\}} G_1  \right)\\
						\rho_{Y_t}(x) &= \mathbb{E}\left( 1_{\{Y_t >x\}} H_1  \right).
					\end{align*}
					Moreover, if  the assumptions of Theorem \ref{densZ} are valid then for all $0\leq s\leq t$, $D_sZ_t >0$ $\mathbb{P}\text{-a.s.}$ and the density $\rho_{Z_t}$ of the solution process $Z$ to \eqref{aa} belongs to $C^0(\mathbb{R})$ and 
					\begin{align*}
						\rho_{Z_t}(x) = \mathbb{E}\left( 1_{\{Z_t >x\}} I_1  \right),
					\end{align*}
					where $I_1 =\delta\Big( \Big(\int_0^T D_sZ_t\mathrm{d}s\Big)^{-1}\Big)$.
			\end{prop}}
			
			{The following result improves the regularity of the densities of solutions $X,Y$ and $Z$ to the system \eqref{aa} obtained above. The method consists to investigate the integrability property of the Malliavin covariance matrix of the different processes (see Lemma \ref{lem73} below).} 
			
			{We recall that for any random variable $F$, the matrix $\Gamma_{F}=(\Gamma_{F}^{ij})_{i,j=1,\ldots,d}$,  given by 
				$\Gamma_{F}^{ij}:= \langle DF^{i}, DF^{j}\rangle_{\mathcal{H}^2} $ 
				is said to be non-degenerate if for all $p\geq 1$
				\begin{equation}\label{nondeg}
					\mathbb{E}\Big[\left(\Gamma_{F}\right)^{-p}\Big] < \infty.  
			\end{equation}}
			
			{\begin{lemm}[Proposition 23 in \cite{BallyCaramellino11}]\label{lem73}
					Let $ F=(F^1,\cdots,F^d)$ with $F^1,\cdots,F^d\in \bigcap_{p\geq 1}\mathbb{D}^{k+1,p}$, such that $\Gamma_F$ satisfies \eqref{nondeg}. 
					Then $\rho_{F} \in C^{k-1,\beta}(\mathbb{R})$ for some $k \geq 1$ and $\beta \in (0,1)$ i.e. $\rho_F$ is $k-1\text{-times}$ differentiable with H\"{o}lder continuous derivatives of exponent $\beta <1$.
				\end{lemm}
				We will first prove that the densities of $X_t$ and $Y_t$ are H\"older continuous. 
				\begin{prop}\label{prop 5.13}
					Let assumption {\bf(A6)},{\bf(A7)} and {\bf(A8)} be in force. {Suppose in addition that} there is $\mathcal{A}\in \mathcal{B}(\mathbb{R})$ such that $\mathbb{P}(X_T \in \mathcal{A}|\mathfrak{F}_t) > 0$ and one of the assumptions (A+) or (A--) holds.  Then, $\rho_{X_t} \in C^{0,\beta}(\mathbb{R})$ for all $t \in [0,T]$ and for all $t \in [0,T)$, $Y_t$ has a $\beta\text{-H\"older}$ continuous density $\rho_{Y_t}$ given by 
					\[ \rho_{Y_t}(x) = \rho_{Y_t}(x_0) \exp\Big( -\int_{x_0}^x w_{Y_t}(z)\mathrm{d}z\Big), \, x \in \supp \rho_{Y_t},  \]
					where $x_0$ is a point in the interior of $\supp (\rho_{Y_t})$ and 
					\[ w_{Y_t}(z):= \mathbb{E}\Big[\delta\Big(\int_0^T D_sY_t\mathrm{d}s\Big)^{-1}\Big| Y_t = z\Big], \]
					$\delta$ denotes the Skorohod integral introduced in \eqref{Skoro}.
				\end{prop}
				\begin{proof}
					From  Lemma \ref{bounds}, we know that $\tilde{b}$ is uniformly bounded and Lipschitz continuous in its spatial variable for all $t \in [0,T)$. Then, {using} \cite[Proposition 4.4]{BanosNilssen16} the following holds
					\begin{align*}
						\Big(\det \Gamma_{X_t}\Big)^{-1} \in \bigcap_{p\geq 1} L^p(\Omega).
					\end{align*} 
					{This} implies that $\rho_{X_t} \in C^{0,\beta}(\mathbb{R})$ for all $t \in [0,T)$ with $\beta \in (0,1)$. On the other hand, for all $t\in[0,T)$ we deduce
					\begin{align*}
						\det \Gamma_{Y_t} = (v'(t,X_t))^2\det \Gamma_{X_t} \geq \alpha^{2}(t) \det \Gamma_{X_t},
					\end{align*}
					where $\alpha(t)$ is the positive function such that $v'(t,X_t)\geq \alpha(t)$ or $v'(t,X_t)\leq -\alpha(t)$ $\mathbb{P}\text{-a.s.}$ (see proof of Theorem \ref{prop 4.10} ). Without loss of generality, we assume that there is a constant $\epsilon_1>0$ such that $\alpha(t) \geq \epsilon_1$.
					Hence, $\Big(\det \Gamma_{Y_t}\Big)^{-1} \leq \alpha^{-2}(t)\Big(\det \Gamma_{X_t}\Big)^{-1} \in \bigcap_{p\geq 1} L^p(\Omega)$. For the explicit representation of the density it follows by combining arguments of the proof of \cite[Proposition 2.1.1]{Nua06} and \cite[Proposition 2]{Dung20}. This concludes the proof.
			\end{proof}}
			{Let us turn now to the case of the control process $Z$.
				\begin{prop}\label{prop 5.14}
					Assume that there is $\mathcal{A}\in \mathcal{B}(\mathbb{R})$ such that $\mathbb{P}(X_T\in \mathcal{A}/\mathfrak{F}_t)>0$ and let assumptions of Theorem \ref{densZ} be in force. Then the control process $Z$ solution to \eqref{aa} has a density $\rho_{Z_t}$ which is H\"older continuous  and the following explicit representation holds
					\[ \rho_{Z_t}(z) = \rho_{Z_t}(x_0) \exp\Big( -\int_{z_0}^z w_{Z_t}(u)\mathrm{d}u\Big), \, z \in \supp \rho_{Z_t},  \]
					where $z_0$ is a point in the interior of $\supp (\rho_{Z_t})$ and 
					\[ w_{Z_t}(u):= \mathbb{E}\Big[\delta\Big(\int_0^T D_sZ_t\mathrm{d}s\Big)^{-1}\Big| Z_t = u\Big], \]
					$\delta$ denotes the Skorohod integral introduced in \eqref{Skoro}.
				\end{prop}
				\begin{proof} We recall that for all $t\in[0,T)$ $Z_t = v'(t,X_t)$ $\mathbb{P}\text{-a.s.}$ Hence
					\begin{align*}
						\det \Gamma_{Z_t} = (v''(t,X_t))^2\det \Gamma_{X_t} \geq \omega(t)^{2} \det \Gamma_{X_t},
					\end{align*}
					where $\omega(t)$ is the positive function such that $v''(t,X_t)\geq \omega(t)$ or $v''(t,X_t)\leq -\omega(t)$ $\mathbb{P}\text{-a.s.}$ (see proof of Corollary \ref{cor520} ).	Then, for all $p\geq 1$
					\[  \mathbb{E}\left[ \Big(\det \Gamma_{Z_t}\Big)^{-p}  \right] \leq \omega^{-2p}(t) \mathbb{E}\left[ \Big(\det \Gamma_{X_t}\Big)^{-p}  \right] < \infty. \]
			\end{proof}}
			
			\subsection{Regime-switching term structure model}
			Here we consider the following "regime-switching" term structure model from J. Chen {\it et al.}(\cite{ChenMaYin17}). Such a model is an extension of the classical term structure model and we briefly describe it below.
			
			Let $r:= \{r_t: t\geq 0\}$ be the short rate process and $X$ be the process defined by $r_t =  h(X_t),$ $t\geq 0$, where $h$ is a smooth function. The dynamics of $X_t$ is then given by
			\begin{equation}
				\mathrm{d}X_t = \kappa(\beta_t - X_t)\mathrm{d}t + \mathrm{d}W_t,
			\end{equation}
			where $W$ stands for the standard Brownian motion. Many different dynamics of the short rate models have been proposed in the literature. Among them, one assumed that the mean reversion $\beta$ has a hidden Markovian structure that is, $\beta$ can be written as a functional of an exogenous factor, that we denote by $Y$ i.e. $\beta_t = k(Y_t)$, where $k(y) \in \{k_1, k_2 \}$. If we consider the process $Y$ as a long term interest rate or to be the price of a long term bond to be comparable with the console price and assume further that $X$ and $Y$ are related by an explicit  relation at the maturity time $T$ such that $Y_T= \phi(X_T)$, with $\phi$ any smooth function. Then, using the arguments developed for instance in \cite{DuffieMaYong95}, one obtains that $X$ and $Y$ are linked by the following FBSDE 
			\begin{equation}\label{example1}
				\begin{cases}
					\mathrm{d}X_t = \left(k(Y_t)-\beta X_t\right)\mathrm{d}t + \mathrm{d}W_t\\
					\mathrm{d}Y_t = \left(h(X_t)Y_t-1\right)\mathrm{d}t + Z_t \mathrm{d}W_t\\
					X_0= x, \quad Y_T = \phi(X_T),
				\end{cases}
			\end{equation}
			where $k(y) = k_1 {\bf 1}_{y\leq \alpha} + k_2 {\bf 1}_{y> \alpha}$, and $\alpha>0$ represents the triggering level of the process $Y$. It is important to underline that the drift of the forward equation in \eqref{example1} exhibits a  discontinuity in $y$ which is beyond most of the models studied in the literature.
			
			The FBSDE \eqref{example1} falls in the realm of the following system
			\begin{equation}\label{example2}
				\begin{cases}
					\mathrm{d}X_t = b(t,X_t,Y_t)\mathrm{d}t +  \mathrm{d}W_t\\
					\mathrm{d}Y_t = -g(t,X_t,Y_t,Z_t)\mathrm{d}t + Z_t \mathrm{d}W_t\\
					X_0= x, \quad Y_T = \phi(X_T),
				\end{cases}
			\end{equation}
			for which the drift $b$ is of the form
			\begin{equation}\label{example3}
				b(t,x,y):= \sum_{i=1}^{m}b^{i}(t,x){\bf 1}_{[c_i,c_{i+1}]}(y),\quad (t,x,y)\in [0,T]\times\mathbb{R}\times\mathbb{R},
			\end{equation}
			where $-\infty <c_1<c_2<\cdots< \infty$ is a finite partition of $\mathbb{R}$ and $b^{i}$ are deterministic Lipschitz functions. This class of FBSDE \eqref{example2} was investigated for example in \cite{ChenMaYin17} and it was shown that under their standing assumptions (H1)-(H4) (see  \cite[Page 4]{ChenMaYin17}), the system \eqref{example2} admits a unique strong solution with
			$Y_t = \hat{u}(t,X_t)$ and $Z_t= \nabla_x \hat{u}(t,X_t)$, where $\hat{u}$ stands for the solution in the distribution sense to the associated quasi-linear PDE to the system \eqref{example2} (see \cite[Theorem 5.1]{ChenMaYin17}). The decoupled function $\hat{u}$ is uniformly bounded and Lipschitz continuous for all $t \in [0,T)$.
			
			Even though the well posedeness of the class of FBSDEs \eqref{example2} with drifts satisfying \eqref{example3} is not treated in Section \ref{solvability} ,  the main goal in this subsection is to prove that the results on the analysis of densities developed in Section \ref{section5} remain valid for this variety of coupled FBSDEs with drifts that are discontinuous in their backward component and which is motivated by an interesting practical application.

		We first study the regularity of solutions to the FBSDE \eqref{example2} since this was not done in \cite{ChenMaYin17}.
			
			\begin{prop}\label{propo65}
				Let the conditions of Theorem 5.1 in \cite{ChenMaYin17} be in force. Assume that the driver $f$ and the terminal value $\phi$ are continuously differentiable with bounded derivatives. Then, the solution $(X,Y,Z)$ to  \eqref{example2} is Malliavin differentiable and for all $0\leq s \leq t \leq T$,  $(D_sY_t,D_sZ_t)$ satisfies  \eqref{MalY} and $D_sX_t$ is given by \eqref{malderX}.
				
				In addition, if $f$ and $\phi$ are twice continuously differentiable with bounded derivatives, then the triple $(X,Y,Z)$ is twice Malliavin differentiable for all $t\in [0,T)$.
			\end{prop}
			\begin{proof}
				The assumptions of \cite[Theorem 5.1]{ChenMaYin17} guarantee the existence of a unique strong solution to the system \eqref{example2}. Then, \eqref{example2} can be written as follows
				\begin{equation}\label{example4}
					\begin{cases}
						X_t = x + \int_0^t \bar {b}(s,X_s)\mathrm{d}s + W_t\\
						Y_t = \phi(X_T) + \int_t^Tg(s,X_s,Y_s,Z_s)\mathrm{d}s + \int_t^T Z_s \mathrm{d}W_s
					\end{cases}
				\end{equation}
				where $\bar{b}(t,x) = b(t,x,\hat{u}(t,x))$ is uniformly bounded for all $t\in [0,T]$. Hence from Proposition \ref{prop 3.4}, the solution $(X,Y,Z)$ is Malliavin differentiable. On the other hand, using the fact $\bar{b}$ is Lipschitz continuous for all $t\in[0,T)$ and assuming further that $f$ and $\phi$ are twice continuously differentiable with bounded derivatives, we can invoke Theorem \ref{secondMal} to conclude that the triple $(X,Y,Z)$ is twice Malliavin differentiable for all $t\in [0,T)$ .
			\end{proof}
			We have the subsequent Theorem as a combination of the results obtained above.
			\begin{thm} Let Assumptions of Proposition \ref{propo65} hold. 
				If there is a Borel set $\mathcal{A}$ such that $\mathbb{P}(X_T \in \mathcal{A}|\mathfrak{F}_t)>0$ and one of (A+) or (A-) is valid. Then, both processes $X_t$ and $Y_t$ solution to \eqref{example2} admit, respectively H\"older continuous densities $\rho_{X_t}$ and $\rho_{Y_t}$.
				
				In addition, if the conditions of Theorem \ref{densZ} are in force. Then the control component $Z$ solution to FBSDE \eqref{example3} has an absolute continuous law which is H\"older continuous, for all $t\in [0,T).$
			\end{thm}
			
			\begin{proof}
				Thanks to the results derived from Propositions \ref{propo65}, one can follow the same strategy as in the proof od Propositions \ref{prop 5.13} and \ref{prop 5.14} to derive the desired result. This ends the proof.
			\end{proof}
			
			\begin{remark}[Numerical analysis of FBSDEs with rough drifts]
	It is well known that, closed form solutions to FBSDEs are not always available, for instance for FBSDEs of practical relevance as the one described in \eqref{example1}. Hence, the numerical analysis of FBSDEs stands as a powerful tools to develop computational methods to efficiently approximate the solutions of these equations. We refer for instance to the recent survey \cite{Chessarietal23} on numerical method for BSDEs and FBSDEs. However, a review of all the methods listed in \cite{Chessarietal23}, shows that only numerical approximation of BSDEs coupled with smooth-type SDEs has been considered so far.  As far as we know, none of the numerical methods reviewed in \cite{Chessarietal23} provide an efficient scheme for the solution to \eqref{example1}. One reason could be non-smoothness of the drift coefficients. The non -smoothness of drift is one of the motivations of the current work. As such we believe that the results obtained could be a first step toward studying numerical schemes for FBSDEs with non-Lipschitz drift.
			\end{remark}

			\subsection{Pricing and hedging of derivatives based on non-tradable assets with discontinuous drift}
			In this section, we discuss the pricing and hedging of financial derivatives based on non-tradable assets, but which are correlated with tradable assets. We assume that the logarithm of the non-tradable asset has a discontinuous drift. The pricing of such derivatives is done via  exponential utility-based indifference price. As we will see, this problem can be reduced to solving a quadratic FBSDE with nonsmooth drift coefficient. 
			
			We summarise the model presented in \cite{DosReis}. Let $\{W_t\}_{t\geq 0}$ be a one dimensional Brownian motion on a probability space $(\Omega, \mathfrak{F}, \mathbb{P})$ and $\{\mathfrak{F}_t\}_{t\in[0,T]}$ be its natural filtration enlarged in the usual way by the $\mathbb{P}$-null sets. Here we consider a modified dynamic of the non-tradable asset. As in \cite{BMBPD17}, we suppose an extended Black and Scholes model where the non-tradable stock pays a dividend yield that switches to a higher level when the stock value passes a certain threshold. More precisely, we suppose that the logarithm of the non-tradable index has the following dynamics
			\begin{equation}\label{SDEfact1}
				\mathrm{d}X_t = -(\lambda1_{(-\infty,R)}(X^x_t)+\hat \lambda1_{[R,\infty)}(X^x_t)+\frac{\sigma^2}{2})\mathrm{d}t + \sigma \mathrm{d}W_t,\,\,\,X^x_0=x\in \mathbb{R},
			\end{equation}
			where $\lambda,\hat \lambda\in \mathbb{R}_+$ are the dividend yields and $R\in  \mathbb{R}$ is the given threshold. Note that the drift $b(x)=-(\lambda_11_{(-\infty,R)}(x)+\lambda_21_{[R,\infty)}(x)+\frac{\sigma^2}{2})$ is bounded and measurable. Notice that such a model cannot be covered by the work \cite{AnkImkPo,ArImDR10,DosReis}. The above SDE \eqref{SDEfact1} has a unique strong Malliavin differentiable solution (see \cite{MMNPZ13}).  In addition this solution as a Sobolev differentiable flows (see \cite{NilssenProske}).

			Suppose that an economic agent has expenses at time $T>0$ of the form $F(X_T )$, where $F:  \mathbb{R}\rightarrow  \mathbb{R}$ is bounded and measurable function. At time $t \in [0, T ]$, the expected payoff
			of $F(X_T)$, conditioned on $X_t = x \in \mathbb{R}$, is given by $F(X^{t,x}_T)$, where $X^{t,x}_\cdot$ is the solution of the SDE is the \eqref{SDEfact1} starting at $x$ at time $t$.
			
			We assume that the financial market is made of two assets: a non risky asset representing the numeraire and a risky asset in units of the numeraire whose dynamics is given by the following SDE 
			\begin{equation}\label{SDEfact2}
				\mathrm{d}S_t =S_t\big(\alpha(t,X_t)\mathrm{d}t+\mathrm{d}W_t\big),\,\,\,S_0=s_0\in \mathbb{R},
			\end{equation}
			where the $\alpha$ is measurable function. We assume that $\alpha$ is uniformly bounded and continuously differentiable in the space variable with bounded derivative.
			
			We denote by $U$ the exponential utility function with risk aversion coefficient $\mu>0$ that is
			$$
			U(x)=-e^{-\mu x}, \,x \in \mathbb{R}.
			$$
			Let $\pi$ be an investment strategy that is a real-valued predictable process $\pi$ such that the integral $\int_0^t\pi_u\frac{\mathrm{d}S_u}{S_u}$ is well defined. $\pi$ can be seen as the fraction of portfolio invested in the risky asset. It is assumed that $\pi_t(\omega)\in \tilde C\subset \mathbb{R}$, a closed (not necessarily convex). For such a trading strategy, the wealth process given $X_t^{t,x}=x$ satisfies
			\begin{equation}\label{SDEfact3}
				G^{\pi,t,x}_s= \int_t^s\pi_u\big\{\alpha(u,X^{t,x}_u)\mathrm{d}u+\mathrm{d}W_u\big\},\,0\leq t\leq s\le T,\,\,x\in \mathbb{R}.
			\end{equation}
			A strategy $\pi$ with values in $\tilde C$ is called admissible if it satisfies
			$$
			\mathbb{E}\big[\int_t^T|\pi_u\beta(u,X^{t,x}_u)|^2\mathrm{d}u\big]<\infty
			$$
			and if the family
			$$
			\{e^{-\mu G^{\pi,t,x}_\tau}: \tau \text{ is a stopping time with values in } [t,T]\}
			$$
			is uniformly integrable. We denote by $\mathcal{A}$ the set of admissible strategies. Notice that the above criteria is similar to that of \cite[Section 2]{HuImkMuller}. Recall that for a closed subset of $\mathbb{R}$ and $a\in \mathbb{R}$, the distance between $a$ and $C$ is given by
			$$
			\text{dist}_C(a)=\underset{b\in C}\min|a-b|
			$$
			We also denote by $\Pi_C(a)$ the set of elements of $C$ at which the minimum is attained that is:
			$$
			\Pi_C(a)=\{b\in C: |a-b|=\text{dist}_C(a)\}.
			$$

			We will assume in addition that $C$ is convex. Again, we are interested in a replicating and pricing problem of an economic agent who wishes to buy a contingent claim that pays off $F=F(X_T)$ at time $T>0$. We suppose $F:\mathbb{R}\rightarrow \mathbb{R}$ is Lipschitz continuous and bounded. The optimal strategy is chosen to maximise the expected utility of the terminal wealth and the initial price of the claim can be derived via utility indifference pricing. The utility indifference price is as follows: for a utility function $U:\mathbb{R}\rightarrow \mathbb{R}$, the agent with initial value of the non tradable asset $x$ and initial wealth $\nu$ and no endowment of the claim will simply face the problem of maximizing her expected utility of the terminal wealth $G^{\pi, t,x}_T$; that is
			\begin{equation}\label{expouti1}
				V^0(t,\nu,x)=\sup_{\pi\in \mathcal{A}}\mathbb{E}\big[U\big(\nu+G^{\pi, t,x}_T\big)\big]=\mathbb{E}\big[U\big(\nu+G^{\hat{\pi}, t,x}_T\big)\big],
			\end{equation}
			where $\hat{\pi}$ is an optimal control (if it exists). The agent with initial wealth $x$ and who is willing to pay $p$ today for a unit of claim $F$ at time T faces the following expected utility maximization problem
			\begin{align}
				\label{expouti2}
				V^F(t,\nu+p,x)=&\sup_{\pi\in \mathcal{A}}\mathbb{E}\big[U\big(\nu+p+G^{\pi, t,x}_T-F(X_T)\big)\big] =\mathbb{E}\big[U\big(\nu+p+G^{\hat{\pi}, t,x}_T-F(X_T)\big)\big].
			\end{align}
			According to the utility indifference pricing principle, the \textsl{fair price} of the claim with payoff $F(X_T)$ at time $T$ is the solution to the equation
			\begin{equation}
				\label{expouti13}
				V^{F}(t,\nu+p,x)=V^0(t,\nu,x).
			\end{equation}
			
			For $U(x)=-e^{-\gamma x},\,\,\gamma>0$, the exponential utility maximization problem \eqref{expouti1} (respectively \eqref{expouti2}) can be solved (see \cite{HuImkMuller}) by finding the generator $f$ of the BSDE
			\begin{align}\label{expoBSDE1}
				Y^{t,x}_s=0-\int_s^Tf(u,X^{t,x}_u,Z^{t,x}_u)\mathrm{d}u -\int_s^T Z^{t,x}_u \mathrm{d}W_u,
			\end{align}
			respectively
			\begin{align}\label{expoBSDE2}
				\hat Y^{t,x}_s=F(X^{t,x}_T)-\int_s^Tf(u,X^{t,x}_u,\hat{Z}^{t,x}_u)\mathrm{d}u -\int_s^T\hat{Z}^{t,x}_u \mathrm{d}W_u,
			\end{align}
			such that the process $R_t^\pi=-\exp(-\gamma (G^{\pi, t,x}_t-Y_t)), \,t \in [0,T]$ (respectively $\hat R_t^\pi=-\exp(-\gamma (G^{\pi, t,x}_t-\hat Y_t)), \,t \in [0,T]$), is a supermartingale for all strategies $\pi$ and a martingale for a particular strategy $\pi^o$. Following similar steps as in \cite{ArImDR10}(see also \cite[page 1697-1698]{HuImkMuller}), the generator $f$ is given by
			\begin{align}\label{drivopcon1}
				f(t,x,z)=-\frac{\gamma}{2}\text{dist}^2_C(z+\frac{\alpha(t,x)}{\gamma})+z\alpha(t,x)+\frac{1}{2\gamma}\alpha^2(t,x).
			\end{align}
			The boundedness of $\alpha$ yields the existence of a positive constant $C$ suct that for all $(t,x,z)\in [0,T]\times \mathbb{R}^2$
			$$
			f(t,x,z)\leq C(1+|z|^2).
			$$
			In addition for all $t\in [0,T],x,x'\in  \mathbb{R}, z,z'\in  \mathbb{R}^2$ we have (see \cite[page 85-86]{DosReis})
			$$
			|f(t,x,z)-f(t,x,z')|\leq C(1+|z|+|z'|)|z-z'|.
			$$
			Hence, the associated semi-linear PDE to the decoupled  FBSDE \eqref{SDEfact1}-\eqref{expoBSDE1} is given by 
			\begin{align}\label{expoPDE1}
				\begin{cases}
					\frac{\partial v}{\partial t}+ \frac{1}{2}\sigma^2\frac{\partial^2 v}{\partial x^2} + b(x)\frac{\partial v}{\partial x} - f(t,x,\nabla_x v)=0,\\
					v(T,x)= 0.
				\end{cases}
			\end{align}
			Similarly, for the decoupled FBSDE \eqref{SDEfact1}-\eqref{expoBSDE2}:
			\begin{align}\label{expoPDE2}
				\begin{cases}
					\frac{\partial \hat{v}}{\partial t}+ \frac{1}{2}\sigma^2\frac{\partial^2 \hat{v}}{\partial x^2} + b(x)\frac{\partial \hat{v}}{\partial x} - f(t,x,\nabla_x \hat{v})=0,\\
					\hat{v}(T,x)= -F(x).
				\end{cases}
			\end{align}
			We notice that, the semi-linear PDEs \eqref{expoPDE1} and \eqref{expoPDE2} are both particular cases of the quasi-linear PDE \eqref{eqmainP1}. Indeed, observe that the coefficients $b, \sigma$ and  $f$ considered in this section, fall into the setup of Assumption \ref{assum1} and precisely $b(t,x,y,z):= b(x)$, $\sigma(t,x,y)= \sigma$, $f(t,x,y,z):= f(t,x,z)$ and $\ell$ is a positive constant. Thus, from Theorem \ref{th3.3}, the  FBSDE \eqref{SDEfact1}-\eqref{expoBSDE2} (resp. FBSDE \eqref{SDEfact1}-\eqref{expoBSDE1}) has a unique strong Malliavin and Sobolev differentiable solution $(X,\hat{Y},\hat{Z}) \in \mathcal{S}^2\times \mathcal{S}^{\infty}\times \mathcal{H}_{BMO}$ (resp. $(X,Y,Z) \in \mathcal{S}^2\times \mathcal{S}^{\infty}\times \mathcal{H}_{BMO}$). In particular, the PDE \eqref{expoPDE2}(resp. PDE \eqref{expoPDE1}) has a unique solution $\hat{v}(t,x) \in W^{1,2}([0,T]\times\mathbb{R},\mathbb{R})$ such that $\hat{v}$ is bounded and Lipschitz continuous and $\hat{v}'_x$ is continuous (${v}(t,x) \in W^{1,2}([0,T]\times\mathbb{R},\mathbb{R})$ such that ${v}$ is bounded and Lipschitz continuous and ${v}'_x$ is continuous) and the following hold $\mathbb{P}\text{-a.s. } \text{for all } s \in [t,T]$
			\begin{align}\label{620}
				\hat{Y}^{t,x}_s=\hat{v}(s,X^{t,x}_s) \text{ and } \hat Z^{t,x}_s=\hat v'(s,X^{t,x}_s)\sigma,
			\end{align}
			respectively,
			\begin{align}\label{621}
				Y^{t,x}_s=v(s,X^{t,x}_s) \text{ and }  Z^{t,x}_s= v'(s,X^{t,x}_s)\sigma.
			\end{align}
			\begin{remark} We notice that the driver $f$ given by \eqref{drivopcon1} is continuously differentiable in $x$ and $z$ and does not depend on $y$. Then, an inspection of the proof of \cite[Theorem 5.12]{ImkRhossOliv} enables us to conclude that even though the drift of the forward is merely bounded and measurable, the Malliavin derivatives of $(Y,Z)$ (resp. $(\hat{Y},\hat{Z}))$ solution to \eqref{expoBSDE1} (resp. \eqref{expoBSDE2}) satisfy a similar linear equation given by \eqref{MalY}.
			\end{remark}
			The value function and the optimal investment are given below
			
			\begin{thm}\label{thoptcont1}
				The value function of the optimal control problem \eqref{expouti1} (respectively \eqref{expouti2}) is given by 
				$$
				V^0(t,\nu,x)=\exp\{-\alpha(x-Y^{t,x}_t)\},
				$$
				\textup{(}respectively $V^F(t,\nu,x)=\exp\{-\alpha(x-\hat Y^{t,x}_t)\}$\textup{)} where $Y^{t,x}_t$ \textup{(}respectively $\hat Y^{t,x}_t$\textup{)} is defined by the unique solution of the BSDE \eqref{expoBSDE1} \textup{(}respectively \eqref{expoBSDE2}\textup{)}, with $f$  given by \eqref{drivopcon1}. Moreover the optimal strategy $ \pi \in \mathcal{A}$ for the problem \eqref{expouti1} on $[t,T]$ satisfies
				\begin{align}\label{optstrat1}
					\pi^\ast_s= \Pi_{C(s,X_s^{t,x})}(Z^{t,x}_s+\frac{1}{\gamma}\alpha(s,X_s^{t,x})), \,\,\, 
				\end{align}
				Similar  equality holds for the optimal strategy $\hat \pi^\ast$ of problem \eqref{expouti2} with $\hat Z^{t,x}_t$ instead of $Z^{t,x}_t$. Hence, using the linearity of the operator $\Pi_{C(s,X_s^{t,x})}$, the derivative hedge is given for all $s\in [t,T]$ by
				\begin{align}
					\Delta_s = \Pi_{C(s,X_s^{t,x})}(\hat Z^{t,x}_s- Z^{t,x}_s).
				\end{align}
			\end{thm}
			\begin{proof}
Follows as in \cite{ArImDR10, HuImkMuller}.
\end{proof}
			Below, we deduce a representation of the indifference price.
			\begin{thm}\label{thoptcont2}
				There exists a deterministic function  $p:[0,T]\times \mathbb{R} \rightarrow \mathbb{R}$ such that for all $\nu\in \mathbb{R}$ and $(t,x) \in [0,T]\times \mathbb{R}$, it holds
				\begin{align}\label{eqindiffpri1}	
					V^{F}(t,\nu+p(t,x),x)=V^0(t,\nu,x).
				\end{align}
				$p$ can be written as
				$$
				p(t,x)=v(t,x) -\hat v(t,x) .
				$$
				where $v$ and $\hat{v}$ are given by \eqref{621} and \eqref{620}, respectively.
			\end{thm}
			
				\begin{proof}
				Follows as in \cite{ArImDR10}.
			\end{proof}
			We can also formulate the above problem in a more dynamic way as follows: Let $\tau$ be a stopping time and let $G_\tau$ be $\mathcal{F}_\tau$-measurable random variable describing the wealth at time $\tau$. Let us now consider the problem
			\begin{align}
				\label{expouti4}
				V^F(\tau,G_\tau)=&\underset{\pi\in \mathcal{A}}\essup\,\mathbb{E}\big[U\big(G_\tau+G^{\pi, t,x}_T-F(X_T)\big)|\mathcal{F}_\tau\big] =\mathbb{E}\big[U\big(G_\tau+G^{\hat{\pi}, t,x}_T-F(X_T)\big)|\mathcal{F}_\tau\big].
			\end{align}
			The following dynamic programing principle holds:
			
			\begin{cor}\label{thoptcont3}
				
				The value function of the optimal control problem \eqref{expouti4}  is given by 
				$$
				V^F(\tau,G_\tau)=\exp\{-\alpha(G_\tau-\hat Y^{0,x}_\tau)\},
				$$
				$\hat Y^{t,x}_t$ is the solution of the BSDE  \eqref{expoBSDE2}.  In addition 	$\hat{\pi}$ is given by \eqref{optstrat1}. Moreover
				$$
				V^F(\tau,G_\tau-p(\tau,X^{0,x}_\tau))=V^0(\tau,G_\tau),
				$$
				where $p$ is the price given by:
				$$
				p(\tau,X^{0,x}_\tau)=Y^{0,x}_\tau-\hat Y^{0,x}_\tau.
				$$
				
			\end{cor}
			\begin{proof}
				The proof simply follows by \cite[Propostion 9]{HuImkMuller}.
			\end{proof}
			
			Given the weaker condition on the drift of the non-traded asset, we can still derive a differentiability results of the price process with respect to the initial condition of the non-traded asset. Such result can be linked to the delta hedging as we will see later. 
			
			\begin{thm}
				Under the assumption of this section, for $(t,x)\in [0,T]\times \mathbb{R}$ the indifference price function $(t,x)\mapsto p(t,x)$ is continuous in $t$ and differentiable in $x$.
				
				Moreover for any bounded open set $\mathcal{O}\subset \mathbb{R}$, the process $p(\cdot,X^{0,x}_\cdot)$ satisfies:
				$$
				x\mapsto p(s,X^{0,x}_s) \in L^2(\Omega, W^{1,p}(\mathcal{O})).
				$$
				In addition, we have the following representation. 
				\begin{align*}
					\frac{\partial}{\partial x}p(s,X^{0,x}_s) =&\big(v'(s,X^{0,x}_s)	-\hat v'(s,X^{0,x}_s)\big) \frac{\partial}{\partial x}X_s^{0,x}\\
					=&\big(v'(s,X^{0,x}_s)	-\hat v'(s,X^{0,x}_s)\big)\\
					&\times 	\exp\big\{
					2(\tilde b(X^{0,x}_s) - \tilde b(X^{0,x}_t)) -2\int_t^s  b^2(X^{0,x}_u)\mathrm{d}u-\int_t^s  b(X^{0,x}_u)\mathrm{d}W_u\big\},
				\end{align*}
				where $v'$ is the spatial partial derivative of $v$ solution to the PDE \eqref{eqmainP1} and  $\tilde b(x) = \int_0^x \tilde b(y)\mathrm{d}y$.
				Furthermore $p(t,X^{0,x}_t)$ is Malliavin differentiable.
			\end{thm}
			\begin{proof}
				The first part of the theorem follows from the fact that the function $v$ is differentiable in the space variable. The second part is deduced from Proposition \ref{th4.4}. The representation result comes from Proposition \ref{th4.4} and the It\^o's formula applied to $\tilde b$ as in Remark \ref{remlocdeco1}.
			\end{proof}
			Let us point out that the above smoothness of the indifference price $p$ is derived under much weaker assumptions than in \citealp{DosReis}.
			
			
			
			The following corollary provides an explicit representation of the derivative hedge. We denote the {\it conditional derivative hedge} by $\Delta(t,r)= \hat{\pi}^\ast_t-\pi_t^\ast$
			\begin{cor}
				Under the assumptions of this section, the derivative hedge satisfies
				\[ \Delta(t,x) = -\frac{\partial}{\partial x}p(t,x)  \].
			\end{cor}

			\subsection*{Conclusion and Possible extension }\leavevmode
			In this paper, we derive some conditions under which the solution of a system of coupled forward-backward SDEs with non-smooth drift coefficient and quadratic driver admits a density which is absolutely continuous with respect to the Lebesgue measure. Two major cases were considered: the H\"older continuous drift case (Subsection \ref{Hcase}) and the bounded and measurable drift one (Subsection \ref{section6}). In the latter, we only consider a Brownian motion with a drift coefficient which in fact depends on the couple solution $(Y,Z)$ of the backward SDE (see \eqref{540}). It is worth noting that, there is less hope to obtaining such densities' results for coupled FBSDEs with non constant diffusive coefficient while the drift is allowed to be discontinuous in space. Indeed, via the {\it decoupling field} (Section \ref{solvability}) the forward equation reduces to the classical SDE \eqref{eqfor1}, with only bounded and measurable drift and as far as we know, there is not yet a general result of existence or uniqueness for such equations with non constant volatility. In the former case, we consider a non trivial diffusive coefficient (see equation \eqref{eqmain36}). The It\^o-Tanaka trick was successively applied to transfer the parabolic regularization of the law of diffusion to the existence of an absolutely continuous density with respect to the Lebesgue measure of components  $X$ and $Y$ solution to  \eqref{eqmain36}, respectively. 
			
			Due to the quadratic behaviour of the driver, most of the results derived in this paper are restricted to the one dimensional setting. However, the result on the weak solvability of FBSDEs obtained in Section \ref{solvability} (see Theorem \ref{th3.3}) can be extended to the multidimensional setting if one considers that the driver $f$ is at most linear growth in $z$ (see Section 8 in \cite{Delarue1}). In this case, the solution $(X,Y,Z)$ will be taken in the space $\mathcal{S}^2(\mathbb{R}^d)\times\mathcal{S}^2(\mathbb{R}^q)\times \mathcal{H}^2(\mathbb{R}^{q\times d})$, where $q$ represents the dimension of the backward component $Y$. In this context, the strong solvability results of FBSDEs under the additional assumptions that the coefficients are H\"older continuous derived in Section \ref{section5} remains valid along with the result on the density of the forward component $X$. The backward component's case requires more attention, since the question relative to the strict positivity of solution to multidimensional quasi-linear PDE remains a non trivial problem to handle. However, we believe that, by combining the techniques developed in this paper and the interesting paper \cite{ChertoShama22}, one can successively handle this question. Many other interesting generalisation of this paper are conceivable, among them we point out a possible extension to the multidimensional quadratic setting of fully coupled FBSDEs by using the recent results from \cite{CHEN2023126948} and \cite{Jackson23}. These questions are under investigation for future work.
			
			\appendix
			\section{Integration with local time}\label{App1}
			Here, we provide some basic notions related to the integration with local time. We refer the reader to \cite{Ein2000}, \cite{Ein2006} and \cite{BMBPD17} {for more information}.
			The integration with respect to local time-space is defined for integrands that taking value in the space $(\mathcal{H}_x, \|\cdot\|_x)$ (see e.g. \cite{Ein2000}) {representing} the space of Borel measurable functions $\phi:[0,T]\times \mathbb{R}\rightarrow \mathbb{R} $ with the norm
			\begin{align*}
				\left\|\phi\right\|_x&:=2\Big(\int_0^1\int_{\mathbb{R}}\phi^2(s,z)\exp(-\frac{|z-x|^2}{2s})\frac{\diffns s \diff z}{\sqrt{2\pi s}}\Big)^{\frac{1}{2}}\notag
				+\int_0^1\int_{\mathbb{R}}|z-x| |\phi(s,x)|\exp(-\frac{|z-x|^2}{2s})\frac{\diffns s \diff z}{s\sqrt{2\pi s}}.
			\end{align*}
			In addition, any bounded function $\phi$ belongs to  $\mathcal{H}_x$ and the following representation holds for all $t\in [0,T]$ (see for instance  \cite[Lemma 2.11]{BMBPD17}) :
			\begin{align}\label{eqtransLT1}
				\int_0^t\partial_x\phi(s,X^x(s))\diffns s=-\int_0^t\int_{\mathbb{R}}\phi(s,z) L^{X^x}(\diffns s,\diffns z).
			\end{align}
			Furthermore for $\phi \in {\mathcal H}_0$, the subsequent decomposition is valid (see for example \cite[Theorem 2.1]{Ein2006}) 
			\begin{align}\label{eqslocalt1}
				\int_0^t\int_{\mathbb{R}}\phi(s,z) L^{W^x}(\diffns s,\diffns z) &= \int_0^t \phi (s,W^{x}_s)\diffns W_s	+\int_{T-t}^T \phi (T-s,\widehat{W}^{x}(s))\diffns B(s)\notag\\
				&\quad -\int_{T-t}^T \phi (T-s,\widehat{W}^x(s))\frac{\widehat{W}(s)}{T-s}\diffns s,\quad 0\leq t\leq T \text{ a.s.},
			\end{align}
			where $L^{W^x}(\diffns s,\diffns z)$ denotes integration with respect to the local time of the Brownian motion $W^x$ in both time and space, $W^x_{\cdot}:= x+ W_{\cdot}$ is the Brownian motion started at $x$ and 
			$\widehat{W}$ is the time-reversed Brownian motion defined by
			\begin{align}\label{eqstimrevbm1}
				\widehat{W}_t:=W_{T-t},\,\,0\leq t\leq T.
			\end{align}
			The process $B=\{B(t),\,\,\,0\leq t\leq T\}$ stands for an independent Brownian motion with respect to the filtration $\mathcal{F}_t^{\widehat{W}}$ generated by $\widehat{W}_t$, and satisfies:
			\begin{align}\label{eqstimrevbm2}
				B_t= \widehat{W}_t-W_T+\int_t^T\frac{\widehat{W}_s}{T-s}\diffns s.
			\end{align}
			
			We borrow the following result in  \cite[Lemma A.2]{BMBPD17}. It gives us the exponential bound of the local time-space integral of any bounded function.
			
			\begin{lemm}\label{lemmaexpoloc}
				Let $b:[0,T]\times \mathbb{R} \rightarrow \mathbb{R}$ be a bounded and measurable function. Then for $t\in [0,T],\, \lambda \in  \mathbb{R}$ and compact subset $K\subset \mathbb{R}$, we have
				$$
				\underset{x\in K}{\sup} E\Big[\exp\Big(\lambda \int_0^t\int_{\mathbb{R}}b(s,y)L^{W^x}(\diffns s,\diffns y)\Big) \Big]<C(\|b\|_{\infty}),
				$$
				where $C$ is an increasing function and $L^{W^x}(\diffns s,\diffns y)$ denotes integration with respect to the local time of the Brownian motion $W^x$ in both time and space.
			\end{lemm}
			
			\section{Comparison theorem for quadratic BSDE}
			Consider a BSDE of the form
			\begin{equation}\label{bsdeap}
				Y_t = \xi +\int_{t}^{T} f(s, Y_s, Z_s)\mathrm{d}s - \int_{t}^{T} Z_s\mathrm{d}W_s,
			\end{equation}
			satisfying the following set of assumptions
			\begin{assum}\label{assumap}
				\begin{itemize} There exist nonegative constant $\Lambda,K$ and a positive locally bounded function  $\ell:\mathbb{R}\mapsto \mathbb{R}_{+}$ such that $\mathbb{P}\text{-a.s.}$
					\item[(i)] $\xi$ is an $\mathfrak{F}_T$ measurable uniformly bounded random variable i.e.,$ \Vert \xi \Vert_{L^{\infty}} \leq \Lambda.$
					\item[(ii)] $f$ is an $\mathfrak{F}$ predictable continuous function and for all $(\omega,t,y,z)\in \Omega\times [0,T]\times\mathbb{R}\times\mathbb{R}^d,$
					\begin{align*}
						|f(\omega,t,y,z) |\leq \Lambda(1+ |y| + \ell(y)|z|^2) .
					\end{align*}
					\item[(iii)] For all $(\omega,t,y,z), (\omega,t,y',z')\in \Omega\times [0,T]\times\mathbb{R}\times\mathbb{R}^d,$ 
					\begin{align*} |f(\omega,t,y,z)-f(\omega,t,y',z')|&\leq K \Big(1+  \ell(|y-y'|^2)(|z|+|z'|)\Big)\Big(|y-y'| + |z-z'|\Big).
					\end{align*} 
				\end{itemize}
			\end{assum}
			Here is the main Theorem developed in this section. We mimic the proof of Theorem 3.12 in \cite{Frei}.
			\begin{thm}\label{thap}
				Let $(Y^{i}, Z^{i}) \in \mathcal{S}^{\infty}\times \mathcal{H}_{\text{BMO}}$ be the solution to BSDE\eqref{bsdeap}, with terminal value $\xi^i$ and generator $g^i$ for $i\in \{1,2 \}.$ Under Assumption\eqref{assumap} and that 
				\begin{align*}
					\xi^1 \leq \xi^2, \text{ and } f^1(t,Y_t^2,Z_t^2) \leq f^2(t,Y_t^2,Z_t^2)\quad \mathrm{d}\text{t}\otimes \mathrm{d}\mathbb{P}\text{-a.s.}.
				\end{align*} 
				Then for all $t \in [0,T]$ $Y^1_t \leq Y^2_t \quad \mathbb{P}\text{-a.s.}.$
				
				Moreover, if either $\xi^1 < \xi^2$ or $f^1(t,Y_t^2,Z_t^2) < f^2(t,Y_t^2,Z_t^2)$ in a set of positive $\mathrm{d}\text{t}\otimes \mathrm{d}\mathbb{P}\text{-measure}$ then, $Y_0^1 < Y_0^2.$ 
			\end{thm}
			\begin{proof}
				Let us first recall that $\mathcal{S}^{\infty}\text{-norms of  } Y^1,Y^2$ are bounded by $\Lambda$ (see e.g. \cite[Theorem 3.1]{Bahlali1}), and the $\mathcal{H}_{\text{BMO}}\text{-norms of  } Z^1,Z^2$ are bounded by a constant $C$ only depending on $\Lambda,$ and $T$ ( see  \cite[Corollary 3.2 ]{Bahlali1}). 
				
				Se
				\begin{align*}
					\begin{cases}
						\delta Y = Y^1 - Y^2,& \delta Z = Z^1 - Z^2\\
						\delta \xi = \xi^1 -\xi^2,& \delta f = f^1(\cdot,Y^2,Z^2) - f^2(\cdot,Y^2,Z^2),
					\end{cases}	
				\end{align*}
				and define the processes
				\begin{align}
					\Gamma_t:&= \frac{f^1(t,Y_t^1,Z_t^1) - f^1(t,Y_t^2,Z_t^1)}{Y^1_t - Y^2_t}{1}_{\{Y^1_t - Y^2_t\neq 0\}}, \quad e_t:= \exp\Big( \int_{0}^{t} \Gamma_s \mathrm{d}s \Big)\\
					\Pi_t:&= \frac{f^1(t,Y_t^2,Z_t^1) - f^1(t,Y_t^2,Z_t^2)}{|Z_t^1 - Z_t^2|^2}(Z^1_t - Z^2_t){1}_{\{|Z^1_t - Z^2_t|\neq 0\}}.
				\end{align}
				Observe that 
				$|\Pi| \leq \Lambda \left(1 + \ell(0)(|Z^1| + |Z^2|) \right)$, thus $\int_{0}^{\cdot}|\Pi_s|\mathrm{d}B_s$ is a BMO martingale, since $Z^1,Z^2 \in \mathcal{H}_{\text{BMO}}.$ Therefore, the probability measure $\tilde{\mathbb{P}}$ given by $\mathrm{d}\tilde{\mathbb{P}}/\mathrm{d}\mathbb{P} = \mathcal{E} (\int_{0}^{\cdot}\Pi_s\mathrm{d}W_s)$ is well defined (see \cite{Kazamaki}) and the process defined by $W_{\cdot}^{\tilde{\mathbb{P}}} = W_{\cdot} - \int_{0}^{\cdot}\Pi_s\mathrm{d}s$ is a $\tilde{\mathbb{P}}\text{-Brownian motion}.$ We deduce from \cite{Kazamaki}, the existence of $p^{*}>1$ such that 
				$\mathcal{E} (\int_{0}^{\cdot}\Pi_s\mathrm{d}W_s) \in L^{p^{*}}.$
				
				Since $Y^1, Y^2$ are bounded and $\ell$ is locally bounded, we have $|\Gamma_t| \leq \Lambda( 1+ 2 \ell(|Y_t^1 - Y_t^2|^2) |Z_t^1| )\leq \Lambda (1+ |Z_t^1|)$. Using the same arguments as before, we have $\Gamma \in \mathcal{H}_{\text{BMO}},$ thus $\Gamma \in \mathcal{H}^{2p}$ for every $p\geq 1.$
				
				The dynamics of $(\delta Y_t)$ is given by:
				\begin{align}
					\delta Y_t = \delta \xi + \int_{t}^{T} \delta f_s\mathrm{d}s + \int_{t}^{T}\left( f^1(s,Y_s^1,Z_s^1) - f^1(s,Y_s^2,Z_s^2)\right)\mathrm{d}s -\int_{t}^{T}\delta Z_s \mathrm{d}W_s.
				\end{align}
				The It\^{o}'s formula yields
				\begin{align}\label{A5}
					e_t\delta Y_t = e_t\delta \xi + \int_{t}^{T} e_s\delta f_s\mathrm{d}s -\int_{t}^{T}e_s\delta Z_s \mathrm{d}W_s^{\tilde{\mathbb{P}}},
				\end{align}
				where $\int_{0}^{\cdot}e_s\delta Z_s \mathrm{d}W_s^{\tilde{\mathbb{P}}}$ is a true $\tilde{\mathbb{P}}\text{-martingale}.$ Indeed,
				\begin{align*}
					\mathbb{E}^{\tilde{\mathbb{P}}}\int_{0}^{T}|e_s|^2|\delta Z_s|^2 \mathrm{d}s &\leq \mathbb{E}\Big[ \Big(\mathcal{E} (\int_{0}^{\cdot}\Pi_s\mathrm{d}W_s)\Big)^{p^{*}}\Big]^{\frac{1}{p^{*}}} \mathbb{E}\Big[|e_T|^{2q^{*}} \Big(\int_{0}^{T}|\delta Z_s|^2 \mathrm{d}s\Big){q^{*}}\Big]^{\frac{1}{p^{*}}},
				\end{align*}
				and this is finite thanks to the $\mathcal{H}_{\text{BMO}}$ property of $\Pi,\Gamma, Z^1$ and $Z^2$. Then, equation \eqref{A5} becomes 
				\begin{align}
					e_t\delta Y_t = \mathbb{E}^{\tilde{\mathbb{P}}}\Big[  e_t\delta \xi + \int_{t}^{T} e_s\delta g_s\mathrm{d}s |\mathfrak{F}_t \Big].
				\end{align} 
				Thanks to the Theorem's assumption we conclude that $e_t\delta Y_t \leq 0$ and hence for all $t\in [0,T]$ we have $ Y^1_t \leq Y^2_t\quad \tilde{\mathbb{P}}\text{-a.s.}$ and $\mathbb{P}\text{-a.s.}.$
				
				In particular, for $t= 0$ and if $\delta < 0$ or if $\delta g < 0$ in a set of positive $\mathrm{d}\text{t}\otimes \mathrm{d}\mathbb{P}\text{-measure},$ then we conclude that $Y_0^1 < Y_0^2,$ and this ends the proof.
			\end{proof}
			
			\section{Second order Malliavin differentiability}\label{apdix3}
			In this appendix, we provide some conditions under which a BSDE  admits a {second order} Malliavin differentiable solution. This result is taken from \cite{ImkDos}. Let us consider the following BSDEs ({which can be seen as the BSDE satisfied by the Malliavin derivative of the backward process })
			\begin{align}\label{second}
				U_{s,t} &= 0, \quad V_{s,t} = 0, \quad t \in [0,s),\notag\\
				U_{s,t} &= D_s\xi + \int_t^T G(r,s, U_{s,r}, V_{s,r})\mathrm{d}r - \int_t^T V_{s,r}\mathrm{d}W_r, \quad t\in [s,T]
			\end{align}
			where $h$ is a measurable function $\Omega \times[0,T]\times[0,T]\times\mathbb{R}\times\mathbb{R}^d\rightarrow \mathbb{R}$ given by
			\[  G(\cdot,t,s,y,z) = A_{s,t}(\cdot) + B_t(\cdot)y + \langle C_t(\cdot), z\rangle   \] and $\xi$ is an $\mathfrak{F}_T\text{-measurable}$ random variable, satisfying the following assumptions:
			\begin{itemize}
				\item[(H1)] $\xi$ is Malliavin differentiable, $\mathfrak{F}_T\text{-measurable}$ bounded random variable with Malliavin derivative given by $D_s\xi$ satisfying
				\[ \sup_{0\leq s \leq T} \| D_s\xi \|_{L^{2p}} < \infty, \text{ for all } p> 1.  \]
				\item[(H2)] For all $s\in [0,T]$, $D_s\xi$ is Malliavin differentiable, with second order derivative given by $(D_{s'}D_s\xi)_{s',s \in [0,T]}$ and satisfying $\sup_{0\leq s',s \leq T} \|D_{s'} D_s\xi \|_{L^p} < \infty$ for all $p >1.$
				\item[(H3)] $B: \Omega\times[0,T] \rightarrow \mathbb{R}$ and $C: \Omega\times[0,T] \rightarrow \mathbb{R}^d$ are $\mathfrak{F}_t\text{-adapted}$ processes bounded by a constant $\Lambda > 0$. $A: \Omega\times[0,T] \times [0,T] \rightarrow \mathbb{R}$ satisfies $\sup_{0\leq s \leq T}\|A_{s,\cdot}\|_{\mathcal{S}^{2p}}< \infty$ for all $p >1$. For each fixed $s \in [0,T]$, the process $(A_{s,t})_{t\in [0,T]}$ is progressively measurable.
				\item[(H4)] $A_{s,t}, B_t, C_t$ are Malliavin differentiable for all $s\in [0,T]$ and $t\in [0,T]$. Measurable versions of their Malliavin derivatives are respectively given by $D_{s'}A_{s,t}, D_sB_t$ and $D_sC_t$ for $s',s\in [0,T]$ such that for all $p> 1$
				\[ \sup_{0\leq s',s\leq T} \mathbb{E} \Big[\Big(\int_0^T(|D_{s'}A_{s,t}|^2 + |D_sB_t|^2 + |D_sC_t|^2)\mathrm{d}t\Big)^p\Big]  < \infty \]
			\end{itemize}
			\begin{thm}\label{thapendix}
				Under Assumptions (H1)--(H4), the BSDE \eqref{second} has a unique Malliavin differentiable solution $(U,V)\in \mathcal{S}^{2p}\times\mathcal{H}^{2p}$ for $p> 1$. A version of $(D_{s'}U_{s,t},D_{s'}V_{s,t})_{ 0\leq s'\leq s\leq t\leq T}$ satisfies
				\begin{align}
					D_{s'}U_{s,t} &= D_{s'}D_s\xi -\int_t^T D_{s'}V_{s,r}\mathrm{d}W_r \notag\\
					&+\int_t^T \Big((D_{s'}G)(r,s,U_{s,r},V_{s,r}) + \langle \nabla G(r,s,U_{s,r},V_{s,r}) ,(D_{s'}U_{s,r},D_{s'}V_{s,r})\rangle \Big)\mathrm{d}r.
				\end{align}
				Moreover, a version of $(V_{s,t})_{0\leq s\leq t\leq T}$ is given by $(D_tU_{s,t})_{0\leq s\leq t\leq T}$.
			\end{thm}
			
			{\bf Acknowledgments:}{ The authors thank the Editor in Chief and two anonymous referees for their remarks which helped to considerably improve the presentation of the paper. A part of this work was carried out while Rhoss Likibi Pellat visited the Department of Mathematics, of the Brandenburg University of Technology Cottbus-Senftenberg, whose hospitality is greatly appreciated and he personally thanks Prof Ralf W\"underlich.}

			\bibliographystyle{plain}
			\bibliography{Biblio, Biblio1}  
		\end{document}